\documentclass[12pt,reqno]{amsart}
\usepackage[foot]{amsaddr}
\usepackage{amsmath}
\usepackage{amssymb}
\usepackage{verbatim}
\usepackage{bbm}
\usepackage{thmtools}
\usepackage{graphicx}
\usepackage[font=footnotesize]{caption}
\usepackage{subcaption}
\usepackage{epstopdf}
\usepackage{multirow}
\usepackage{url}
\usepackage{makecell}
\usepackage{tabularx}
\usepackage{longtable}
\usepackage{graphicx}
\usepackage[percent]{overpic}
\usepackage{tikz}
\usepackage{pgfplots}
\usepackage{xcolor}
\usepackage{pgfplots}
\usepackage{pgfplotstable}
\usepackage{wavelet}
\allowdisplaybreaks
\usepackage[toc]{appendix}
\usepackage[capitalise]{cleveref}
\usepackage{color}
\usepackage{cases}
\allowdisplaybreaks

\usepackage[top=0.5in,bottom=0.47in,left=0.7in,right=0.7in]{geometry}

\newtheorem{theorem}{Theorem}[section]
\newtheorem{lemma}[theorem]{Lemma}

\newtheorem{remark}[theorem]{Remark}

\newtheorem{example}{Example}

\newcommand{\nv}{\vec{n}} 

\newcommand{\ia}{\textsf{i}}

\def\XXint#1#2#3{{\setbox0=\hbox{$#1{#2#3}{\int}$ }
		\vcenter{\hbox{$#2#3$ }}\kern-.6\wd0}}

\renewcommand{\ge}{\geqslant}
\renewcommand{\le}{\leqslant}

\renewcommand{\lrs}[3]{(l_{#1}(\mathbb{Z}))^{#2\times #3}} 
\newcommand{\LpI}[1]{L^{#1}(\mathcal{I})}

\newcommand{\LpO}[1]{L^{#1}(\Omega)}
\newcommand{\blf}{B} 
\newcommand{\mm}{m} 
\newcommand{\SSupp}{\text{Supp}}

\numberwithin{equation}{section}
\numberwithin{example}{section}

\begin{document}

\title[Galerkin Scheme Using Wavelets for Elliptic Interface Problems]{Galerkin Scheme Using Biorthogonal Wavelets on Intervals for Elliptic Interface Problems}

	\author{Bin Han}
	\address{Department of Mathematical and Statistical Sciences, University of Alberta, Edmonton, Alberta, Canada T6G 2N8.}
	 \email{bhan@ualberta.ca, mmichell@ualberta.ca}
	
	\author{Michelle Michelle}
	
	\thanks{Research supported in part by
		Natural Sciences and Engineering Research Council (NSERC) of Canada under grants RGPIN-2024-04991, RGPIN-2026-05814, and DGECR-2026-00371, the University of Alberta startup grant, and the Digital Research Alliance of Canada.}


	\makeatletter \@addtoreset{equation}{section} \makeatother

\begin{abstract}
This paper presents a wavelet Galerkin method for solving elliptic interface problems of the form $-\nabla\cdot(a\nabla u)=f$ in $\Omega\backslash \Gamma$, where $\Gamma$ is a smooth interface within $\Omega$. Since the scalar variable coefficient $a>0$ and source term $f$ are often discontinuous across $\Gamma$, the solution $u$ typically has discontinuous gradient $\nabla u$ across $\Gamma$ and hence $u\not\in H^{1.5}(\Omega)$, posing significant challenges for traditional numerical methods. By utilizing a compactly supported biorthogonal wavelet for $H^1_0(\Omega)$, we develop a strategy that incorporates additional wavelet elements (or basis functions) along the interface to resolve the complex geometry of the interface $\Gamma$ and the resulting gradient discontinuities. For the two-dimensional (2D) elliptic interface problem, the proposed method achieves near-optimal convergence rates: $\bo(h |\log(h)|)$ in the $H^1(\Omega)$-norm and $\bo(h^2 |\log(h)|^2)$ in the $\LpO{2}$-norm with respect to the approximation order. A key theoretical contribution is the use of the dual biorthogonal wavelet basis to establish the $H^1(\Omega)$ convergence results. This is supported by the development of weighted Bessel properties for wavelets and several inequalities in fractional Sobolev spaces. To maintain high accuracy and robustness against high-contrast coefficients, our method leverages an augmented set of wavelet elements, similar to meshfree approaches, thereby eliminating the need for the complex re-meshing required by finite element methods. Unlike existing techniques, this wavelet Riesz basis framework captures the geometry of $\Gamma$ seamlessly while ensuring that the condition numbers of the coefficient matrices remain small and uniformly bounded, independent of the problem size.
\end{abstract}

	\maketitle
	
	\pagenumbering{arabic}
	
	\section{Introduction and Motivations}\label{sec:intro}
In this paper, we
are interested in solving elliptic interface problems.
Such problems have been seen in many real-world applications of science and engineering such as
modeling of underground waste disposal,
mechanics of composite materials, oil reservoir simulations and other fluid flows through porous media.
Let $\Omega\subset \R^d$ be a bounded domain. Consider a smooth $(d-1)$-dimensional interface $\Gamma$ inside $\Omega$ such that it splits the whole domain $\Omega$ into two subdomains $\Omega_+$ and $\Omega_-$.
Throughout the paper, for a function $v$ in $\Omega$, we define $v_+:=v\chi_{\Omega_+}$, $v_-:=v\chi_{\Omega_-}$, and
\[
\lb v \rb(x)
:= \lim_{h \rightarrow 0^{+}} v_+(x+h \nv) - \lim_{h \rightarrow 0^{-}} v_-(x-h \nv), \quad x\in \Gamma,
\]
where $\nv$ is the unit normal vector of $\Gamma$ pointing into the subregion $\Omega_+$. The $d$-dimensional (dD) elliptic interface problem on the domain $\Omega$ considers
	%
		\begin{subnumcases} {\label{model0}}
		- \nabla \cdot (a \nabla u) = f
& \text{in \quad $\Omega \setminus \Gamma$}, \label{pde}\\
		\lb u\rb=g \quad & \text{on \quad $\Gamma$}, \label{jump1}\\
\lb a \nabla u \cdot \nv\rb  = g_\Gamma & \text{on \quad $\Gamma$}, \label{jump2}\\
		u = g_b & \text{on \quad $\partial \Omega$}, \label{boundary}
		\end{subnumcases}
	%
where the variable diffusion coefficient $a \in L^{\infty}(\Omega)$ satisfies $\text{ess-inf}_{x\in \Omega}(a(x)) > 0$, the source term $f\in L^{2}(\Omega)$,  the jump functions $g \in H^{1/2}(\Gamma)$ and $g_{\Gamma} \in H^{-1/2}(\Gamma)$,
	and the boundary condition $g_b\in H^{1/2}(\partial \Omega)$. Note that $\lb u\rb=g$ is the first jump condition \eqref{jump1} for possible discontinuity of $u$ across $\Gamma$, while $\lb a \nabla u \cdot \nv\rb=g_\Gamma$ is the second jump condition \eqref{jump2} for possible discontinuity of the flux across $\Gamma$. Hence, the solution $u$ often has singularities across $\Gamma$ and hence $u$ has low regularity.

Following the standard approach in finite element methods (FEMs), one often assumes $g=0$ in \eqref{jump1} and $g_b=0$ in \eqref{boundary}, which can be achieved by using auxiliary functions, see \cref{sec:exp} for details.
One can observe that the model problem \eqref{model0} for the case $g=0$ on $\Gamma$ is equivalent to
\begin{equation}\label{model}
\begin{cases}
- \nabla \cdot (a \nabla u) = f-g_\Gamma \delta_\Gamma
& \text{in \quad $\Omega$},\\
u = g_b & \text{on \quad $\partial \Omega$},
\end{cases}
\end{equation}
where $g_\Gamma \delta_\Gamma$ is the Dirac function along the interface $\Gamma$ with weight $g_\Gamma$.
Then the weak formulation of the model problem \eqref{model} with $g_b=0$ seeks $u \in H^{1}_0(\Omega)$ such that
	\be \label{weak}
	\blf(u,v):=\langle a \nabla u, \nabla v\rangle_{\Omega} = \langle f, v\rangle_{\Omega} - \langle g_{\Gamma}, v \rangle_{\Gamma}, \quad \forall  v \in H^{1}_0(\Omega),
	\ee
where the Sobolev space
$H^{1}_0(\Omega):=\{u\in H^{1}(\Omega)\; :\; u=0 \text{ on } \partial \Omega\}$.

The existence and uniqueness of a weak solution $u\in H^1(\Omega)$ to the model problem \eqref{model}, under the assumption that $u_+\in H^2(\Omega_+)$ and $u_-\in H^2(\Omega_-)$, have been extensively addressed in \cite[Sections 16 and 17 of Chapter~3]{LUbook}.
In most literature on elliptic interface problems, rectangular domains $\Omega$ are considered, the 2D setting is more common than the 3D setting, and both the diffusion coefficient $a$ and the source term $f$ are assumed to be smooth in the subdomains $\Omega_+$ and $\Omega_-$. However, both $a$ and $f$ can be discontinuous across the interface $\Gamma$. As a result, even with $g_{\Gamma}=0$ in \eqref{model}, due to the discontinuity of the diffusion coefficient $a$ across $\Gamma$, the solution $u$ to the problem \eqref{model} typically has a discontinuous gradient across the interface, which results in $u\notin H^{1.5}(\Omega)$, even though $u_+ \in H^{2}(\Omega_+)$ and $u_- \in H^{2}(\Omega_-)$. Specifically, the gradient $\nabla u$ must be discontinuous across the interface $\Gamma$ if the diffusion coefficient $a$ is discontinuous across $\Gamma$, the source term $f$ is discontinuous across $\Gamma$ (even if $a=1$), or $g_{\Gamma}$ is not identically zero.
Consequently, the solution $u$ to \eqref{model} has singularities near the interface $\Gamma$ and $u\not \in H^{1.5}(\Omega)$,
leading to deteriorated convergence rates of at most $0.5$ in the $H^1(\Omega)$-norm for classical numerical methods.

\subsection{Wavelet method and the central question}

Wavelets have been used for solving various PDEs (e.g., see \cite{Cer19, cohbook, dah97, cdd01, dku99, HM19, HM21a, HM23, HM26} and references therein). To the best of our knowledge, no studies on wavelet methods specifically address elliptic interface problems in \eqref{model} and develop a framework tailored to them for handling the low regularity of their solutions.
Due to the ability of wavelets to capture singularities, we regard wavelets as an attractive tool for solving such problems. 

In a traditional wavelet method for solving PDEs, one fixes the scale level, which then determines the number of basis functions through scaling and shifting of wavelets in the approximation. In fact, these wavelet basis functions span the same space as FEM basis functions. Hence, traditional wavelet methods are almost identical to FEMs, except that they use wavelet bases and possess uniformly bounded condition numbers. Due to its strong resemblance with FEMs, standard wavelet methods also suffer from reduced convergence rates when applied to the elliptic interface problem \eqref{model}.

A celebrated work in wavelets is the adaptive wavelet method for numerically solving general PDEs, introduced in \cite{cdd01}, which we now discuss in the context of \eqref{model}. Let $n$ be adaptively selected wavelet elements or basis functions. According to \cite[Section~3.1 and Proposition~3.2]{cdd01}, given that the solution $u$ is in the Besov space $B^{sd+1}_\tau (L^\tau(\Omega))$ with $\frac{1}{\tau}=s+\frac{1}{2}$, the adaptive wavelet method converges at the rate of $\bo(n^{-s})$ in the $H^1(\Omega)$-norm. That is, if $n=\bo(h^{-d})$ with $h$ being the mesh size, the convergence rate is $\bo(h^{sd})$. Due to the piecewise regularity of the solution $u$ of \eqref{model} and the discontinuity of $\nabla u$ across the interface, we have $u \in H^{1.5-\delta}(\Omega)$ for every $\delta>0$. To obtain the Besov regularity using the global Sobolev regularity of $u$, we need $sd+1 < 1.5 - \delta$. The embedding result gives $H^{3/2-\delta}(\Omega) = B^{3/2-\delta}_2(L^2(\Omega)) \hookrightarrow B^{1+sd}_\tau (L^\tau(\Omega))$, where $\tfrac{1}{\tau} = s+\tfrac{1}{2}$. This means that the adaptive wavelet method in \cite{cdd01} has a theoretically guaranteed convergence rate of $\bo(n^{-1/(2d)+\mu})$ for any fixed constant $0 < \mu < \tfrac{1}{2d}$ in the $H^1(\Omega)$-norm; equivalently, since $n=\bo(h^{-d})$, this rate is $\bo(h^{0.5-\gep})$ in the $H^1(\Omega)$-norm for any fixed constant $\gep:=d\mu>0$.

This naturally raises the question of
\textbf{\emph{whether we can design a biorthogonal wavelet method such that the convergence rate is higher than $0.5$ in the $H^1(\Omega)$-norm for the elliptic interface problem \eqref{model}.}}
We already answered this question in \cite{HM26} for the elliptic interface problem \eqref{model} with $d=1$. For convenience, we give a summary of the 1D result here. Suppose $u_{\pm}\in H^{m}(\Omega_{\pm})$ for an integer $m\ge 2$ and let us take any compactly supported biorthogonal wavelet on $\Omega$ having the approximation order $m$. Next, we suitably form wavelet elements $\mathcal{B}^{S,H^1_0(\Omega)}_{J_0,J}$ in \cite[(2.4)]{HM26} for all $J\ge J_0$ with cardinality $\bo(2^J)$.
Then, \cite[Theorem~2.1]{HM26} theoretically proves that the numerical solution $u_h$ satisfies
\begin{equation}\label{1D}
\|u-u_h\|_{H^1(\Omega)}=\bo(h^{m-1}),\qquad
\|u-u_h\|_{L^2(\Omega)}=\bo(h^m)
\end{equation}
with $h:=2^{-J}$ for all $J\ge J_0$.
For $d=2$, under the widely adopted assumption
$u_{\pm}\in H^2(\Omega_{\pm})$, this paper theoretically proves and numerically verifies that a suitably designed biorthogonal wavelet method can indeed achieve a near-optimal convergence rate $\bo(h|\log(h)|)$ in the $H^1(\Omega)$-norm and $\bo(h^2|\log(h)|^2)$ in the $L^2(\Omega)$-norm.
The logarithm factor comes from the fact that $u\in H^{1.5-\gep}(\Omega)$ for all $\gep>0$ but $u\not \in H^{1.5}(\Omega)$. Consequently, we obtain the order $\bo(h|\log(h)|)$ by exploiting the limit $\gep \to 0^+$ together with several inequalities on fractional Sobolev spaces in our proof.

\subsection{Related work on FDMs and FEMs} \label{subsec:FEMs}

We now review the literature on finite difference methods (FDMs) and finite element methods (FEMs) for the elliptic interface problems.
Various FDMs for solving the model problem \eqref{model0} have been studied in the literature \cite{FZ20, FHM23, FHM22, FHM24, LL94, LI06, ZZFW06} and many references therein. One way is to use the immersed interface method introduced in \cite{LL94}, whose later developments were discussed in \cite{LI06} and other papers. Another way is to use the matched interface and boundary method \cite{FZ20, ZZFW06} and further developments. More recently, a sixth-order hybrid FDM for the elliptic interface problem on a rectangular domain with mixed boundary conditions was developed in \cite{FHM24}. Under the assumption $u_{\pm}\in C^m(\overline{\Omega_{\pm}})$ with sufficiently large $m$ for FDMs, a key idea common to these studies is that they modify the finite difference stencils according to information near the interface. But the convergence proofs of these FDM methods are often not known in the literature, due to very complicated expressions and structures of modified stencils near the interface.

As mentioned earlier, wavelet methods are closely related to FEMs. Under the widely adopted assumption $u_{\pm}\in H^2(\Omega_{\pm})$ for FEMs, much of the literature on FEMs aims to develop second-order schemes to restore the optimal first-order convergence in the $H^1(\Omega)$-norm by modifying FEM elements near the interface.
One way is to use the body-fitted FEM with its mesh generated depending on the shape of the interface and the boundary of the domain \cite{CWW17,CZ98}. This can be challenging especially when the interface has a complicated geometry, and more so for time-dependent problems \cite{ABGL23, ZCBB22}. There is also a large class of FEMs that do not necessitate a mesh generation that conforms to the interface, which is called unfitted FEMs. Some methods that fall into this category are the immersed FEM (IFEM) \cite{ABGL23, GLL08, GL19, JWCL22, LLW03, LLZ15}, the CutFEM \cite{BCHLM15, HH02}, the extended FEM (XFEM) \cite{BBK17,KBB16,VSC06,ZB20,ZCBB22}, and the unfitted high-order $hp$ method \cite{bcde21, BE18,CL23,CLX21,M12}. There are far fewer studies on high-order methods, and they typically require a stronger smoothness assumption on each subdomain. After fixing a mesh that is independent of the interface, the IFEM proceeds by modifying shape functions of interface elements \cite{GL19}. As a recent development in the IFEM, a high-order method that addresses nonhomogeneous (first and second) jump conditions and achieves optimal convergence was studied in \cite{ABGL23}. Instead of modifying the shape functions near the interface, one can still choose to use the standard FEM shape functions, but employ the Nitsche's penalty along the interface \cite{GL19}. This is a key idea of the CutFEM, which was first studied by \cite{HH02} and then reviewed in \cite{BCHLM15}. A related method using penalties is the discontinuous Galerkin method for elliptic interface problems \cite{BE18,CYZ11,M12}.
The XFEM incorporates special basis functions near the interface in the approximation space to recover the optimal convergence rate \cite{BBK17,KBB16,ZB20,ZCBB22}. The shape functions in the XFEM are all continuous, and thus the method is conforming. Also,
unlike other FEMs that use discontinuous/nonconforming elements near the interface $\Gamma$, no penalties are used by the XFEM. The downside is that it may lead to ill-conditioning of the linear system. However, there are further studies that deal with the stabilization for such a method so that the conditioning behaves like the standard FEM \cite{BBK17,KBB16,ZB20}. Some studies that assume variable piecewise coefficients $a$ are \cite{BBK17,CWW17,GLL08,JWCL22,KBB16,M12}, whereas the other previously mentioned studies assume piecewise constant coefficients $a$.

	\subsection{Main contributions of this paper}
	We present a new wavelet Galerkin method based on the tensor product of compactly supported biorthogonal wavelets on intervals for the elliptic interface problem \eqref{model}. For $d=2$, our method achieves convergence rates  $\bo(h |\log(h)|)$ in the $H^1(\Omega)$-norm and $\bo(h^2 |\log(h)|^2)$ in the $\LpO{2}$-norm, as stated in \cref{thm:converg}. These rates get arbitrarily close to
first-order in the $H^1(\Omega)$-norm and
second-order in the $\LpO{2}$-norm as the scale level increases. Our proof extensively uses the dual part of the biorthogonal wavelets, relies on the weighted Bessel property and results on wavelets in fractional Sobolev spaces, and employs FEM arguments. Standard techniques in the literature are far from sufficient to establish the near-optimal convergence rates of our method, which has to handle the complex interface geometry and capture singularities of the exact solution $u$ with low regularity near the interface $\Gamma$. Our approach contrasts sharply with FEMs and traditional wavelet methods which critically rely on the polynomial approximation and Bramble-Hilbert lemma.
	
	At the same time, this paper complements and contrasts with our earlier work \cite{HM26}, where we analyzed the same scheme applied to the 1D elliptic interface problem. In that case, the method can achieve arbitrarily high convergence rates, largely because the interface $\Gamma$ is a single point and because of special results for Sobolev spaces. Moreover, the $\log(h)$ factors for the 2D case in this paper do not appear in \cite[Theorem~2.1]{HM26} (see \eqref{1D}), largely because we used the Sobolev embedding result $\|\nabla u\|_{L^\infty(\Omega_{\pm})}\le C\|u\|_{H^2(\Omega_{\pm})}$, which only holds for $d=1$. 
Hence, the convergence proof for the 2D case here is much more challenging than the 1D case.

	We shall  discuss in \Cref{subsec:method}
how the analysis extends to higher dimensions $d > 2$, even though the main theorem and numerical results in this paper concern the 2D elliptic interface problem in \eqref{model}. For $d\ge 2$, we informally argue in \Cref{subsec:method} that our method has a convergence rate of $\frac{d}{2(d-1)}-\gep$ with any $\gep>0$ in the $H^{1}(\Omega)$-norm. Though the convergence rate $\frac{d}{2(d-1)}-\gep$ is always greater than $1/2$ for sufficiently small $\gep$, it is strictly less than $1$ for $d>2$, for example, $\frac{d}{2(d-1)}=0.75$ for $d=3$.
Hence, we do not address the full theoretical and numerical details for $d>2$ in this paper. Important issues such as how to modify our method to regain a first-order convergence rate in the $H^{1}(\Omega)$-norm for $d\ge 3$ and how to develop high-order versions of our method are left for future work.

	To achieve a convergence rate higher than $0.5$ in the $H^{1}(\Omega)$-norm for the 2D elliptic interface problem, our method adds extra wavelet basis functions/elements from higher scale levels whose supports intersect the interface to the approximation. The resulting method has the following advantages:
	\begin{itemize}
		\item \textit{Simplicity.} Complex interface geometries (including those with high curvature), high-contrast variable coefficients $a$, and singularities along the interface are simultaneously handled by methodically incorporating wavelet basis functions from higher scale levels. As long as they touch the interface $\Gamma$, they are included in the approximation.
We do not modify wavelet elements near the interface, contrary to existing approaches based on FEMs or FDMs.
		\item \textit{Meshfree, uniformly bounded condition numbers, and robustness under numerical perturbations.} Our method preserves these known core wavelet advantages \cite{cdd01,DK92}. Being meshfree, it eliminates the need for domain re-meshing; accuracy is instead improved by increasing the wavelet scale level. Furthermore, because the biorthogonal wavelets form a Riesz basis for \(H^1_0(\Omega)\), condition numbers remain uniformly bounded, specifically by \(C_w \|a\|_{L^\infty}\|a^{-1}\|_{L^\infty}\) in \cref{thm:converg}. Meanwhile, the condition numbers for FEMs grow like \(h^{-2}\). Since the smallest singular values of wavelet coefficient matrices are bounded away from zero, iterative solvers tend to converge more rapidly. The Riesz basis property also allows us to show that the convergence rates remain valid when the interface curve approximation and quadrature errors are at most of the same order as the discretization errors. 

		\item \textit{Conforming.} All wavelet elements are continuous and belong to the Sobolev space $H^1_0(\Omega)$, same as the XFEM \cite{BBK17,KBB16,ZB20,ZCBB22}. The conforming nature of our method also differs from many studies cited in \cref{subsec:FEMs}, where elements near the interface may be discontinuous. We do not use the distance function $\mbox{dist}(x):=\inf_{y\in \Gamma} \|x-y\|$ for $x\in \Omega$ as in the XFEM to construct modified finite elements near $\Gamma$, which can be computationally expensive.
		\item \textit{Fixed constant-factor increase in the number of degrees of freedom.} The added wavelet elements have scale levels that are at most double the maximum scale level of the other regular basis functions. Hence, the number of terms used in the approximation with these extra functions is only a fixed constant multiple of the number of terms without them. This constant only depends on the interface shape and the support of the wavelets.
		\item \textit{Efficient implementation.} Thanks to the refinability properties of the wavelet basis and the fast wavelet transform, the implementation can be carried out efficiently.
	\end{itemize}

	\subsection{Organization of this paper}
	In \cref{sec:method}, we revisit some basic concepts and definitions of wavelets, present our method for solving \eqref{model}, and state our main result, \cref{thm:converg}, on convergence rates and uniform boundedness of condition numbers. 
	We also discuss how the analysis extends to higher dimensions. In \cref{sec:exp}, we discuss the setup of the numerical experiments, describe how we handle nonhomogeneous first jump conditions and/or Dirichlet boundary conditions, and present the numerical results of our method applied to \eqref{model0}-\eqref{model}. Finally, we provide the proofs in \cref{sec:proof}.

	\section{Wavelet Galerkin Method for the Model Problem \eqref{model}}	 \label{sec:method}

In this section, we describe our Galerkin scheme using the tensor product of biorthogonal wavelets on the unit interval $(0,1)$ for solving the $d$D elliptic interface problem in \eqref{model}. As usual in FEMs or traditional wavelet methods, the implementation of our Galerkin scheme solely relies on the primal part of the biorthogonal wavelets. However, the dual part of the biorthogonal wavelets will play a key role for proving theoretical convergence of our proposed method.


	\subsection{Preliminaries on wavelet bases in $\Lp{2}$ and $\LpI{2}$ with $\mathcal{I}:=(0,1)$} \label{sec:prelim}

Let us review some basic concepts of wavelets, which follow a similar presentation as in \cite{HM23}. Let $\phi:=\{\phi^{1}, \dots, \phi^{r}\}^{\mathsf{T}}$ and $\psi:=\{\psi^{1}, \dots, \psi^{s}\}^{\mathsf{T}}$ be square integrable functions in $\Lp{2}$. Define a wavelet affine system by
		 \begin{equation}\label{Bphipsi}
			 \AS_{J_0}(\phi;\psi):=  \left\{\phi^{\ell}_{J_0;k}: k \in \mathbb{Z},  \ell =1,\ldots,r \right\}
			\cup \left\{\psi^{\ell}_{j;k}: j \ge J_0, k \in \mathbb{Z}, \ell=1,\ldots, s \right\},
		\end{equation}
		where $J_0\in \Z$, $\phi^{\ell}_{J_0;k}:=2^{J_0/2} \phi^{\ell}(2^{J_0}\cdot-k)$, and $\psi^{\ell}_{j;k}:=2^{j/2} \psi^{\ell}(2^{j}\cdot-k)$.
		
		In the literature, when $r,s>1$, such a system is also referred to as a multiwavelet affine system. On the other hand, if $r=s=1$, then such a system is often referred to as a scalar wavelet affine system. Throughout this paper, we refer to multiwavelets and scalar wavelets simply as wavelets, since our analysis and discussion apply to both cases.
		
		We say that
		$\AS_{J_0}(\phi;\psi)$ is \emph{a Riesz basis for $\Lp{2}$} if
		(1) the linear span of $\AS_{J_0}(\phi;\psi)$ is dense in $\Lp{2}$, and (2) there exist positive constants $C_{1}, C_{2} >0$ such that
		 \begin{equation}\label{Riesz:L2}
			C_1 \sum_{\eta \in \AS_{J_0}(\phi;\psi)} |c_{\eta}|^2 \le \Big\| \sum_{\eta \in \AS_{J_0}(\phi;\psi)} c_{\eta} \eta \Big\|^2_{L^2(\R)} \le C_2 \sum_{\eta \in \AS_{J_0}(\phi;\psi)} |c_{\eta}|^2
		\end{equation}
		for all finitely supported sequences $\{c_\eta\}_{\eta \in \AS_{J_0}(\phi;\psi)}$. The relation in \eqref{Riesz:L2} holds for some $J_0\in \Z$ if and only if it holds for all $J_0\in \Z$ with identical positive constants $C_1$ and $C_2$ (see for example \cite[Theorem~6]{han12}). As a result, we simply refer to $\{\phi;\psi\}$ as \emph{a Riesz wavelet in $\Lp{2}$} if $\AS_{0}(\phi;\psi)$ is a Riesz basis for $\Lp{2}$.
		Further, let $\tilde{\phi}:=\{\tilde{\phi}^{1}, \dots, \tilde{\phi}^{r}\}^{\mathsf{T}}$ and $\tilde{\psi}:=\{\tilde{\psi}^{1}, \dots, \tilde{\psi}^{s}\}^{\mathsf{T}}$ be vectors of square integrable functions in $\Lp{2}$. We say that $(\{\tilde{\phi};\tilde{\psi}\},\{\phi;\psi\})$ is \emph{a biorthogonal wavelet} in $\Lp{2}$ if $(\AS_{0}(\tilde{\phi};\tilde{\psi}), \AS_{0}(\phi;\psi))$ is a biorthogonal basis in $\Lp{2}$, i.e.,
		(1) $\AS_{0}(\tilde{\phi};\tilde{\psi})$ and $\AS_{0}(\phi;\psi)$ are Riesz bases in $\Lp{2}$, and (2) $\AS_{0}(\tilde{\phi};\tilde{\psi})$ and $\AS_{0}(\phi;\psi)$ are biorthogonal to each other in $\Lp{2}$.
		
		The wavelet function $\psi$ has \emph{$m$ vanishing moments} if $\int_\R x^j \psi(x)dx=0$ for all $j=0,\ldots,m-1$. By convention, we define $\vmo(\psi):=m$ with $m$ being the largest of such an integer.
		
		The Fourier transform is defined by $\wh{f}(\xi):=\int_\R f(x) e^{-\ia x\xi} d x, \xi\in \R$ for $f\in \Lp{1}$ and is naturally extended to square integrable functions in $\Lp{2}$. Meanwhile, the Fourier series of $u=\{u(k)\}_{k\in \Z}\in \lrs{0}{r}{s}$ is defined by $\wh{u}(\xi):=\sum_{k\in \Z} u(k) e^{- \ia k\xi}$ for $\xi\in \R$, which is an $r\times s$ matrix of $2\pi$-periodic trigonometric polynomials. By $\td$ we denote the sequence such that $\td(0)=1$ and $\td(k)=0$ if $k\ne 0$.
		
		Now, we are ready to recall a key result of biorthogonal wavelets in $\Lp{2}$.
		
		\begin{theorem} \label{thm:bw} (\cite[Theorem~6.4.6]{hanbook} and \cite[Theorem~7]{han12})
			Let $\phi,\tilde{\phi}$ be $r\times 1$ vectors of compactly supported distributions on $\R$ and $\psi,\tilde{\psi}$ be $s\times 1$ vectors of compactly supported distributions on $\R$. Then $(\{\tilde{\phi};\tilde{\psi}\},\{\phi;\psi\})$ is a biorthogonal wavelet in $\Lp{2}$ if and only if the following are satisfied
			\begin{enumerate}
				\item[(1)] $\phi,\tilde{\phi}\in (\Lp{2})^r$ and $\ol{\wh{\phi}(0)}^\tp \wh{\tilde{\phi}}(0)=1$.
				\item[(2)] $\phi$ and $\tilde{\phi}$ are biorthogonal to each other: $\la \tilde{\phi},\phi(\cdot-k)\ra:=\int_{\R} \tilde{\phi}(x) \overline{\phi(x-k)}^\tp dx
= \td(k) I_r$ $\forall k\in \Z$.
				\item[(3)] There exist low-pass filters $a,\tilde{a}\in \lrs{0}{r}{r}$ and high-pass filters
				$b,\tilde{b}\in \lrs{0}{s}{r}$ such that
				\begin{align}
					 &\phi=2\sum_{k\in \Z} a(k)\phi(2\cdot-k),\qquad
					 \psi=2\sum_{k\in \Z} b(k)\phi(2\cdot-k), \label{phipsi}
					\\
					 &\tilde{\phi}=2\sum_{k\in \Z}\tilde{a}(k)
					 \tilde{\phi}(2\cdot-k),\qquad
					 \tilde{\psi}=2\sum_{k\in \Z} \tilde{b}(k)
					 \tilde{\phi}(2\cdot-k), \label{tphitpsi}
				\end{align}
				and $(\{\tilde{a};\tilde{b}\},\{a;b\})$ is a biorthogonal wavelet filter bank, i.e., $s=r$
				and
				\[
				\left [ \begin{matrix}
					 \wh{\tilde{a}}(\xi) &\wh{\tilde{a}}(\xi+\pi)\\
					 \wh{\tilde{b}}(\xi) &\wh{\tilde{b}}(\xi+\pi)
				 \end{matrix}\right]
				\left[ \begin{matrix}
					 \ol{\wh{a}(\xi)}^\tp &\ol{\wh{b}(\xi)}^\tp\\
					 \ol{\wh{a}(\xi+\pi)}^\tp &\ol{\wh{b}(\xi+\pi)}^\tp
				 \end{matrix}\right]
				=I_{2r}, \qquad \xi\in \R.
				\]
				\item[(4)] Every element in $\psi$ and $\tilde{\psi}$ has at least one vanishing moment, i.e., $\int_{\R} \psi(x)dx=\int_{\R} \tilde{\psi}(x) dx=0$.
			\end{enumerate}
		\end{theorem}
		
		To solve the elliptic interface problem \eqref{model}, we take the tensor product of wavelets on $\mathcal{I}:=(0,1)$. Without explicitly involving the dual, the direct approach presented in \cite{HM21a} allows us to construct all possible locally compactly supported biorthogonal wavelets in $\LpI{2}$ satisfying prescribed boundary conditions and maximum vanishing moments from any compactly supported biorthogonal wavelets in $\Lp{2}$. Moreover, we can avoid using more boundary wavelet basis functions than necessary and retain as many interior wavelet basis functions as possible. This consideration is essential for practical numerical implementation. Our direct approach produces a biorthogonal wavelet $(\tilde{\mathcal{B}}^{1D}_{J_0},\mathcal{B}^{1D}_{J_0})$ in $\LpI{2}$, where
		\[
		 \mathcal{B}^{1D}_{J_0}:=\Phi_{J_0}\cup \cup_{j=J_0}^\infty \Psi_j \subseteq \LpI{2},
		\qquad \tilde{\mathcal{B}}^{1D}_{J_0}:=\tilde{\Phi}_{J_0}\cup \cup_{j=J_0}^\infty \tilde{\Psi}_j \subseteq \LpI{2},
		\]
		the integer $J_0\in \N$ denotes the coarsest scale level, and
		\be \label{PhiPsij}
		\begin{aligned}
			&\Phi_{J_0} := \{\phi^{L}_{J_0;0}\}\cup
			\{\phi_{J_0;k} \setsp n_{l,\phi} \le k\le 2^{J_0}-n_{h,\phi}\}\cup \{ \phi^{R}_{J_0;2^{J_0}-1}\},\\
			&\Psi_{j} := \{\psi^{L}_{j;0}\}\cup
			\{\psi_{j;k} \setsp n_{l,\psi} \le k\le 2^{j}-n_{h,\psi}\}\cup \{ \psi^{R}_{j;2^{j}-1}\},\quad j\ge J_0,
		\end{aligned}
		\ee
		with $n_{l,\phi}, n_{h,\phi}, n_{l,\psi}, n_{h,\psi}$ being known integers, $\phi^L, \phi^R$ being boundary refinable functions, and $\psi^L, \psi^R$ being boundary wavelets that are finite subsets of functions in $\LpI{2}$.
Recall that $\psi_{j;k}:=2^{j/2}\psi(2^j\cdot-k)$.
We define $\tilde{\mathcal{B}}^{1D}_{J_0}$ the same way, except we add $\sim$ to each element in $\mathcal{B}^{1D}_{J_0}$ for a natural bijection.
		
	%
	To obtain $d$-dimensional biorthogonal wavelets in $\LpO{2}$ with $\Omega=(0,1)^d$, we use the tensor product of the one-dimensional biorthogonal wavelet in $\LpI{2}$. Given 1D functions $f_1,\ldots,f_d$, define $(f_1 \otimes \ldots \otimes f_d) (x_1,\ldots,x_d) := f_1(x_1) \ldots f_d(x_d)$ for $x_1,\ldots,x_d \in \R$. If $F_1,\ldots,F_d$ and $F$ are sets containing 1D functions, we define
	\begin{align*}
F_1 \otimes \cdots \otimes F_d & := \{f_1 \otimes \cdots \otimes f_d : f_i \in F_i, 1\le i \le d\},\qquad
\otimes^d F:=F \otimes \cdots \otimes F.
	\end{align*}
	Now, define
		\be \label{B2D}
		\mathcal{B}^{dD}_{J_0} := \Phi^{dD}_{J_0} \cup \cup_{j=J_0}^{\infty} \Psi^{dD}_{j},\qquad
		 \tilde{\mathcal{B}}^{dD}_{J_0} := \tilde{\Phi}^{dD}_{J_0} \cup \cup_{j=J_0}^{\infty} \tilde{\Psi}^{dD}_{j},
		\ee
		where
		\[
		\begin{aligned}
			& \Phi^{dD}_{J_0}:=\otimes^d \Phi_{J_0}, \quad \Psi^{dD}_{j}:= (\otimes^d \{\Phi_j ,\Psi_j\}) \backslash \Phi^{dD}_{j}, \quad
			& \tilde{\Phi}^{dD}_{J_0}:=\otimes^d  \tilde{\Phi}_{J_0}, \quad \tilde{\Psi}^{dD}_{j}:=(\otimes^d \{\tilde{\Phi}_j ,\tilde{\Psi}_j\}) \backslash \tilde{\Phi}^{dD}_{j},
		\end{aligned}
		\]
		and $\Phi_j, \Psi_j, \tilde{\Phi}_j, \tilde{\Psi}_j$ are defined as in \eqref{PhiPsij}. As a concrete example, we have the following sets for $d=2$
		\be \label{PhiPsi2D}
		\begin{aligned}
			& \Phi^{2D}_{J_0} =\Phi_{J_0}\otimes \Phi_{J_0}, \quad \Psi^{2D}_{j} =\{\Phi_j \otimes \Psi_j,  \Psi_{j}  \otimes \Phi_{j}, \Psi_j \otimes \Psi_j\}, \\
			& \tilde{\Phi}^{2D}_{J_0} =\tilde{\Phi}_{J_0}\otimes \tilde{\Phi}_{J_0}, \quad \tilde{\Psi}^{2D}_{j} =\{\tilde{\Phi}_j \otimes \tilde{\Psi}_j,  \tilde{\Psi}_{j}  \otimes \tilde{\Phi}_{j}, \tilde{\Psi}_j \otimes \tilde{\Psi}_j\}.
		\end{aligned}
		\ee
		By \cite[Theorem 1.2]{HM23}, if $\phi \in H^{1}(\R)$ and $\Phi_{J_0} \subseteq H^1_0(0,1)$, then $(\tilde{\mathcal{B}}^{dD}_{J_0},\mathcal{B}^{dD}_{J_0})$ is a biorthogonal wavelet in $\LpO{2}$ and its properly scaled version defined below 
		\be \label{BJH1}
		 \mathcal{B}^{H^1_0(\Omega)}_{J_0} := [2^{-J_0}\Phi^{dD}_{J_0}] \cup \cup_{j=J_0}^{\infty} [2^{-j}\Psi^{dD}_{j}]
		\ee
		is a Riesz basis of the Sobolev space $H^1_0(\Omega)$.
That is, (1) the linear span of $\mathcal{B}^{H^1_0(\Omega)}_{J_0}$ is dense in $H^1_0(\Omega)$, and (2) there exist positive constants $C_{\mathcal{B},1},C_{\mathcal{B},2}>0$ such that
		 \begin{equation}\label{stability}
			C_{\mathcal{B},1} \sum_{\eta \in \mathcal{B}^{H^1_0(\Omega)}_{J_0}} |c_{\eta}|^2 \le \Big\| \sum_{\eta \in \mathcal{B}^{H^1_0(\Omega)}_{J_0}} c_{\eta} \eta \Big\|^2_{H^{1}(\Omega)} \le C_{\mathcal{B},2} \sum_{\eta \in \mathcal{B}^{H^1_0(\Omega)}_{J_0}} |c_{\eta}|^2
		\end{equation}
for all sequences $\{c_\eta\}_{\eta\in \mathcal{B}^{H^1_0(\Omega)}_{J_0}}\in \ell^2$.
Furthermore, $(\tilde{\mathcal{B}}^{H^{-1}(\Omega)}_{J_0},\tilde{\mathcal{B}}^{H^{1}_0(\Omega)}_{J_0})$ forms a biorthogonal basis for the dual Sobolev spaces $(H^{-1}(\Omega),H^{1}_0(\Omega))$ (see \cite{hanbook} for the definition of such spaces) and any function $u \in H^1_0(\Omega)$ can be represented in terms of the following wavelet representation
		\be \label{u:repr}
		u = \sum_{\eta \in \mathcal{B}^{H^1_0(\Omega)}_{J_0}} c_\eta \eta
\qquad \mbox{with} \qquad
\{c_\eta:=\la u, \tilde{\eta}\ra\}_{\eta \in \mathcal{B}^{H^1_0(\Omega)}_{J_0}}
\in \ell^2
		\ee
		%
with the above series converging absolutely in $H^{1}_0(\Omega)$, where $\tilde{\eta}\in \tilde{\mathcal{B}}^{H^{-1}(\Omega)}_{J_0}$ is the corresponding element of $\eta\in \mathcal{B}^{H^1_0(\Omega)}_{J_0}$.
Indeed, there are many biorthogonal wavelet and scalar wavelet bases in $H^1_0(\Omega)$, which can be used for solving PDEs numerically (e.g., \cite[Section 7]{HM21a}, \cite[Section~3.2]{HM23} and \cite{Cer19,dku99}).

\subsection{Methodology for solving the model problem \eqref{model}}	 \label{subsec:method}
		
We first review the classical/standard wavelet method. Let $\Omega:=(0,1)^d$ with $d\in \N$.
%
For a given coarse level $J_0$ and for all $J\ge J_0$, we define the traditional finite-dimensional wavelet element space truncated at the scale level $J$ as follows:
\be \label{B2DS}
\mathcal{B}^{dD}_{J_0,J} := \Phi^{dD}_{J_0} \cup \cup_{j=J_0}^{J-1} \Psi^{dD}_{j}
\quad \mbox{and}\quad
\quad
\mathcal{B}^{H^1_0(\Omega)}_{J_0,J} := [2^{-J_0}\Phi^{dD}_{J_0}] \cup \cup_{j=J_0}^{J-1} [2^{-j} \Psi^{dD}_{j}].
\ee
Obviously, $\mathcal{B}^{H^1_0(\Omega)}_{J_0,J}$ is a finite subset of
$\mathcal{B}^{H^1_0(\Omega)}_{J_0}$,
where $\mathcal{B}^{H^1_0(\Omega)}_{J_0}$ is defined as in \eqref{BJH1}. Using a uniform grid,
the standard FEM only uses the basis $\Phi^{dD}_{J}$ and its finite element space $V_{J}:=\mbox{span}( \Phi^{dD}_{J})$.
It is very important to notice that
$\mbox{span}(\mathcal{B}^{dD}_{J_0,J}) = V_{J}$ and
$\mbox{span}(\mathcal{B}^{H^1_0(\Omega)}_{J_0,J}) = V_{J}$. In other words, both $\mathcal{B}^{H^1_0(\Omega)}_{J_0,J}$ (or $\mathcal{B}^{dD}_{J_0,J}$) and $\Phi^{dD}_{J}$ span the same (finite element) space $V_{J}$. The numerical solution to \eqref{model} obtained by the traditional wavelet Galerkin method using only $\mathcal{B}^{dD}_{J_0,J}$ is the same as the solution obtained by using $\Phi^{dD}_{J}$ in FEMs. In the context of the model problem \eqref{model}, using only the traditional wavelet basis $\mathcal{B}^{H^1_0(\Omega)}_{J_0,J}$ inevitably suffers from the same convergence issue faced in the standard FEM. More specifically, the observed convergence rate will typically be at most $0.5$ in the $H^1(\Omega)$-norm because the solution $u\not\in H^{1.5}(\Omega)$ and has discontinuous gradients across the interface $\Gamma$.

We now describe our proposed method.
Let $\SSupp(\tilde{\eta})$ be the smallest closed interval such that the function $\tilde{\eta}$ vanishes outside it. To address the issue of discontinuous gradients $\nabla u$, we incorporate higher-resolution wavelets defined at the scale level $j$ by
\be \label{Sj}
\mathcal{S}_{j} := \{\eta \in \Psi^{dD}_{j} \; : \; \SSupp(\tilde{\eta})  \cap \Gamma \neq \emptyset \;\; \text{and} \;\; \tilde{\eta} \in \tilde{\Psi}^{dD}_{j}\}, \qquad j\in \N
\ee
in addition to the standard wavelet elements $\mathcal{B}^{dD}_{J_0,J}$. 
Define $\#\mathcal{S}_j$ to be the cardinality of $\mathcal{S}_j$.
Because the interface $\Gamma$ is $(d-1)$-dimensional with $d\ge1$, we have $\#\mathcal{S}_j=\bo(2^{j(d-1)})$ for all $j\ge J_0$. Consequently,
the total number of elements in the sets $\mathcal{S}_j$ for all $j=J, \ldots, \tilde{J}-1$ with to-be-determined $\tilde{J} \in \N$ is
\[
N_J^{\tilde{J}}:=
\sum_{j=J}^{\tilde{J}-1} \#\mathcal{S}_j=
\sum_{j=J}^{\tilde{J}-1} \bo(2^{j(d-1)}) =
\begin{cases}
\bo(\tilde{J}-J), & \quad d=1,\\
\bo(2^{\tilde{J}(d-1)}), & \quad d \ge 2.
\end{cases}
\]
Since we require the augmented finite subset of the wavelet basis to be a fixed constant multiple of the traditional one, we must require $N_J^{\tilde{J}}=\bo(\#\mathcal{B}^{dD}_{J_0,J})=\bo(2^{dJ})$. Consequently,
the largest possible scale level
$\tilde{J}$ must satisfy
$\tilde{J}=\bo(2^{J})$ for $d=1$ and $\tilde{J}=\frac{dJ}{d-1}+\bo(1)$ for $d\ge2$. For simplicity, we take
\be \label{tildeJ}
\tilde{J}:=2^J \quad\mbox{for }  d=1
\quad \mbox{and}\quad
\tilde{J}:=\left \lfloor \tfrac{dJ}{d-1} \right \rfloor-1\quad \mbox{for } d \ge 2.
\ee
More specifically, with the choice of $\tilde{J}$ in \eqref{tildeJ}, we shall use
\be \label{BJ}
\mathcal{B}^S_{J_0,J}
:=\mathcal{B}^{dD}_{J_0,J} \cup \cup_{j=J}^{\tilde{J}-1} \mathcal{S}_j,
\quad \mbox{or equivalently},\quad
\mathcal{B}^{S,H^1_0(\Omega)}_{J_0,J}
:=\mathcal{B}^{H^1_0(\Omega)}_{J_0,J} \cup \cup_{j=J}^{\tilde{J}-1} [2^{-j}\mathcal{S}_j],
\ee
where the superscript $S$ indicates that we add extra wavelet elements $S_j, j=J,\ldots, \tilde{J}-1$ to the traditional wavelet basis $\mathcal{B}^{dD}_{J_0,J}$.
Note that $\#\mathcal{B}^{dD}_{J_0,J}=\bo(h^{-d})$ with mesh size $h:=2^{-J}$. Due to the choice of $\tilde{J}$ in \eqref{tildeJ}, it is easy to observe that we still preserve $\#\mathcal{B}^{S,H^1_0(\Omega)}_{J_0,J}=\bo(h^{-d})$ for all $J\ge J_0$.

Recall that a bilinear form is $\blf(u,v):=\la a \nabla u, \nabla v\ra_\Omega$ as in \eqref{weak}.
By the weak formulation in \eqref{weak}
with $g=0$ and $g_b=0$ in \eqref{model0} (i.e., \eqref{model} with $g_b=0$) and
considering the approximated function \[
u_J := \sum_{\eta \in \mathcal{B}^{S,H^1_0(\Omega)}_{J_0,J}} c_{\eta} \eta\quad \mbox{ with to-be-determined coefficients } \{c_\eta\}_{\eta \in \mathcal{B}^{S,H^1_0(\Omega)}_{J_0,J}},
\]
our wavelet Galerkin method reduces to finding all the coefficients $c_\eta$ for $\eta \in \mathcal{B}^{S,H^1_0(\Omega)}_{J_0,J}$ such that
\be \label{auJu}
\blf(u_J,v):=\langle a \nabla u_J, \nabla v\rangle_{\Omega} = \langle f, v \rangle_{\Omega} - \langle g_{\Gamma}, v \rangle_{\Gamma}, \quad \forall  v \in \mathcal{B}^{S,H^1_0(\Omega)}_{J_0,J}.
\ee

To understand our choice of the additional higher-scale level elements in $\cup_{j=J}^{\tilde{J}-1} \mathcal{S}_j$,
we provide some intuition and informal arguments for the proof of convergence.
Assume that $u_{\pm}\in H^m(\Omega_{\pm})$ with $m\ge 2$ and we use biorthogonal wavelets in $L^2(\Omega)$ with approximation order at least $m$.
Note that the true solution $u\in H^1_0(\Omega)$ to the elliptic interface problem \eqref{model} has the wavelet representation in
\eqref{u:repr}.
We consider an approximated function in the space spanned by $\mathcal{B}^{S,H^1_0(\Omega)}_{J_0,J}$:
\[
\mathring{u}_J:=\sum_{\eta\in \mathcal{B}^{S,H^1_0(\Omega)}_{J_0,J}}
\la u, \tilde{\eta}\ra \eta.
\]
Therefore, the error of our approximation is $u-\mathring{u}_J$ satisfies the following identity:
\[
u-\mathring{u}_J=
\sum_{j=J}^\infty \sum_{\alpha\in \Psi^{dD}_j \bs \mathcal{S}_j}
\la u,2^j \tilde{\alpha_j}\ra 2^{-j} \alpha_j+
\sum_{j=\tilde{J}}^\infty \sum_{\alpha_j\in \mathcal{S}_j} \la u, 2^j \tilde{\alpha}_j\ra 2^{-j} \alpha_j.
\]
Hence, to apply the standard argument in FEMs for convergence in the $H^1(\Omega)$-norm,
we have to estimate the first summation for the smooth part in $H^m(\Omega_{\pm})$ and the second summation for the singular part in $H^{1.5-\gep}(\Omega)$ for any $\gep>0$.
Note that $h:=2^{-J}$ is the mesh size. In what follows, the positive constants $C_1,C_2,C_3$  are generic, and independent of the grid size and the scale level, but could depend on $\gep>0$.
The wavelet characterizations of the smooth parts $u_{\pm}\in H^m(\Omega_{\pm})$ give us
\[
\sum_{j=J}^{\infty} \sum_{\alpha_j \in \Psi^{dD}_j \backslash \mathcal{S}_j} 2^{2j} |\la u, \tilde{\alpha}_j \ra|^2 \le 2^{-2 J(m-1)} C_1=h^{2(m-1)}C_1,
\]
while for $u\in H^{1.5-\gep}(\Omega)$ for any $\gep>0$, we have
\[
\sum_{j=\tilde{J}}^{\infty} \sum_{ \alpha_j \in \mathcal{S}_j} 2^{2j} |\la u, \tilde{\alpha}_j \ra|^2 \le 2^{(2\gep-1) \tilde{J}} C_2=h^{(1-2\gep)\tilde{J}/J}C_2.
\]
Hence,
using the Riesz basis property \eqref{stability},
it follows that the squared error of the approximation is
\be \label{heuristic:rates}
\begin{aligned}
\|u-\mathring{u}_J\|_{H^1(\Omega)}^2
&\le C_{\mathcal{B},2} \left( \sum_{j=J}^{\infty} \sum_{\alpha_j \in \Psi^{d D}_j \backslash \mathcal{S}_j} 2^{2j} |\la u, \tilde{\alpha}_j \ra|^2
+
\sum_{j=\tilde{J}}^{\infty} \sum_{ \alpha_j \in \mathcal{S}_j} 2^{2j} |\la u, \tilde{\alpha}_j \ra|^2\right)\\
&\le C_3 h^{\min(2(m-1), 2(1/2- \gep) \tilde{J}/J)}
\end{aligned}
\ee
for all $\gep>0$, where the generic constant $C_3$ depends on $\gep>0$.
By the choice of $\tilde{J}$ in \eqref{tildeJ}, under the $H^1(\Omega)$-norm, we see that the convergence order is $m-1$ for $d=1$ (in fact, it suffices to take any $\tilde{J}$ satisfying $\frac{2(m-1)J}{1-2\gep}\le \tilde{J}\le 2^J$ for this case) and $\frac{d}{2(d-1)}-\frac{d}{d-1}\gep$ with any $\gep>0$ for $d\ge 2$.
That is, the maximum attainable convergence rate in the $H^{1}(\Omega)$-norm for our method is $m-1$ in 1D for every $m\ge 2$, but only order $1-2\gep$ in 2D, and $\tfrac{1}{2} < \tfrac{d}{2(d-1)}-\frac{d}{d-1}\gep < 1$ in $d$D with $d \ge 3$, for any approximation order $m\ge 2$. In the formal proof of convergence discussed in \cref{sec:proof}, we use several technical steps to remove this dependence on $\gep$ and introduce a $\log(h)$ factor.

This conclusion crucially relies on the assumptions that we enrich the traditional finite subset of the wavelet basis only with basis functions from higher scale levels near the interface $\Gamma$, and that the augmented subset is a fixed constant multiple of the traditional one.
To obtain a high-order method for $d = 2$ or to enhance the convergence rates for $d\ge 3$, more sophisticated ideas are needed, which are left for future work.
The next theorem presents our main result on the convergence of our method for solving the 2D elliptic interface problem using $\mathcal{B}^{S,H^1_0(\Omega)}_{J_0,J}$ in \eqref{BJ} as $J\to\infty$, and the uniform boundedness of the condition numbers of the coefficient matrices.
We assume that $g=0$ on $\Gamma$ in the first jump condition \eqref{jump1} to avoid a discontinuous $u$ and $g_b=0$ in $\partial \Omega$, 
consistent with the standard FEM argument. Since $\mathcal{B}^{S,H^1_0(\Omega)}_{J_0,J}$ is a finite subset of the Riesz wavelet basis $\mathcal{B}^{H^1_0(\Omega)}_{J_0}$
in $H^1_0(\Omega)$,
the condition numbers of coefficient matrices from \eqref{auJu} are uniformly bounded and independent of the mesh size $h$ and resolution level $J$.
For technical reasons, we defer the proof and its auxiliary results to \cref{sec:proof}.

%
\begin{theorem}\label{thm:converg}
Let $\Omega:=(0,1)^2$ and set $\tilde{J}:=2J-1$ in \eqref{tildeJ} for $d=2$.
Assume that $\mathcal{B}^{S,H^1_0(\Omega)}_{J_0,J}$ in \eqref{BJ} is constructed from a compactly supported biorthogonal wavelet $(\{\tilde{\phi};\tilde{\psi}\},\{\phi;\psi\})$ in $L^{2}(\R)$ such that each of its elements is compactly supported, $\phi,\psi \in H^1(\R)$, $\phi$ has the approximation order at least $2$ (equivalently, $\vmo(\tilde{\psi}) \ge 2$), and $\tilde{\phi}, \tilde{\psi} \in H^t(\R)$ for some $t>0$. Under the standard assumptions $g=0$ and $g_b=0$ in
\eqref{model0} (i.e., $g_b=0$ in \eqref{model}) as in
finite element methods,
let $u\in H^1_0(\Omega)$ be the exact solution of the model problem \eqref{model} with variable functions $a,f,g_\Gamma$ such that
\be \label{assumption:u}
u_+:=u\chi_{\Omega_+}\in H^2(\Omega_+)\quad \mbox{ and }\quad u_-:=u\chi_{\Omega_-}\in H^2(\Omega_-).
\ee
We assume that the interface $\Gamma$ is of class $\mathscr{C}^2$.
For $J\ge J_0$, define
$h:=2^{-J}$ as the mesh size and $N_J$ as the cardinality of 
$\mathcal{B}^{S,H^1_0(\Omega)}_{J_0,J}$. Define $V_h^{wav}:=
\mbox{span}(\mathcal{B}^{S,H^1_0(\Omega)}_{J_0,J})$.
Let $u_h =u_J:=\sum_{\eta \in \mathcal{B}^{S,H^1_0(\Omega)}_{J_0,J}} c_{\eta} \eta \in V_h^{wav}$ be the numerical solution obtained from \eqref{auJu} (i.e., the weak formulation of \eqref{model} in the wavelet subspace $V_h^{wav}$) by using the basis $\mathcal{B}^{S,H^1_0(\Omega)}_{J_0,J}$ in \eqref{BJ}.
Then for all $J \ge J_0$, there exists a positive constant $C$, independent of all $J$, $h$ and $N_J$, such that
\be \label{converg:H1}
\|u_h - u \|_{H^{1}(\Omega)}\le C h |\log (h)|, \qquad
\|u_J - u \|_{H^{1}(\Omega)}\le C N_J^{-1/2} J,
\ee
and
\be \label{converg:L2}
\|u_h - u \|_{\LpO{2}}\le C h^2 |\log (h)|^2,
\qquad \|u_J - u \|_{\LpO{2}}\le C N_J^{-1} J^2,
\ee
where $\log (\cdot)$ is the natural logarithm and in fact, the above generic constant $C$ in \eqref{converg:H1} and \eqref{converg:L2} is bounded by $c (\|u_+\|^2_{H^2(\Omega_+)}+\|u_-\|^2_{H^2(\Omega_-)})^{1/2}$ with a positive constant $c$ only depending on 
the interface $\Gamma$ and the wavelet basis.
Moreover, the condition number must satisfy
\be \label{cond:number}
\kappa\Big([\blf(\alpha,\beta)]_{\alpha,\beta\in \mathcal{B}^{S,H^1_0(\Omega)}_{J_0,J}}\Big) \le C_w \| a\|_{L^\infty(\Omega)} \|a^{-1}\|_{L^\infty(\Omega)}, \quad \text{for all } J \ge J_0,
\ee
where $\kappa$ denotes the condition number of the coefficient matrix and $C_w$ is a positive constant that only depends on the wavelet basis and the domain $\Omega$, but $C_w$ is independent of the interface $\Gamma$.
\end{theorem}
	
We shall prove \cref{thm:converg} under the abstract assumption \eqref{assumption:u} on $u_+$ and $u_-$, which can be satisfied by specifying concrete conditions on variable functions $a, f$ and $g_\Gamma$. For example,
according to \cite[Theorem 10.1 and Section~16]{LUbook},
the assumption \eqref{assumption:u} on $u$ is satisfied if $a_+:=a\chi_{\Omega_+}\in C^1(\overline{\Omega_+})$, $a_-:=a\chi_{\Omega_-}\in C^1(\overline{\Omega_-})$, $f\in \LpO{2}$, and $g_\Gamma\in H^{1/2}(\Gamma)$.
Even though we assume $\mbox{ess}$-$\inf_{x,y\in \Omega} a(x,y)>0$, the variable functions $a\in L^\infty(\Omega)$ and $f\in \LpO{2}$ could be discontinuous across the interface $\Gamma$.
It is also important to notice that $N_J$, the cardinality of the set $\mathcal{B}^{S,H^1_0(\Omega)}_{J_0,J}$, satisfies $h^{-2}\le N_J\le C_\Gamma h^{-2}$ with $h:=2^{-J}$ for a positive constant $C_\Gamma$ only depending on the interface curve $\Gamma$, in particular, the length of $\Gamma$. Our proof of \cref{thm:converg}
extensively uses the dual part of the biorthogonal wavelet basis and relies on the weighted Bessel properties and results of wavelets in fractional Sobolev spaces, plus standard FEM arguments and various inequalities in fractional Sobolev spaces. The key techniques for proving convergence are quite similar across dimensions, with the 1D case being special, because the interface $\Gamma$ is a single point and because of special results for Sobolev spaces in one dimension.

\begin{remark} \label{remark:quad}
	\cref{thm:converg} is formulated using the exact interface and Galerkin forms. The same convergence rates remain valid, when the interface curve and inner products are approximated numerically, provided that the curve-approximation and quadrature errors are at most of the same order as the discretization errors. We justify this remark at the end of \cref{sec:proof} using a perturbation argument that exploits the Riesz basis property.  
\end{remark}

 The linear system in \eqref{auJu} can be conveniently formed by first computing the inner products of the bilinear form involving the scaling functions $\phi$ at the highest scale level, and then using the refinability and fast transform of our wavelet basis (see \eqref{refmat}). To evaluate these inner products, we locally approximate the interface curve $\Gamma$ by a polynomial of high enough degree, map each curve-sided triangle or quadrilateral integration domain to a reference right triangle or square, and then apply the Gaussian quadrature. These three key ingredients enable the inner products to be computed with sufficient accuracy to attain the convergence rates established in \cref{thm:converg}. 

		\begin{figure}[htbp]
			\centering
			 \begin{subfigure}[b]{0.24\textwidth} \includegraphics[width=\textwidth]{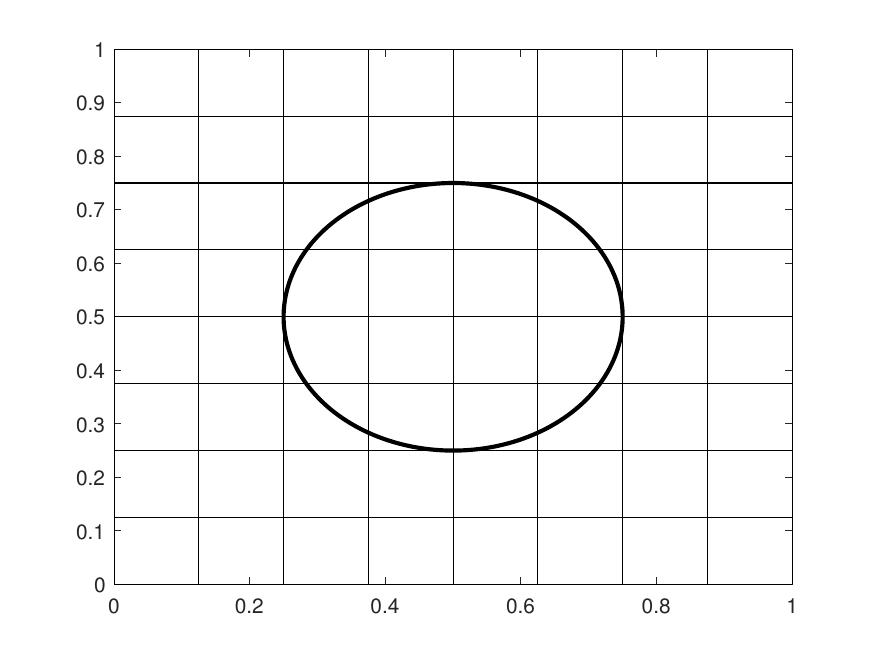}
				 \caption{$\mathcal{B}^{2D}_{3,4}$}
			\end{subfigure}
			 \begin{subfigure}[b]{0.24\textwidth} \includegraphics[width=\textwidth]{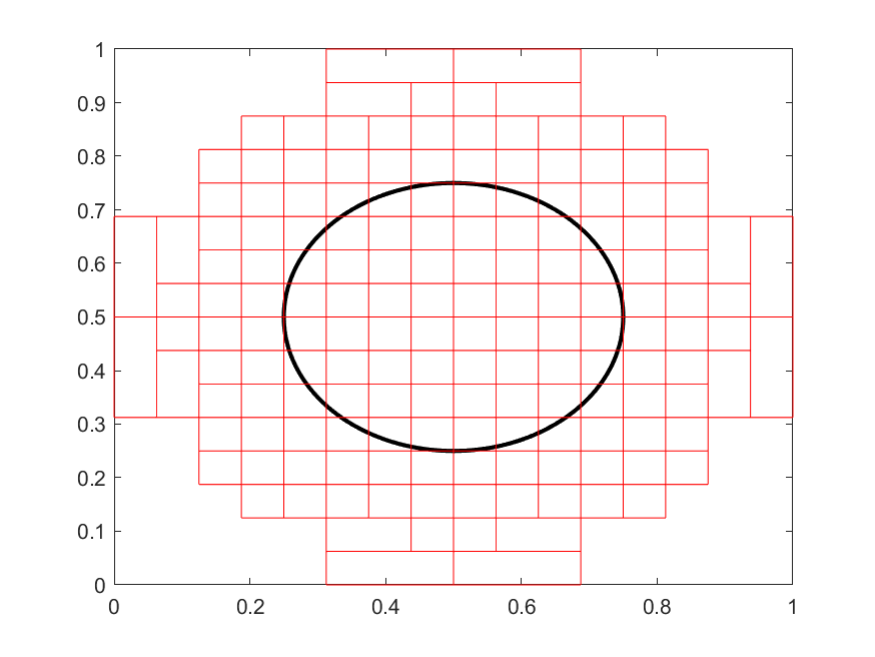}
				 \caption{$\mathcal{S}_4$}
			\end{subfigure}				
			 \begin{subfigure}[b]{0.24\textwidth} \includegraphics[width=\textwidth]{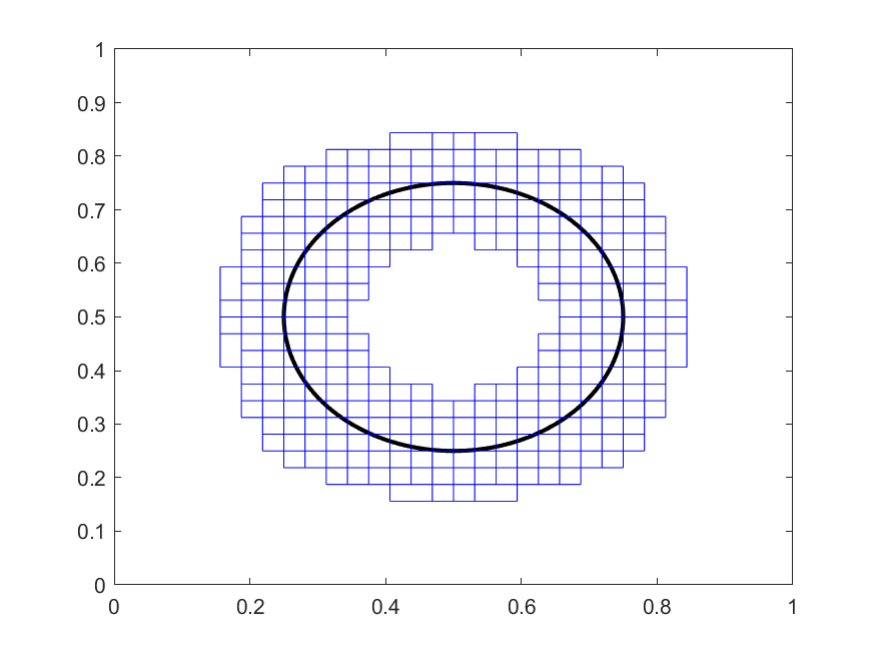}
				 \caption{$\mathcal{S}_5$}
			\end{subfigure}
			 \begin{subfigure}[b]{0.24\textwidth} \includegraphics[width=\textwidth]{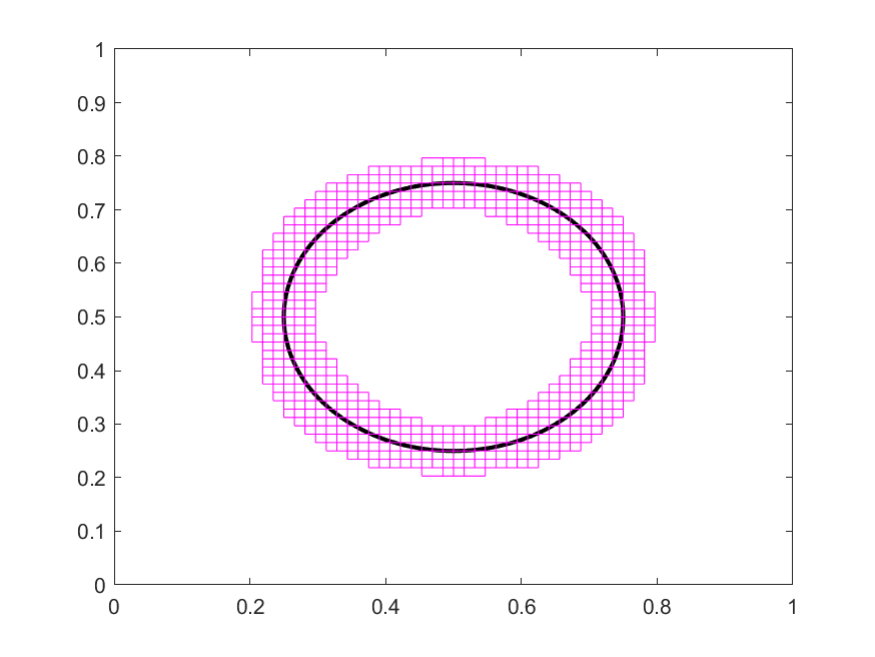}
				 \caption{$\mathcal{S}_6$}
			\end{subfigure}
			\caption{For simplicity, we assume that the interface curve, $\Gamma$, is a circle. Panel (a) depicts the overlapping supports of wavelets in $\mathcal{B}^{2D}_{3,4}$. Panels (b)-(d) depict the overlapping supports of extra wavelet elements/basis functions added along the interface $\Gamma$, which make up the set $\cup_{j=4}^{6} [2^{-j} \mathcal{S}_j]$.}
			\label{fig:BSsupp}
		\end{figure}
		
	\section{Numerical Experiments} \label{sec:exp}
		
		In this section, we present the setup of the numerical experiments and the corresponding performance of our wavelet Galerkin method.
		
		 \subsection{A biorthogonal wavelet basis in $H^1_0(\Omega)$ derived from bilinear finite elements}\label{subsec:wavelet}
		
		Interpolating functions play a critical role in numerical PDEs, wavelet analysis, and computer aided geometric design (e.g., see \cite{han24} and references therein).
 The simplest example of compactly supported interpolating functions is probably the hat function $\phi(x):=\max(1-|x|,0)$ for $x\in \R$, which is extensively used in numerical PDEs and approximation theory. The hat function $\phi$ satisfies the refinement equation $\phi=\frac{1}{2}\phi(2\cdot-1)+\phi(2\cdot)+\frac{1}{2}\phi(2\cdot+1)$ and $\wh{\phi}(0)=1$.
		In what follows, we recall a biorthogonal wavelet basis in $\LpI{2}$ derived from the hat function $\phi$ and discussed in \cite[Example~7.5]{HM21a}; also see \cite{dku99}. We use this biorthogonal wavelet basis in this section, since it is simple and its second-order approximation meets the minimum requirement of \cref{thm:converg} (also refer to the discussion leading up to \eqref{heuristic:rates}).
		
		Let $\phi$ be the hat function. Consider the scalar biorthogonal wavelet $(\{\tilde{\phi};\tilde{\psi}\},\{\phi;\psi\})$ in $\Lp{2}$ with $\widehat{\phi}(0)=\widehat{\tilde{\phi}}(0)=1$ and a biorthogonal wavelet filter bank $(\{\tilde{a};\tilde{b}\},\{a;b\})$ given by
		\begin{align}	 a=&\left\{\tfrac{1}{4},\tfrac{1}{2},\tfrac{1}{4}\right\}_{[-1,1]}, \quad b=\left\{-\tfrac{1}{8},-\tfrac{1}{4},\tfrac{3}{4},-\tfrac{1}{4},-\tfrac{1}{8}\right\}_{[-1,3]},
			\label{ab} \\
			 \tilde{a}=&\left\{-\tfrac{1}{8}, \tfrac{1}{4}, \tfrac{3}{4}, \tfrac{1}{4}, -\tfrac{1}{8} \right\}_{[-2,2]}, \quad \tilde{b}=\left\{-\tfrac{1}{4}, \tfrac{1}{2}, -\tfrac{1}{4}\right\}_{[0,2]}.\label{tatb}
		\end{align}
		In other words, the refinable functions $\phi, \tilde{\phi}$ and the wavelet functions $\psi, \tilde{\psi}$ are determined through the equations in \eqref{phipsi} and \eqref{tphitpsi}.
		The analytic expression of the hat function is $\phi:=(x+1)\chi_{[-1,0)} + (1-x)\chi_{[0,1]}$. Additionally, $\tilde{\phi},\tilde{\psi}\in H^\tau(\R)$ with $\tau<0.440765$, while $\phi, \psi\in H^\tau(\R)$ with $\tau<1.5$. Both wavelet functions $\psi$ and $\tilde{\psi}$ have order two vanishing moments. As stated in \cite[Example 7.5]{HM21a}, the boundary refinable functions and boundary wavelet functions are defined to be
		\begin{align*}
				\psi^{L} & = \tfrac{1}{2} \phi(2\cdot -1)  - \phi(2\cdot-3) + \tfrac{1}{2} \phi(2\cdot - 4), \\
				\tilde{\phi}^{L} & =
				\begin{bmatrix}
					0 & - \tfrac{1}{2}\\
					1 & \tfrac{3}{2} \\
				\end{bmatrix}
				 \tilde{\phi}^{L}(2\cdot)
				+ \begin{bmatrix}
					\tfrac{1}{2}\\
					\tfrac{1}{2}
				\end{bmatrix} \tilde{\phi}(2\cdot -3)
				+ \begin{bmatrix}
					\tfrac{3}{2}\\
					-\tfrac{1}{4}
				\end{bmatrix} \tilde{\phi}(2\cdot -4)
				+ \begin{bmatrix}
					\tfrac{1}{2}\\
					0
				\end{bmatrix} \tilde{\phi}(2\cdot -5)
				+ \begin{bmatrix}
					-\tfrac{1}{4}\\
					0
				\end{bmatrix} \tilde{\phi}(2\cdot -6),\\
				\tilde{\psi}^{L} & =
				\begin{bmatrix}
					0 & -1\\
					-1 & 2
				\end{bmatrix}
				 \tilde{\phi}^{L}(2\cdot) +
				\begin{bmatrix}
					1\\
					0
				\end{bmatrix}
				\tilde{\phi}(2\cdot -3)
				+ \begin{bmatrix}
					0\\
					-\tfrac{1}{2}
				\end{bmatrix} \tilde{\phi}(2\cdot-4),\\
				\psi^{R} & = \phi^{L}(1-\cdot), \quad \tilde{\phi}^{R} = \tilde{\phi}^{L}(1-\cdot), \quad \text{and} \quad \tilde{\psi}^{R} = \tilde{\psi}^{L}(1-\cdot).
		\end{align*}
		Furthermore, we define
		\be \label{hatwavI}
		\begin{aligned}
			& \Phi_{j} := \{\phi_{j;2},\phi_{j;1}\} \cup \{\phi_{j;k}:3 \le k \le 2^{j}-3\} \cup \{\phi_{j;2^j -2},\phi_{j;2^j -1}\},\\
			& \Psi_{j} := \{\psi_{j;1}, \psi^{L}_{j;0}\} \cup \{\psi_{j;k}:2 \le k \le 2^{j}-3\} \cup \{\psi_{j;2^j-2},\psi^{R}_{j;2^{j}-1}\},\\
			& \tilde{\Phi}_{j} := \{ \tilde{\phi}^{L}_{j;0}\} \cup \{\tilde{\phi}_{j;k}:3 \le k \le 2^{j}-3\} \cup \{\tilde{\phi}^{R}_{j;2^{j}-1}\},\\
			& \tilde{\Psi}_{j} := \{ \tilde{\psi}^{L}_{j;0}\} \cup \{\tilde{\psi}_{j;k}:2 \le k \le 2^{j}-3\} \cup \{\tilde{\psi}^{R}_{j;2^{j}-1}\},
		\end{aligned}
		\ee
		$\mathcal{B}^{1D}_{J_0} := \Phi_{J_0} \cup \cup_{j=J_0}^{\infty} \Psi_{j}$, and $\tilde{\mathcal{B}}^{1D}_{J_0} := \tilde{\Phi}_{J_0} \cup \cup_{j=J_0}^{\infty} \tilde{\Psi}_{j}$. Then, $(\tilde{\mathcal{B}}^{1D}_{J_0},\mathcal{B}^{1D}_{J_0})$, where $J_0 \ge 3$, is a biorthogonal wavelet in $\LpI{2}$. Moreover, by construction, we observe that $\mathcal{B}^{1D}_{J_0}$ reproduces $x$ and $x-1$ in the neighbourhood of the left and right endpoints of the unit interval $[0,1]$. Using \eqref{B2D}, \eqref{PhiPsi2D}, and \eqref{BJH1}, we obtain the 2D wavelet basis function used in the Galerkin framework for solving \eqref{model}. Due to item (3) of \cref{thm:bw} and the above relations for $\psi^{L},\tilde{\phi}^{L},\tilde{\psi}^{L},\psi^R$, there exist well-defined refinability matrices $A_{j,j'}$ and $B_{j,j'}$ such that
		\be \label{refmat}
		\Phi_{j} = A_{j,j'}
		\Phi_{j'}
		\quad \text{and} \quad
		\Psi_{j} = B_{j,j'} \Phi_{j'}
		\quad
		\text{for all} \quad j < j',
		\ee
		which is convenient to use in the numerical implementation (in forming the coefficient matrix).
		
		In our experiments, the degree of the Gaussian quadrature rule is 5 and we approximate the interface curve locally with a polynomial degree 3. See \cref{remark:quad} and the paragraph that follows for further details on the numerical quadrature and interface curve approximation.
		
		\cref{fig:2Dwavelets} visualizes the basis functions used in our approximate solution $u_J$. Due to the symmetry in the biorthogonal wavelet basis in \cref{subsec:wavelet}, each term of the approximate solution can be obtained by scaling, shifting, and rotating one of the functions in panels (c)-(h) of \cref{fig:2Dwavelets}.
		\begin{figure}[htbp]
			\centering
			 \begin{subfigure}[b]{0.24\textwidth} \includegraphics[width=\textwidth]{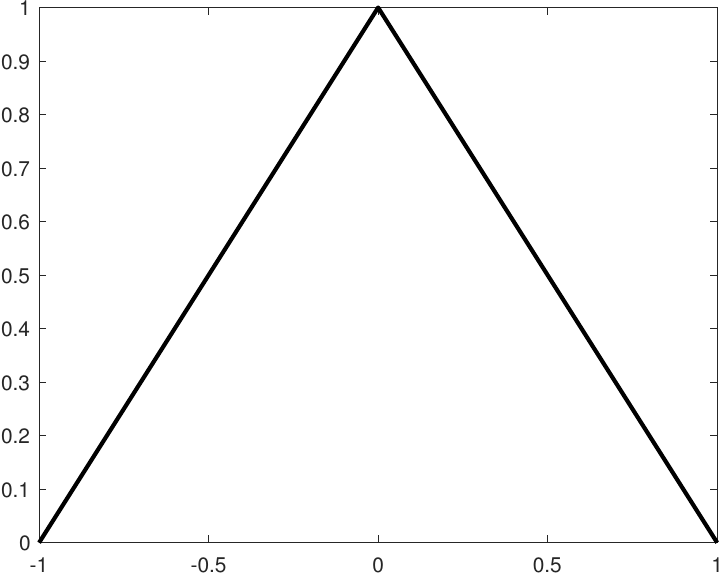}
				\caption{$\phi$}
			\end{subfigure}
			 \begin{subfigure}[b]{0.24\textwidth} \includegraphics[width=\textwidth]{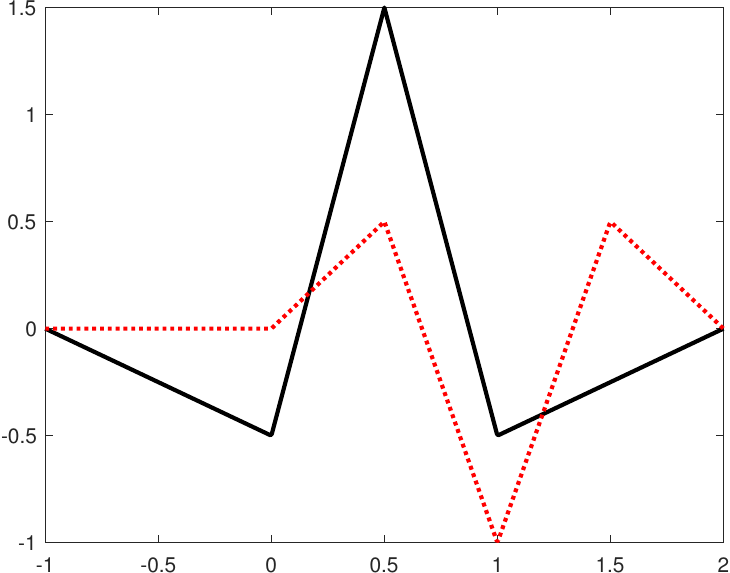}
				\caption{$\psi$ (black), $\psi^{L}$ (red)}
			\end{subfigure}				
			 \begin{subfigure}[b]{0.24\textwidth} \includegraphics[width=\textwidth]{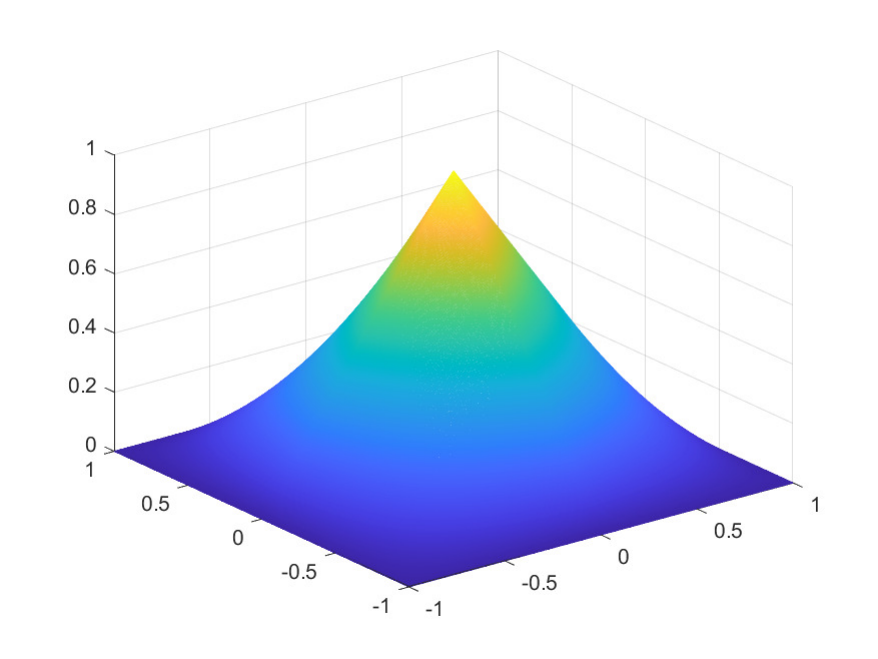}
				\caption{$\phi \otimes \phi$}
			\end{subfigure}
			 \begin{subfigure}[b]{0.24\textwidth} \includegraphics[width=\textwidth]{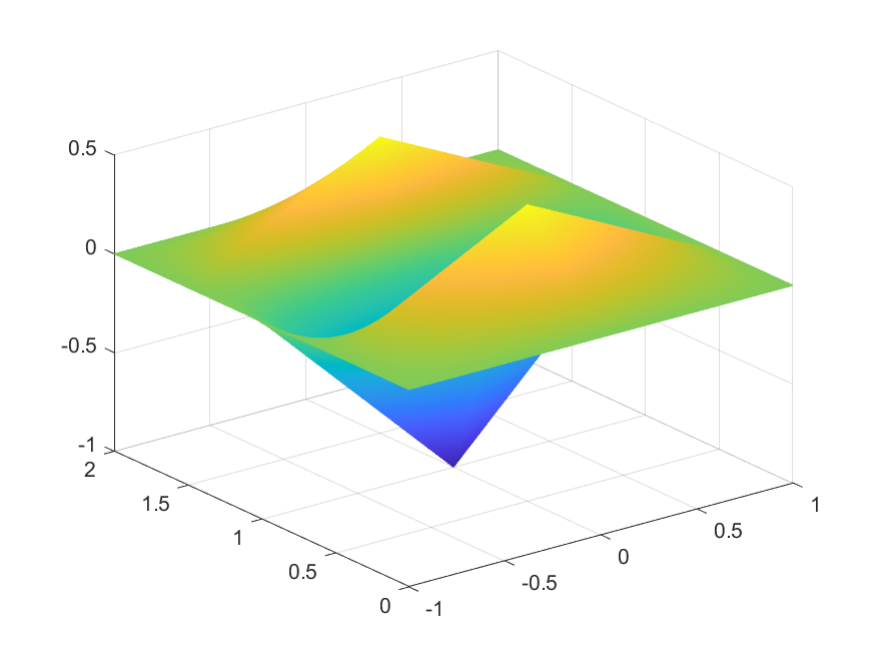}
				\caption{$\phi \otimes \psi^{L}$}
			\end{subfigure}
			 \begin{subfigure}[b]{0.24\textwidth} \includegraphics[width=\textwidth]{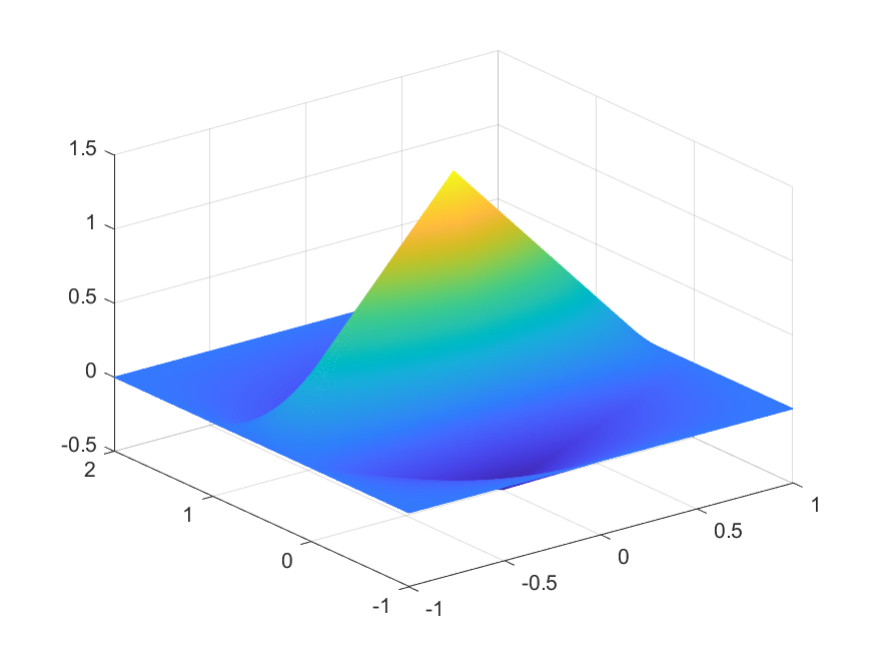}
				\caption{$\phi \otimes \psi$}
			\end{subfigure}		
			 \begin{subfigure}[b]{0.24\textwidth} \includegraphics[width=\textwidth]{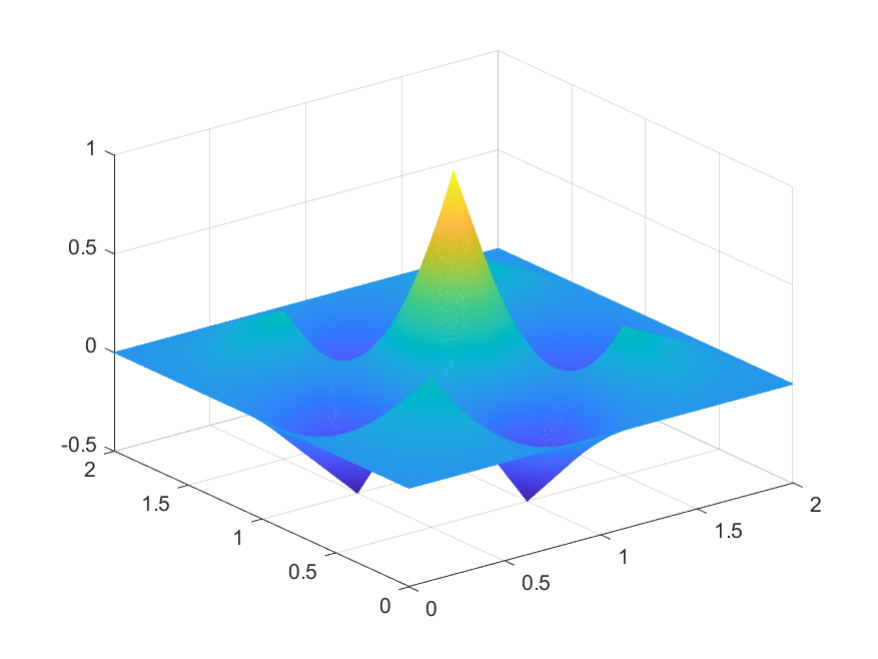}
				\caption{$\psi^{L} \otimes \psi^{L}$}
			\end{subfigure}	
			 \begin{subfigure}[b]{0.24\textwidth} \includegraphics[width=\textwidth]{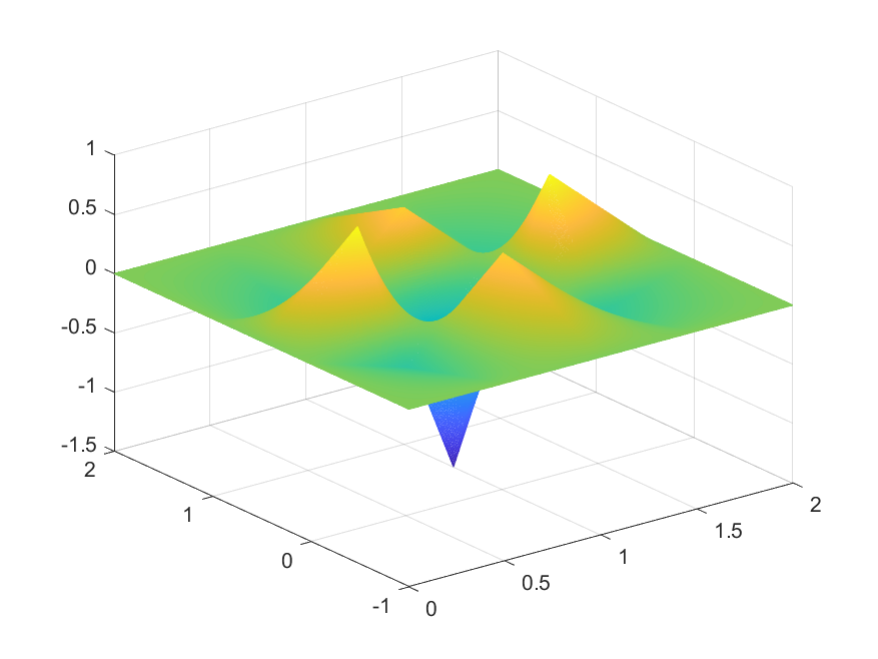}
				\caption{$\psi^{L} \otimes \psi$}
			\end{subfigure}	
			 \begin{subfigure}[b]{0.24\textwidth} \includegraphics[width=\textwidth]{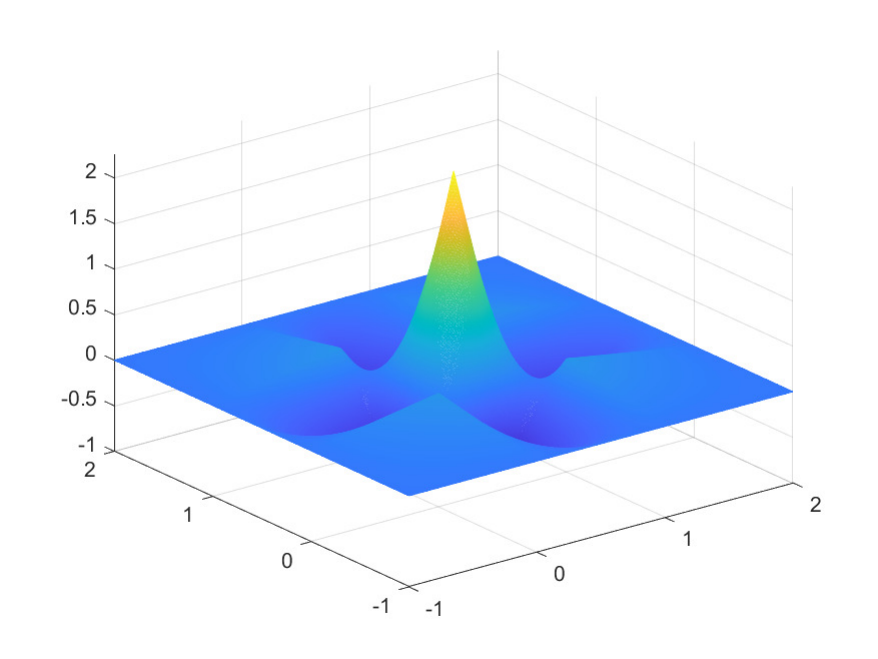}
				\caption{$\psi \otimes \psi$}
			\end{subfigure}	
			\caption{Panels (a)-(b) depict generators of the 1D wavelet basis $\mathcal{B}^{1D}_{J_0}$ with $J_0=3$. Panels (c)-(h) depict generators of the 2D wavelet basis $\mathcal{B}^{2D}_{J_0}$ with $J_0 =3$.}
			\label{fig:2Dwavelets}
		\end{figure}
		Meanwhile, \cref{fig:BSsupp} visualizes the overlapping supports of wavelet basis functions in $\mathcal{B}^{S,H^1_0(\Omega)}_{3,4} \cup \cup_{j=4}^{6} [2^{-j} \mathcal{S}_j]$, and gives us an insight as to how we add the wavelets along the interface in our approximate solution $u_J$.
		
		\subsection{Numerical errors for approximate solutions}
		In each table, $J$ corresponds to the scale level in \eqref{B2DS} with the coarsest scale level $J_0 = 3$. $N_J$ stands for the number of terms (freedom) at the scale level $J$ used in the approximate solution, which is equal to the cardinality of $\mathcal{B}^{S,H^1_0(\Omega)}_{3,J}$ as defined in \eqref{BJ}. For a known exact solution $u$, the quantities under `order' (for the $L^2(\Omega)$ convergence) are computed as follows
		\begin{equation} \label{convorder}
		\text{order} = 2 \log_2\left( \|u_{J-1}-u\|_2 / \|u_{J}-u\|_2 \right)  \left(\log_2\left( N_J / N_{J-1} \right)\right)^{-1}.
		\end{equation}
		If the exact solution $u$ is unknown, we replace $u$ with $u^{\text{ref}}$ in the above formula, where $u^{\text{ref}}$ is the reference solution computed using $\mathcal{B}^{S,H^1_0(\Omega)}_{3,8}$. The convergence in terms of $H^1(\Omega)$-semi-norm is similarly calculated, except we use the solution's gradient. To approximate the $\LpO{2}$-error, we compute all errors using the $l^2$-norm on a fine grid of size $2^{-13}$ in each direction. The condition number $\kappa$ of the coefficient matrix of size $N_J \times N_J$ is calculated by dividing its largest singular value with its smallest singular value. For each example, we compare the errors and the convergence rates of the approximate solution formed by $\mathcal{B}^{S,H^1_0(\Omega)}_{3,J}$ with the one formed by the traditional wavelet method using $\mathcal{B}^{2D,H^{1}_0(\Omega)}_{3,J}$ only. Because both $\mathcal{B}^{2D,H^{1}_0(\Omega)}_{3,J}$ and $\Phi^{2D}_{J}$ span the same finite element space $V_{J}=\mbox{span}(\Phi^{2D}_{J})$, the numerical solutions obtained by the traditional wavelet method using $\mathcal{B}^{2D,H^{1}_0(\Omega)}_{3,J}$  and the FEM $\Phi^{2D}_{J}$ are the same.

		\subsection{Handling nonhomogeneous first jump and/or Dirichlet boundary conditions}
		
		Some of the following examples have nonhomogeneous first jump condition and/or Dirichlet boundary condition. To handle them, we shall exploit the geometry of our interface curve and unit square domain with $(1/2,1/2)$ as its center. For the sake of discussion, we assume that the first jump condition is parameterized in terms of angle and $\Omega_-$ is away from $\partial \Omega$. Since the interface curve is smooth, we are able to radially extend the first jump condition outward and treat its restriction in $\Omega_+$ as an auxiliary solution. More specifically, for $(x,y) \in \R^2 \setminus \Omega_-$, we define
		\be \label{angle}
			\tilde{g}(x,y) := g\left(\Theta(x,y)\right),
			\quad
			\Theta(x,y):= \begin{cases}
				 \arctan\left(\tfrac{y-1/2}{x-1/2}\right), & \text{if} \; x>1/2,\\
				 \arctan\left(\tfrac{y-1/2}{x-1/2}\right) + \pi, & \text{if} \; x<1/2, \; y\ge 1/2,\\
				 \arctan\left(\tfrac{y-1/2}{x-1/2}\right) - \pi, & \text{if} \; x<1/2, \; y< 1/2,\\
				\pi/2, & \text{if} \; x=1/2, \; y>1/2,\\
				-\pi/2, & \text{if} \; x=1/2, \; y<1/2.
			\end{cases}
		\ee
		To handle this nonhomogeneous Dirichlet boundary condition, we build two more auxiliary solutions
		\begin{align*}
			\tilde{u}_{LR} & := (g_b(0,y) - \tilde{g}(0,y)) (1 - x) + (g_b(1,y) - \tilde{g}(1,y)) x,\\
			\tilde{u}_{BT} & := (g_b(x,0) - \tilde{g}(x,0) - \tilde{u}_{LR} (x,0))(1-y) + (g_b(x,1) - \tilde{g}(x,1) - \tilde{u}_{LR} (x,1)) y.
		\end{align*}
		Define the function $G$ such that
		\[
		G_+ = G \chi_{\Omega_+} := \tilde{g}\chi_{\Omega_+} + \tilde{u}_{LR} + \tilde{u}_{BT}
		\quad \text{and} \quad
		G_- = G \chi_{\Omega_-} := \tilde{u}_{LR} + \tilde{u}_{BT}.
		\]
		Next, we aim to find $\tilde{u}_J := \sum_{\eta \in \mathcal{B}^{S,H^1_0(\Omega)}_{J_0,J}} c_{\eta} \eta$ such that
		\[
		\langle a \nabla \tilde{u}_J, \nabla v\rangle_{\Omega} = \langle f, v\rangle_{\Omega} - \langle g_{\Gamma}, v \rangle_{\Gamma} - \langle a \nabla G, \nabla v \rangle , \quad \forall  v \in \mathcal{B}^{S,H^1_0(\Omega)}_{J_0,J}.
		\]
		Finally, we define our approximate solution as $u_J := \tilde{u}_J + \tilde{g} \chi_{\Omega_+} + \tilde{u}_{LR} + \tilde{u}_{BT}$.
		
		\subsection{Examples with known exact solutions $u$}
We present five examples here, where the exact solutions are known.
\cref{thm:converg} guarantees the condition numbers satisfy $\kappa\le C_w \|a\|_{L_\infty(\Omega)}\|a^{-1}\|_{L_\infty(\Omega)}$ in \eqref{cond:number}. In all the numerical examples, we indicate the numerically estimated constant $C_w$ in \eqref{cond:number}.

	\begin{example} \label{ex:circle}
		\normalfont
		We apply our wavelet method to \cite[Section 6]{GL19}, where we transform the original problem so that its domain is the unit square and increase the contrast of the discontinuous $a$. This problem is a typical test problem \eqref{model} with $a_{+}=10^6$,
		$a_{-}=1$,
		\be \label{gamma:circle}
		\Gamma = \{(x,y) \in \Omega \; : \; x(\theta)=\tfrac{1}{4}\cos(\theta)+\tfrac{1}{2}, \; y(\theta)=\tfrac{1}{4}\sin(\theta)+\tfrac{1}{2}, \; \theta \in [0,2\pi)\},
		\ee
		$g=0$, and $f, g_{\Gamma}$ are chosen such that the exact solution, $u$, is
		\[
		u_{+} = a_{+}^{-1} ((x-\tfrac{1}{2})^2 + (y-\tfrac{1}{2})^2)^{3/2} + 2^{-6}\left(a_{-}^{-1} - a_{+}^{-1}\right)
		\quad
		\text{and}
		\quad
		u_{-} = a_{-}^{-1} ((x-\tfrac{1}{2})^2 + (y-\tfrac{1}{2})^2)^{3/2}.
		\]
		This makes $g_{\Gamma} = 0$ on $\Gamma$, and the Dirichlet boundary condition, $g_b$, nonzero on $\partial \Omega$.
		\begin{table}[htbp]
			\begin{center}
				 \resizebox{\textwidth}{!}{%
				\begin{tabular}{c | c c c c c c | c c c c c}
					\hline
					& \multicolumn{6}{c}{$\mathcal{B}^{S,H^1_0(\Omega)}_{3,J}$ (ours)} \vline & \multicolumn{5}{c}{$\mathcal{B}^{2D,H^{1}_0(\Omega)}_{3,J}$ (traditional) or $\Phi^{2D}_{J}$ (FEM)}\\
					\hline
					$J$ & $N_{J}$ & $\kappa$ & $\frac{\|u_{J}-u\|_2}{\|u\|_{2}}$ & order & $\frac{\|\nabla u_J - \nabla u\|_2}{\|\nabla u\|_{2}}$ & order & $N_{J}$ & $\frac{\|u_J - u\|_2}{\|u\|_{2}}$ & order & $\frac{\|\nabla u_J - \nabla u\|_2}{\|\nabla u\|_{2}}$& order \\
					\hline
					4 & 2345 & 5.83E+6  & 1.64E-1 &  & 3.98E-1 &  &  225 & 7.11E-1  &  & 8.97E-1  &   \\
					5 & 10401 & 8.94E+6  & 3.75E-2 & 1.98 & 2.01E-1 & 0.917 & 961  & 4.24E-1  & 0.711 & 7.14E-1  & 0.315 \\
					6 & 43449 & 1.24E+7 & 8.38E-3 & 2.10 & 1.00E-1 & 0.970 & 3969  &  2.30E-1 &  0.863 & 5.42E-1  & 0.388 \\
					7 & 177169 &1.56E+7   & 2.35E-3 & 1.81 & 5.16E-2 & 0.947 & 16129  &  1.06E-1 & 1.10 & 3.71E-1  & 0.539\\
					\hline
				\end{tabular}%
				}
			\end{center}
			\caption{Numerical results for \cref{ex:circle}. The estimated constant $C_w$ in \eqref{cond:number} is less than $16$.}
			\label{ex:tab:circle}
		\end{table}
		See \cref{ex:tab:circle} for numerical results, and \cref{fig:ex:circle} for plots. This example aims to show that the high contrast in the diffusion coefficient $a$ results in large condition numbers, but they are still uniformly bounded, which is consistent with our main result \cref{thm:converg}. Additionally, it demonstrates that even though both jump conditions are equal to zero (i.e., $g=g_\Gamma=0$), the solution still has low regularity. If we compare the degrees of freedom of $\mathcal{B}^{S,H^1_0(\Omega)}_{3,J}$ and $\mathcal{B}^{2D,H^{1}_0(\Omega)}_{3,J}$ (or equivalently $\Phi^{2D}_{J}$) at each scale level, we observe that the former is only a fixed constant multiple of the latter for all scale levels and this constant is independent of the scale level.
		\begin{figure}[htbp]
			 \begin{overpic}[width=0.25\textwidth]{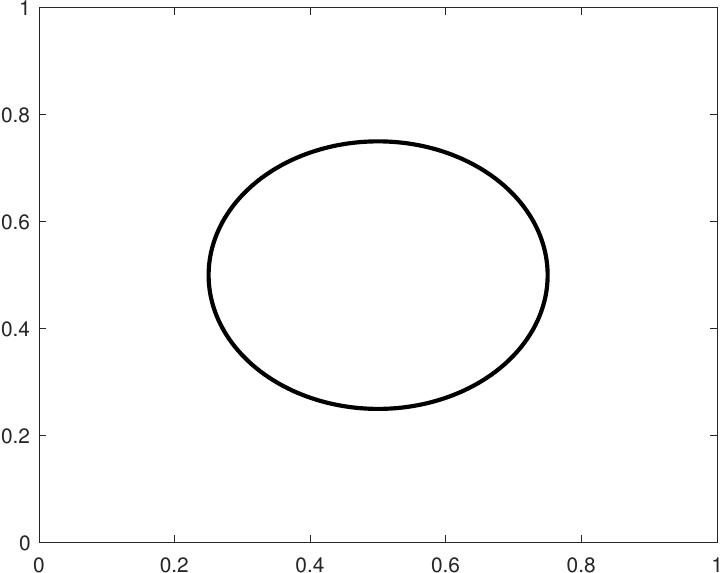}
			 	\put(38,68){$a_+ = 10^6 $}
			 	\put(40,38){$a_-=1$}
			 \end{overpic}
			 \includegraphics[width=0.3\textwidth]{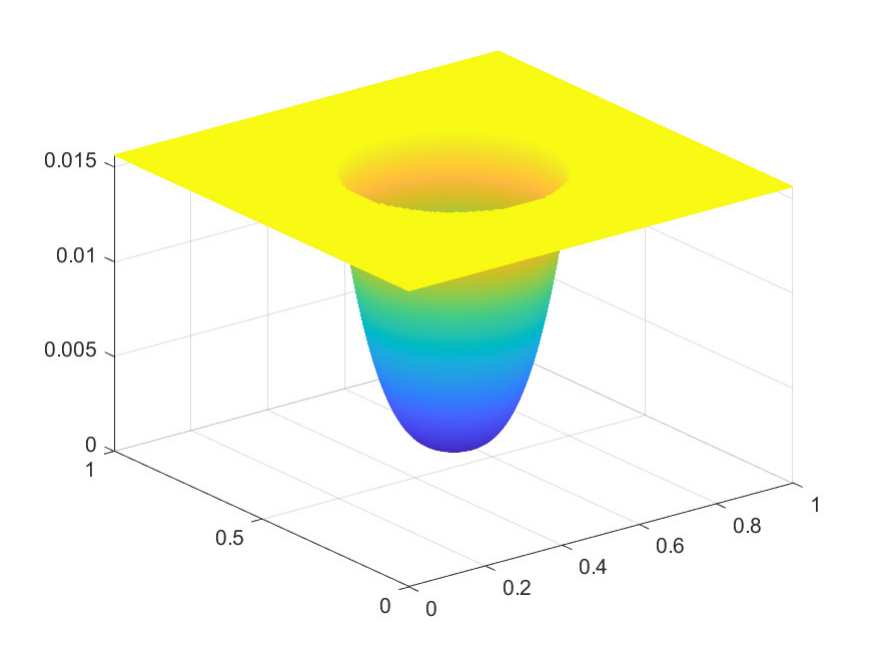}
			 \includegraphics[width=0.3\textwidth]{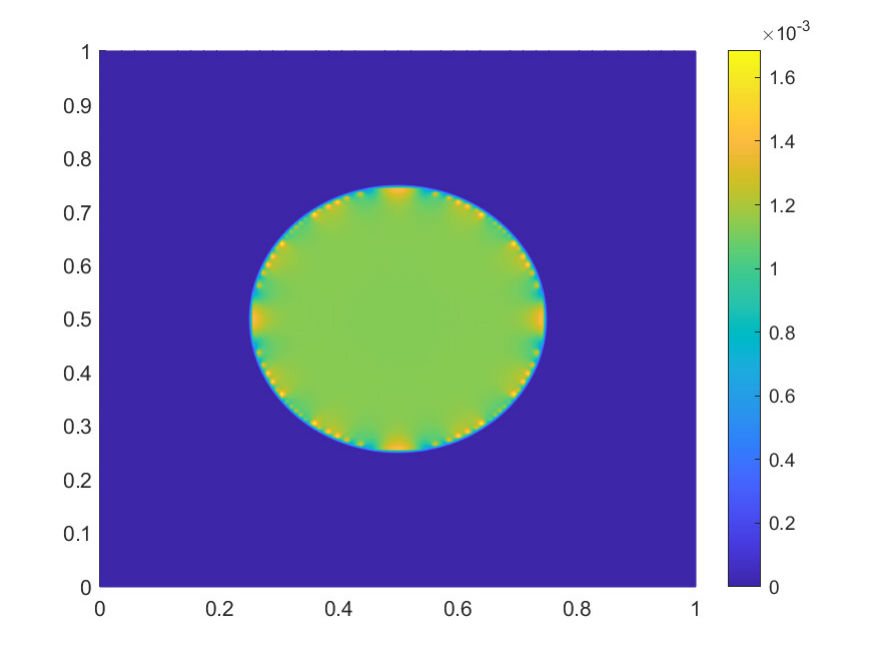}
			 \caption{\cref{ex:circle}. Left: the plot of $\Gamma$. Middle: the plot of the approximate solution at $J=7$, where $a_+ = 10^6$. Right: the plot of the error at $J=7$, where $a_+ = 10^6$.}
			\label{fig:ex:circle}
		\end{figure}
	\end{example}
	
	\begin{example} \label{ex:star:gonglili}
		\normalfont
		We apply our wavelet method to \cite[Example 2]{GLL08}, where we transform the original problem so that its domain is the unit square. More specifically, consider the general elliptic interface problem \eqref{model0}, where $a_+ \in \{10^2, 10^{-2}\}$, $a_- = (2x-1)^2 + (2y-1)^2 +1$,
		\[
		\Gamma = \{(x,y) \in \Omega : x(\theta) = \tfrac{1}{2}(\tfrac{1}{10} \sin(5\theta - \tfrac{\pi}{5}) + \tfrac{1}{2})\cos(\theta)+ \tfrac{1}{2},
		y(\theta) = \tfrac{1}{2}(\tfrac{1}{10} \sin(5\theta - \tfrac{\pi}{5}) + \tfrac{1}{2})\sin(\theta)+ \tfrac{1}{2}, \theta \in [0,2\pi)\},
		\]
		and $f, g, g_\Gamma$ are chosen such that the exact solution, $u$, is
		\[
			u_+ =  a_+^{-1} \left(\sin(2x-1) \cos(2y-1) + \log(\sqrt{(2x-1)^2 + (2y-1)^2})\right)
			\quad \text{and} \quad
			u_- = (2x-1)^2 + (2y-1)^2.
		\]
		This makes $g, g_\Gamma \neq 0$ on $\Gamma$ and the Dirichlet boundary condition, $g_b$, nonzero on $\partial \Omega$. Note that the exact solution $u$ is discontinuous across $\Gamma$.
See \cref{ex:tab:star:gonglili} for numerical results, and \cref{fig:ex:star:gonglili} for plots.
		\begin{table}[htbp]
			\begin{center}
				 \resizebox{\textwidth}{!}{%
					\begin{tabular}{c | c c c c c c | c c c c c}
						\hline
						 \multicolumn{12}{c}{$a_+ = 10^2$}\\
						\hline
						& \multicolumn{6}{c}{$\mathcal{B}^{S,H^1_0(\Omega)}_{3,J}$ (ours)} \vline & \multicolumn{5}{c}{$\mathcal{B}^{2D,H^{1}_0(\Omega)}_{3,J}$ (traditional) or $\Phi^{2D}_{J}$ (FEM)}\\
						\hline
						$J$ & $N_{J}$ & $\kappa$ & $\frac{\|u_J - u\|_2}{\|u\|_{2}}$ & order & $\frac{\|\nabla u_J - \nabla u\|_2}{\|\nabla u\|_{2}}$ & order & $N_{J}$ & $\frac{\|u_J - u\|_2}{\|u\|_{2}}$ & order & $\frac{\|\nabla u_J - \nabla u\|_2}{\|\nabla u\|_{2}}$& order \\
						\hline
						4 & 2847 & 3.48E+2 & 5.21E-2 &  & 2.59E-1 &  & 225 & 2.58E-1 &  & 5.76E-1 &  \\
						5 & 12539 & 4.36E+2 & 1.57E-2 & 1.63 & 1.46E-1 & 0.778 & 961 & 1.41E-1  & 0.834 & 4.04E-1 & 0.490\\
						6 & 52145 & 4.68E+2 & 8.28E-3 & 0.894 & 1.047E-1 & 0.468 & 3969 & 7.94E-2  & 0.805 & 3.01E-1 & 0.413\\
						7 & 212267 & 4.95E+2 & 9.47E-4 & 3.09 & 3.01E-2 & 1.78 & 16129 &  3.94E-2 & 1.00 & 2.07E-1 & 0.535\\
						\hline
						 \multicolumn{12}{c}{$a_+ = 10^{-2}$}\\
						\hline
						& \multicolumn{6}{c}{$\mathcal{B}^{S,H^1_0(\Omega)}_{3,J}$ (ours)} \vline & \multicolumn{5}{c}{$\mathcal{B}^{2D,H^{1}_0(\Omega)}_{3,J}$ (traditional) or $\Phi^{2D}_{J}$ (FEM)}\\
						\hline
						$J$ & $N_{J}$ & $\kappa$ & $\frac{\|u_J - u\|_2}{\|u\|_{2}}$ & order & $\frac{\|\nabla u_J - \nabla u\|_2}{\|\nabla u\|_{2}}$ & order & $N_{J}$ & $\frac{\|u_J - u\|_2}{\|u\|_{2}}$ & order & $\frac{\|\nabla u_J - \nabla u\|_2}{\|\nabla u\|_{2}}$& order \\
						\hline
						4 & 2847 & 3.25E+3 & 4.09E-2 &  & 2.06E-1 &  & 225 & 2.05E-1  &  &  5.12E-1 &  \\
						5 & 12539 & 3.62E+3 & 7.89E-3 & 2.24 & 1.21E-1 & 0.729 & 961 & 1.07E-1  & 0.893 & 3.85E-1  & 0.392\\
						6 & 52145 & 3.80E+3 & 3.46E-3 & 1.16 & 7.87E-2 & 0.598 & 3969 & 5.25E-2  & 1.01 & 2.61E-1  & 0.548\\
						7 & 212267 & 3.82E+3 & 5.40E-4 & 2.65 & 2.75E-2 & 1.50 & 16129 &  2.62E-2 & 0.990 & 1.87E-1  & 0.477\\
						\hline
					\end{tabular}%
				}
			\end{center}
			\caption{Numerical results for \cref{ex:star:gonglili}. The estimated constant $C_w$ in \eqref{cond:number} is less than 6 for $a_+ = 10^2$ and is less than 14 for $a_+=10^{-2}$. As indicated in \cref{fig:ex:star:gonglili}, the flower-shaped interface $\Gamma$ has relatively large curvatures and the discontinuous approximate solution $u_h$ has large jumps across $\Gamma$.}
			 \label{ex:tab:star:gonglili}
		\end{table}
		\begin{figure}[htbp]
			 \begin{overpic}[width=0.25\textwidth]{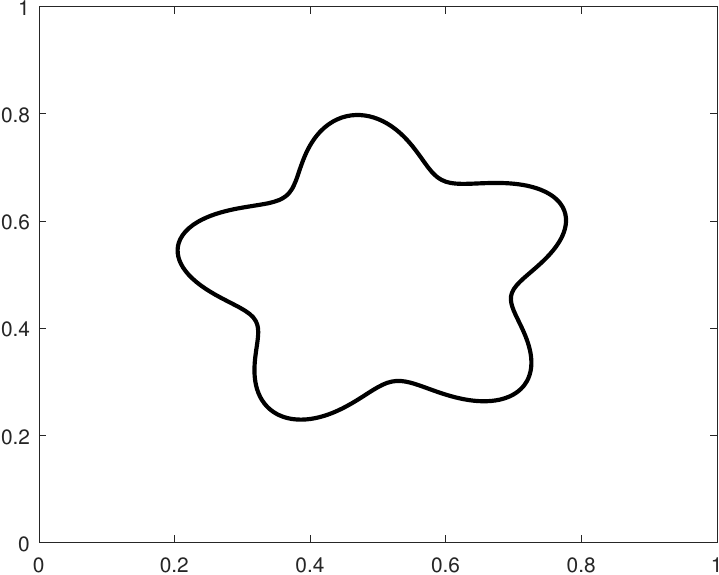}
				\put(23,68){$a_+ \in \{10^2,10^{-2}\}$}
				\put(47,40){$a_-$}
			\end{overpic}
			 \includegraphics[width=0.3\textwidth]{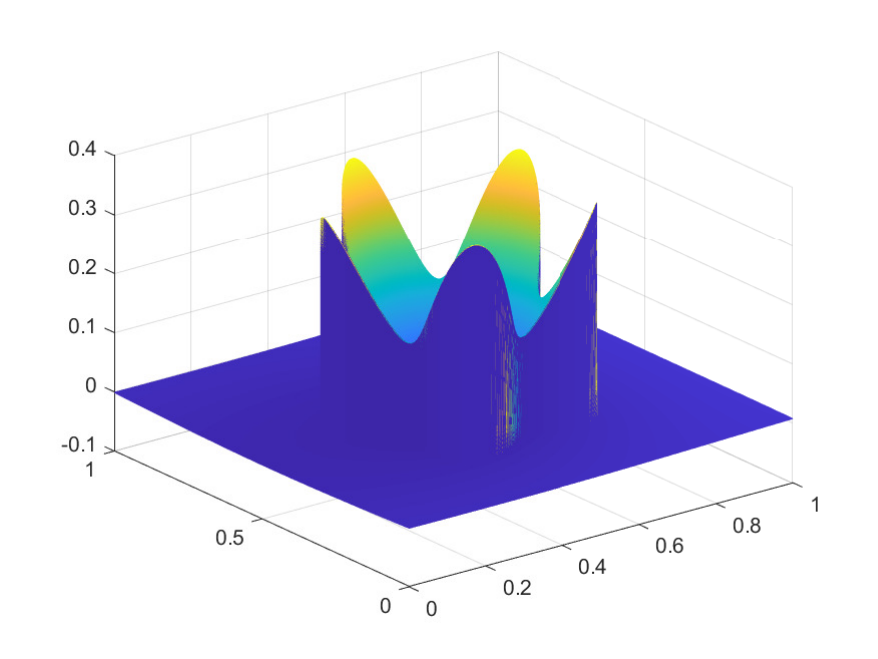}
			 \includegraphics[width=0.3\textwidth]{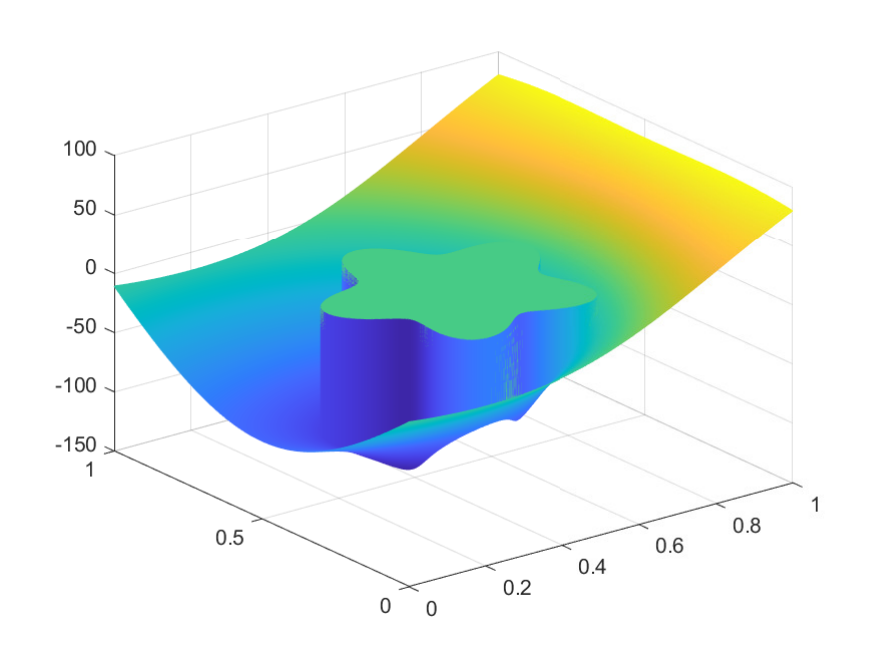}
			 \caption{\cref{ex:star:gonglili}. Left: the plot of $\Gamma$. Middle: the plot of the approximate solution at $J=7$ with $a_+ = 10^2$. Right: the plot of the approximate solution at $J=7$ with $a_+ = 10^{-2} $.}
			 \label{fig:ex:star:gonglili}
		\end{figure}
	\end{example}
	
	\begin{example} \label{ex:flower5}
		\normalfont
		Consider the general elliptic interface model problem \eqref{model0}, 
		where $a_+ = 2 + \sin(5(x-1/2)) \sin(5(y-1/2))$, $a_- = 10^{3} a_+$,
		\[
		\Gamma = \{(x,y) \in \Omega \; : \; x(\theta) = (\tfrac{1}{5} + \tfrac{2}{25}\sin(5 \theta)) \cos(\theta) + \tfrac{1}{2}, \; y(\theta) = (\tfrac{1}{5} + \tfrac{2}{25}\sin(5 \theta)) \sin(\theta) + \tfrac{1}{2},
		\; \theta \in [0,2\pi)\},
		\]
		and $f, g, g_\Gamma$ are chosen such that the exact solution, $u$, is
		\begin{align*}
		u_+ & = \sin(10x-5) \sin(10y-5)\left((x-\tfrac{1}{2})^2 + (y-\tfrac{1}{2})^2 - (\tfrac{1}{5} + \tfrac{2}{25}\sin(5 \Theta(x,y)))^2\right) + 1,\\
		u_- & = 10^{-3} \sin(10x-5) \sin(10y-5)\left((x-\tfrac{1}{2})^2 + (y-\tfrac{1}{2})^2 - (\tfrac{1}{5} + \tfrac{2}{25}\sin(5 \Theta(x,y)))^2\right) + 31,
		\end{align*}
		where $\Theta$ is defined as in \eqref{angle} for $x,y \in \Omega$ and $\Theta(1/2,1/2):=0$. This makes $g \neq 0$, $g_{\Gamma}=0$ on $\Gamma$ (note that the first jump condition is not zero but the second jump condition vanishes in the general elliptic interface problem \eqref{model0}), and the Dirichlet boundary condition, $g_b$, nonzero on $\partial \Omega$. Note that the exact solution $u$ is discontinuous across $\Gamma$.
See \cref{ex:tab:flower5} for numerical results and \cref{ex:fig:flower5} for plots.
		\begin{table}[htbp]
			\begin{center}
				 \resizebox{\textwidth}{!}{%
					\begin{tabular}{c | c c c c c c | c c c  c c}
						\hline
						& \multicolumn{6}{c}{$\mathcal{B}^{S,H^1_0(\Omega)}_{3,J}$ (ours)} \vline & \multicolumn{5}{c}{$\mathcal{B}^{2D,H^{1}_0(\Omega)}_{3,J}$ (traditional) or $\Phi^{2D}_{J}$ (FEM)}\\
						\hline
						$J$ & $N_{J}$ & $\kappa$ & $\frac{\|u_J - u\|_2}{\|u\|_{2}}$ & order & $\frac{\|\nabla u_J - \nabla u\|_2}{\|\nabla u\|_{2}}$ & order & $N_{J}$ & $\frac{\|u_J - u\|_2}{\|u\|_{2}}$ & order & $\frac{\|\nabla u_J - \nabla u\|_2}{\|\nabla u\|_{2}}$& order \\
						\hline
						4 & 2833 & 3.07E+4 & 3.85E-2 &  & 2.03E-1 &  & 225 & 5.68E-2 &  & 2.34E-1 & \\
						5 & 13589 & 4.04E+4 & 1.04E-2 & 1.68 & 1.08E-1 & 0.810 & 961 & 2.50E-2 & 1.13 & 1.26E-1 & 0.857 \\
						6 & 57317 & 4.50E+4 & 2.70E-3 & 1.87 & 5.55E-2 & 0.929 & 3969 &  1.38E-2 & 0.841 & 7.35E-2 & 0.755 \\
						7 & 233583 & 4.54E+4 & 6.88E-4 & 1.95 & 2.75E-2 & 1.00 & 16129 &  7.68E-3 & 0.833 & 4.58E-2 & 0.677\\
						\hline
					\end{tabular}%
				}
			\end{center}
			\caption{Numerical results for \cref{ex:flower5}. The estimated constant $C_w$ in \eqref{cond:number} is less than $16$.}
			\label{ex:tab:flower5}
		\end{table}
		\begin{figure}[htbp]
			 \begin{overpic}[width=0.25\textwidth]{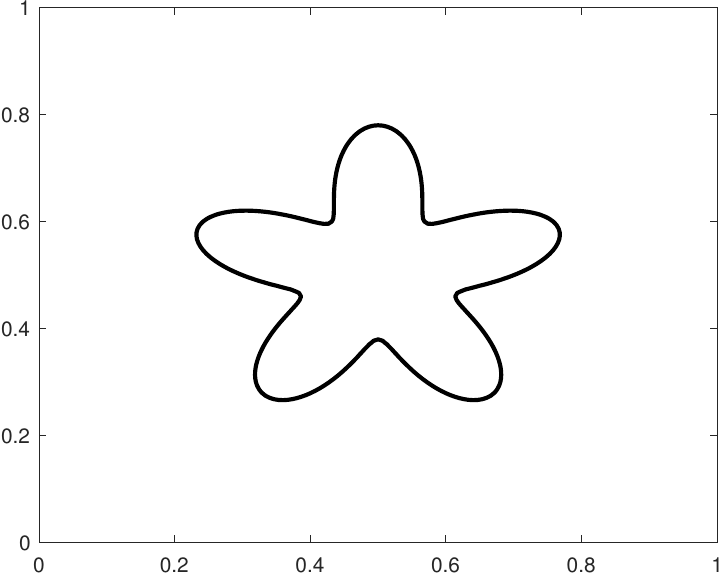}
			 	\put(34,42){\tiny $a_-=10^{3}a_+$}
			 	\put(10,73){\tiny $a_+ = $}
			 	\put(10,66){\tiny $ 2+\sin(5(x-\tfrac{1}{2}))\sin(5(y-\tfrac{1}{2}))$}
			 \end{overpic}
			 \includegraphics[width=0.3\textwidth]{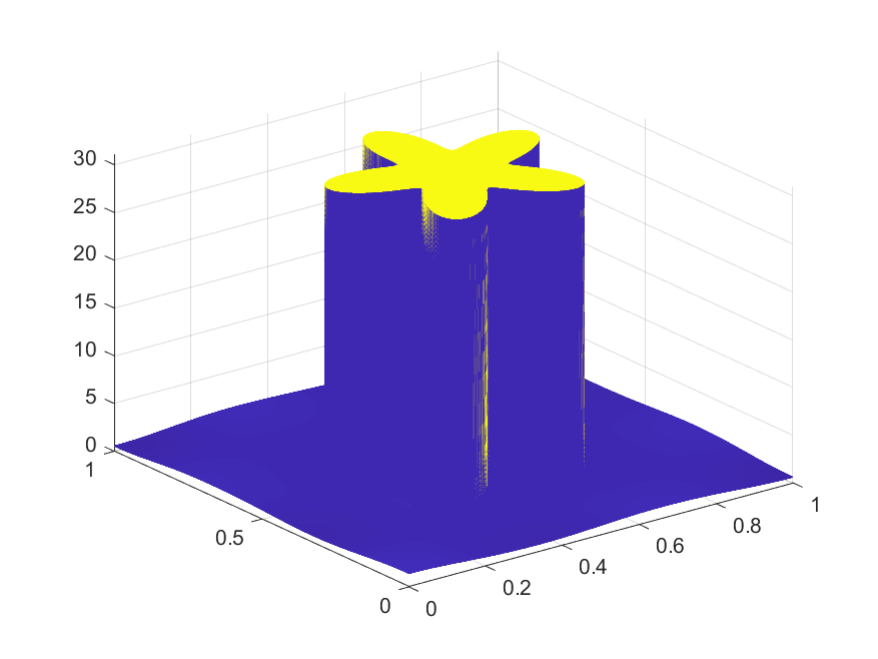}
			 \includegraphics[width=0.3\textwidth]{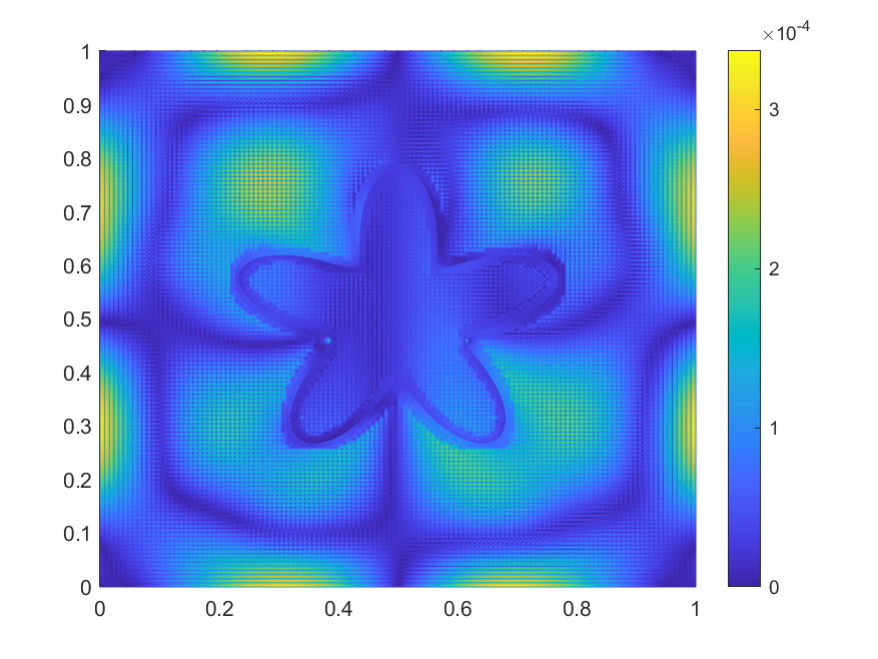}
			 \caption{\cref{ex:flower5}. Left: the plot of $\Gamma$. Middle: the plot of the approximate solution at $J=7$. Right: the plot of the error at $J=7$.}
			\label{ex:fig:flower5}
		\end{figure}
	\end{example}
	
	\begin{example} \label{ex:flower8}
		\normalfont
		Consider the general elliptic interface problem \eqref{model0}, where $a_+ = 1$, $a_-=10^{-3}$,
		\[
		\Gamma = \{(x,y) \in \Omega \; : \; x(\theta) = \tfrac{1}{4}(\tfrac{\pi}{3} + \tfrac{2}{5}\sin(8 \theta))\cos(\theta) + \tfrac{1}{2}, \; y(\theta) = \tfrac{1}{4}(\tfrac{\pi}{3} + \tfrac{2}{5}\sin(8 \theta))\sin(\theta) + \tfrac{1}{2}, \; \theta \in [0,2\pi)\},
		\]
		and $f,g,g_{\Gamma}$ are chosen such that the exact solution, $u$, is
		\[
		u_{+}=\cos(4x-2), \quad u_{-}=10^{3}\sin(4y-2)+1500.
		\]
		This makes $g, g_{\Gamma} \neq 0$ on $\Gamma$ (and hence both jump conditions do not vanish), and the Dirichlet boundary condition, $g_b$, nonzero on $\partial \Omega$. Note that the exact solution $u$ is discontinuous across $\Gamma$.
See \cref{ex:tab:flower8} for numerical results and \cref{ex:fig:flower8} for plots.
		\begin{table}[htbp]
			\begin{center}
				 \resizebox{\textwidth}{!}{%
					\begin{tabular}{c | c c c c c c | c c c  c c}
						\hline
						& \multicolumn{6}{c}{$\mathcal{B}^{S,H^1_0(\Omega)}_{3,J}$ (ours)} \vline & \multicolumn{5}{c}{$\mathcal{B}^{2D,H^{1}_0(\Omega)}_{3,J}$ (traditional) or $\Phi^{2D}_{J}$ (FEM)}\\
						\hline
						$J$ & $N_{J}$ & $\kappa$ & $\frac{\|u - u_{J}\|_2}{\|u\|_{2}}$ & order & $\frac{\|\nabla u - \nabla u_{J}\|_2}{\|\nabla u\|_{2}}$ & order & $N_{J}$ & $\frac{\|u - u_{J}\|_2}{\|u\|_{2}}$ & order & $\frac{\|\nabla u - \nabla u_{J}\|_2}{\|\nabla u\|_{2}}$& order \\
						\hline
						4 & 4585 & 5.33E+3 & 1.27E-1 &      & 3.85E-1 &  & 225 & 2.47E-1 &  & 7.19E-1 & \\
						5 & 22857 & 8.23E+3 & 3.01E-2 & 1.79 & 1.98E-1 & 0.828 & 961 & 1.82E-1 & 0.417 & 5.30E-1 & 0.420\\
						6 & 97497 & 9.71E+3 & 7.79E-3 & 1.86 & 9.27E-2 & 1.05 & 3969 & 9.62E-2 & 0.900 & 3.51E-1 & 0.580 \\
						7 & 398713 & 1.08E+4 & 1.58E-3 & 2.26 & 4.71E-2 & 0.961 & 16129 & 5.24E-2 & 0.866 & 2.44E-1 & 0.519\\
						\hline
					\end{tabular}%
				}
				\caption{Numerical results for \cref{ex:flower8}. The estimated constant $C_w$ in \eqref{cond:number} is less than $11$.}
				 \label{ex:tab:flower8}
			\end{center}
		\end{table}
		\begin{figure}[htbp]
			 \begin{overpic}[width=0.25\textwidth]{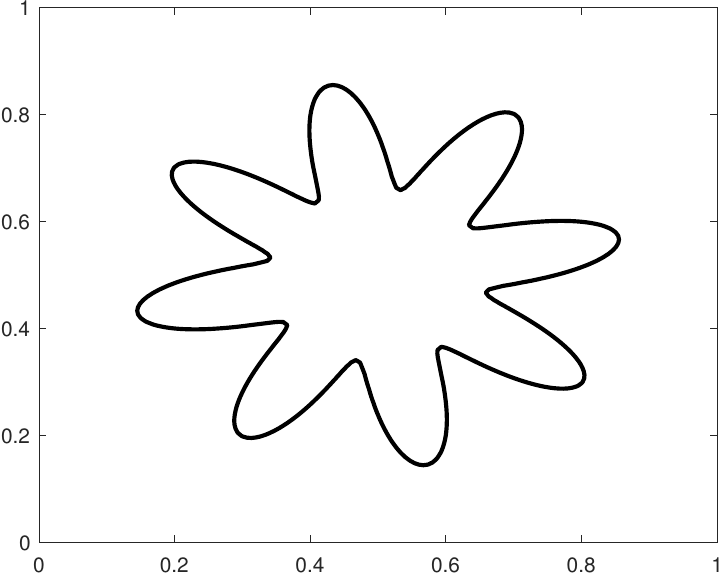}
			 	\put(40,40){\tiny $a_-=10^{-3}$}
			 	\put(40,71){$a_+=1$}
			 \end{overpic}
			 \includegraphics[width=0.3\textwidth]{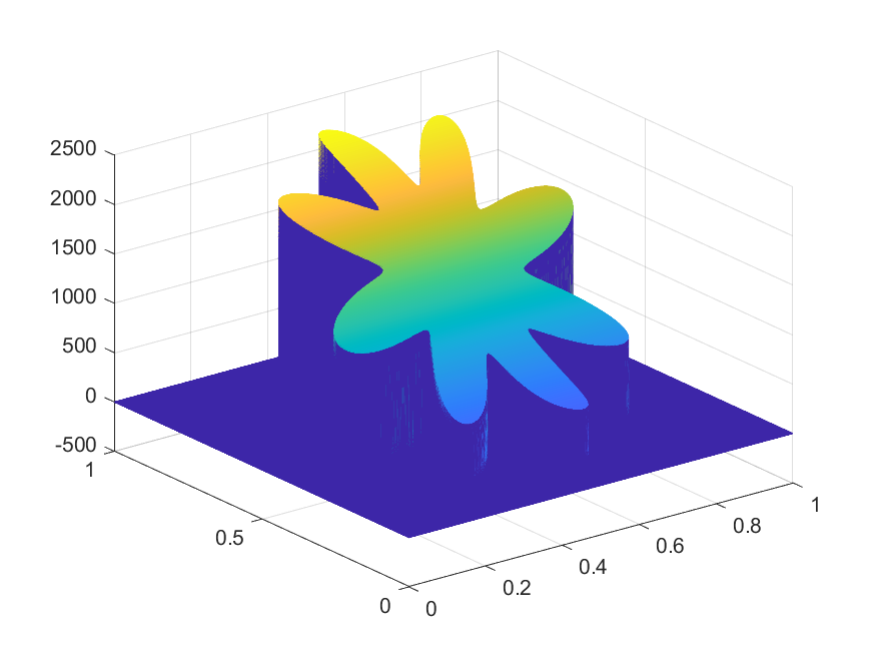}
			 \includegraphics[width=0.3\textwidth]{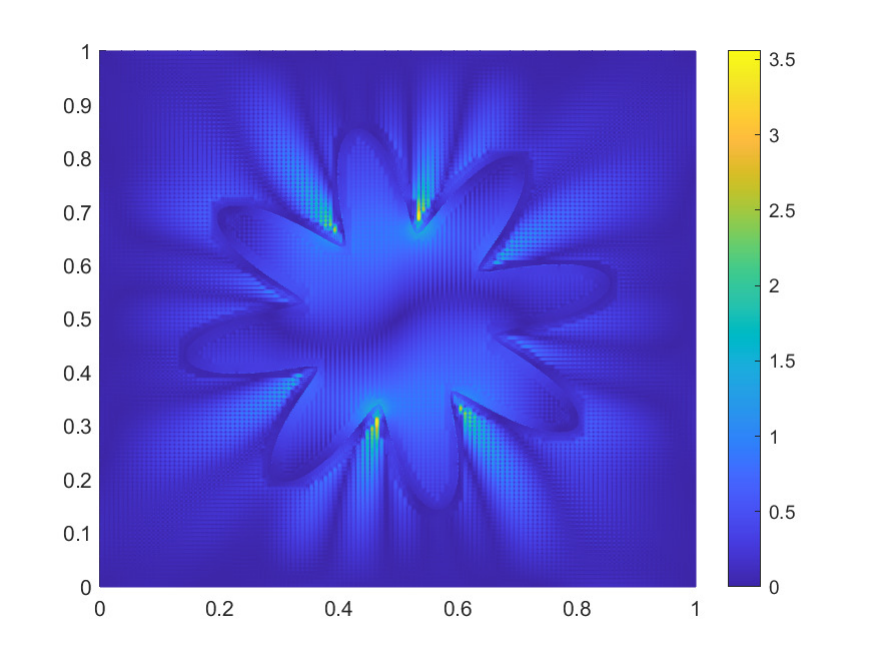}
			 \caption{\cref{ex:flower8}. Left: the plot of $\Gamma$. Middle: the plot of the approximate solution at $J=7$. Right: the plot of the error at $J=7$.}
			\label{ex:fig:flower8}
		\end{figure}
	\end{example}
	
	\begin{example} \label{ex:flower6}
		\normalfont
		Consider the model problem \eqref{model}, where $a_+ = 10^{4}$, $a_-=1$,
		\[
		\begin{aligned}
		\Gamma = \{(x,y) \in \Omega \; :\; & x(\theta)=10^{-1/2}(1 + \tfrac{2}{5}\sin(6 \theta))^{-1/4}\cos(\theta)+\tfrac{1}{2}, \\
		& y(\theta)=10^{-1/2}(1 + \tfrac{2}{5}\sin(6 \theta))^{-1/4} \sin(\theta) +\tfrac{1}{2}, \; \theta \in [0,2\pi)\},
		\end{aligned}
		\]
		and $f,g_{\Gamma}$ are chosen such that the exact solution, $u$, is
		\[
		u_{+} = a_{+}^{-1}(((x-\tfrac{1}{2})^2 + (y-\tfrac{1}{2})^2)^2 (1 + \tfrac{2}{5}\sin(6 \Theta(x,y))) - 10^{-2}), \quad
		u_{-} = a_{-}^{-1} a_{+} u_{+},
		\]
		where $\Theta$ is defined as in \eqref{angle} for $x,y \in \Omega$ and $\Theta(1/2,1/2):=0$. This makes $g = g_{\Gamma} = 0$ on $\Gamma$, and the Dirichlet boundary condition, $g_b$, nonzero on $\partial \Omega$. This is another example, which once again demonstrates that even though both jump conditions are equal to zero (i.e., $g=g_\Gamma=0$), the solution still has low regularity.
See \cref{ex:tab:flower6} for numerical results and \cref{ex:fig:flower6} for plots.
		\begin{table}[htbp]
			\begin{center}
				 \resizebox{\textwidth}{!}{%
					\begin{tabular}{c | c c c c c c | c c c  c c}
						\hline
						& \multicolumn{6}{c}{$\mathcal{B}^{S,H^1_0(\Omega)}_{3,J}$ (ours)} \vline & \multicolumn{5}{c}{$\mathcal{B}^{2D,H^{1}_0(\Omega)}_{3,J}$ (traditional) or $\Phi^{2D}_{J}$ (FEM)}\\
						\hline
						$J$ & $N_{J}$ & $\kappa$ & $\frac{\|u_J - u\|_2}{\|u\|_{2}}$ & order & $\frac{\|\nabla u_J - \nabla u\|_2}{\|\nabla u\|_{2}}$ & order & $N_{J}$ & $\frac{\|u_J - u\|_2}{\|u\|_{2}}$ & order & $\frac{\|\nabla u_J - \nabla u\|_2}{\|\nabla u\|_{2}}$& order \\
						\hline
						4 & 3401 & 7.26E+4 & 1.19E-1 &   & 3.77E-1 &  & 225 & 6.02E-1 &  & 8.37E-1 & \\
						5 & 14361 & 9.47E+4 & 2.87E-2 & 1.96 & 1.97E-1 & 0.901 & 961 & 3.90E-1 & 0.598 & 7.10E-1 & 0.228 \\
						6 & 59361 & 1.21E+5 & 7.92E-3 & 1.82  & 1.04E-1 & 0.903 & 3969 & 2.04E-1 & 0.914 & 5.16E-1 & 0.449 \\
						7 & 241409 & 1.36E+5 & 2.08E-3 & 1.91  & 4.97E-2 & 1.04 & 16129 & 9.68E-2 & 1.06 & 3.60E-1 & 0.515\\
						\hline
					\end{tabular}%
				}
			\end{center}
			\caption{Numerical results for \cref{ex:flower6}. The estimated constant $C_w$ in \eqref{cond:number} is less than $14$.}
			\label{ex:tab:flower6}
		\end{table}
		\begin{figure}[htbp]
			 \begin{overpic}[width=0.25\textwidth]{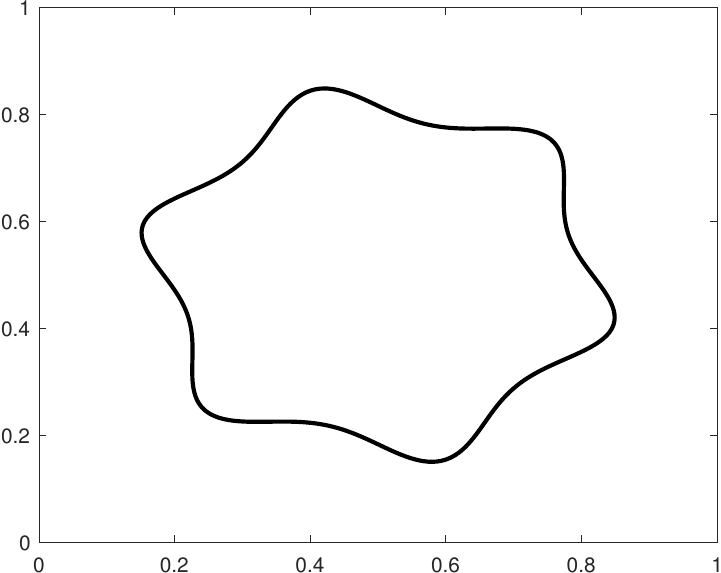}
			 	\put(40,40){$a_-=1$}
			 	 \put(35,70){$a_+=10^4$}
			 \end{overpic}
			 \includegraphics[width=0.3\textwidth]{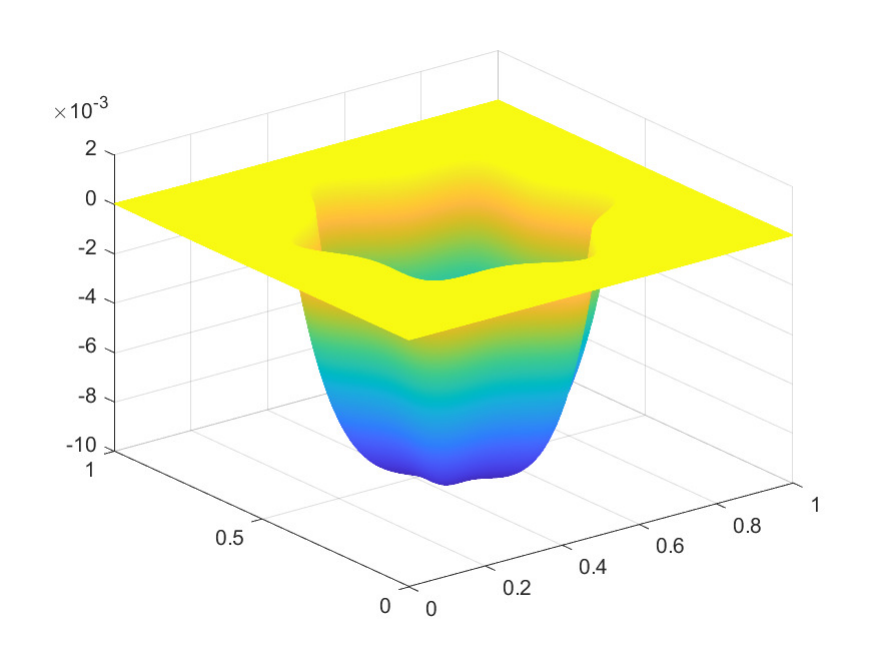}
			 \includegraphics[width=0.3\textwidth]{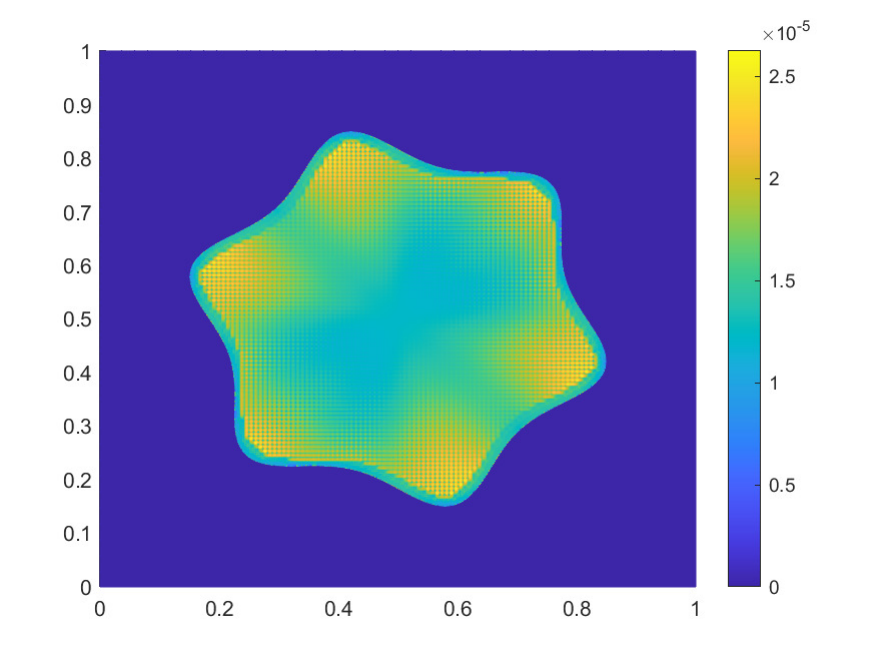}
			 \caption{\cref{ex:flower6}. Left: the plot of $\Gamma$. Middle: the plot of the approximate solution at $J=7$. Right: the plot of the error at $J=7$.}
			\label{ex:fig:flower6}
		\end{figure}
	\end{example}
	
	\subsection{Examples with unknown exact solutions $u$} We present two examples, where the exact solutions $u$ are unknown. Recall that the reference solutions $u^{\text{ref}}$ will be computed using $\mathcal{B}^{S,H^1_0(\Omega)}_{3,8}$.

	\begin{example} \label{ex:circle2}
		\normalfont
		Consider the model problem \eqref{model}, where $a_{+}=1$, $a_{-}=10^4$, $\Gamma$ is defined as in \eqref{gamma:circle}, $f=-16$, and $g=g_{\Gamma}=g_b=0$ (and hence both jump conditions vanish). The exact solution, $u$, is unknown. Here, we observe that even though both jump conditions are equal to zero (i.e., $g=g_\Gamma=0$), the solution still has low regularity. See
		 \cref{ex:tab:circle2,ex:tab:gmres:circle}
		for numerical results and \cref{ex:fig:circle2} for plots.
		\cref{ex:tab:gmres:circle} demonstrates that the number of GMRES iterations required to reach the tolerance level $10^{-8}$ is smaller compared to the standard FEM case and is uniformly bounded irrespective of the matrix size. This is due to the fact that the wavelet coefficient matrices have small condition numbers that are uniformly bounded. On the other hand, in the case of standard FEM, the number of GMRES iterations required to reach a tolerance level of $10^{-8}$ doubles with each increase in the scale level.
		\begin{table}[htbp]
			\begin{center}
				 \resizebox{\textwidth}{!}{%
					\begin{tabular}{c | c c c c c c | c c c  c c}
						\hline
						& \multicolumn{6}{c}{$\mathcal{B}^{S,H^1_0(\Omega)}_{3,J}$ (ours)} \vline & \multicolumn{5}{c}{$\mathcal{B}^{2D,H^{1}_0(\Omega)}_{3,J}$ (traditional) or $\Phi^{2D}_{J}$ (FEM)}\\
						\hline
						$J$ & $N_{J}$ & $\kappa$ & $\|u_J-u^{\text{ref}}\|_2$ & order & $\|\nabla u_J-\nabla u^{\text{ref}}\|_2$ & order & $N_{J}$ & $\|u_J-u^{\text{ref}}\|_2$ & order & $\|\nabla u_J-\nabla u^{\text{ref}}\|_2$& order \\
						\hline
						4 & 2345 & 4.81E+5 & 5.37E-3 &   & 2.55E-1 &  & 225 & 3.13E-2 &  & 5.09E-1 & \\
						5 & 10401 & 5.35E+5 & 1.39E-3 & 1.81  & 1.31E-1 & 0.891 & 961 & 1.66E-2 & 0.869 & 3.48E-1 & 5.24E-1\\
						6 & 43449 & 5.79E+5 & 3.41E-4 & 1.96 & 6.33E-2 & 1.02 & 3969 & 9.06E-3 & 0.856 & 2.52E-1 & 4.53E-1\\
						\hline
					\end{tabular}%
				}
			\end{center}
			\caption{Numerical results for \cref{ex:circle2}. The estimated constant $C_w$ in \eqref{cond:number} is less than $7$.}
			\label{ex:tab:circle2}
		\end{table}
		\begin{table}[htbp]
			\begin{tabular}{c | c c c c | c c c c }
				\hline
				& \multicolumn{4}{c}{$\mathcal{B}^{S, H^1_0(\Omega)}_{3,J}$ (ours)} \vline & \multicolumn{4}{c}{$\Phi^{2D}_{J}$ (FEM)}\\
				\hline
				$J$ & 4 & 5 & 6 & 7 & 6 & 7 & 8 & 9 \\
				\hline
				$N_J$ & 2345 & 10401 & 43449 & 177169 & 3969 & 16129 & 65025 & 261121 \\
				$\#$ of iterations & 779 & 1433 & 2036 & 2512 & 877 & 2316 & 6235 & 16291 \\
				\hline
			\end{tabular}
			\caption{The number of GMRES iterations required to reach the tolerance level of $10^{-8}$ for \cref{ex:circle2} with $a_+ = 1$ and $a_- = 10^4$.}
			 \label{ex:tab:gmres:circle}
		\end{table}
		\begin{figure}[htbp]
			 \begin{overpic}[width=0.25\textwidth]{circle_gamma.pdf}
			 	\put(38,68){$a_+ = 1$}
			 	\put(36,38){$a_- = 10^4$}
			 \end{overpic}
			 \includegraphics[width=0.3\textwidth]{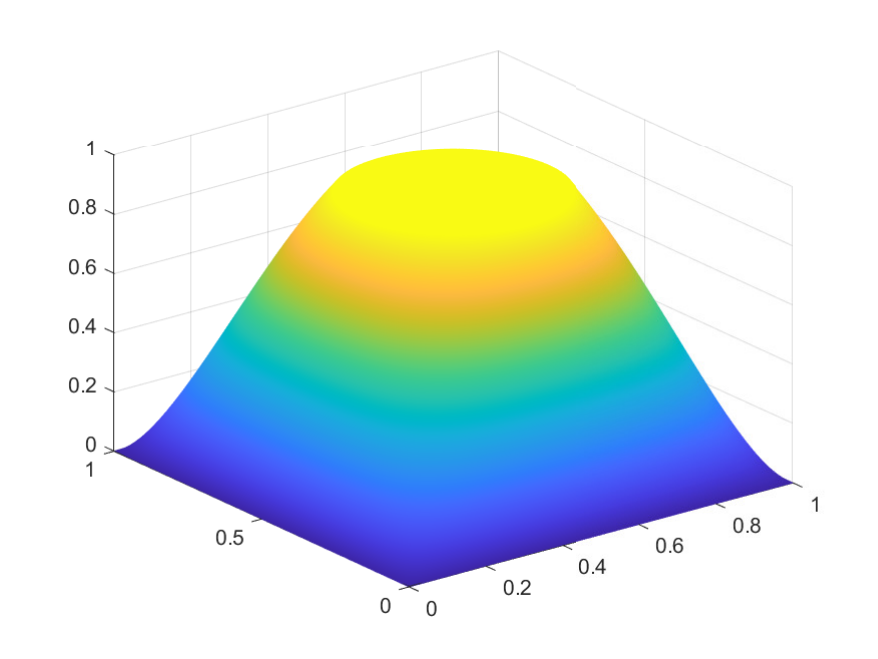}
			 \includegraphics[width=0.3\textwidth]{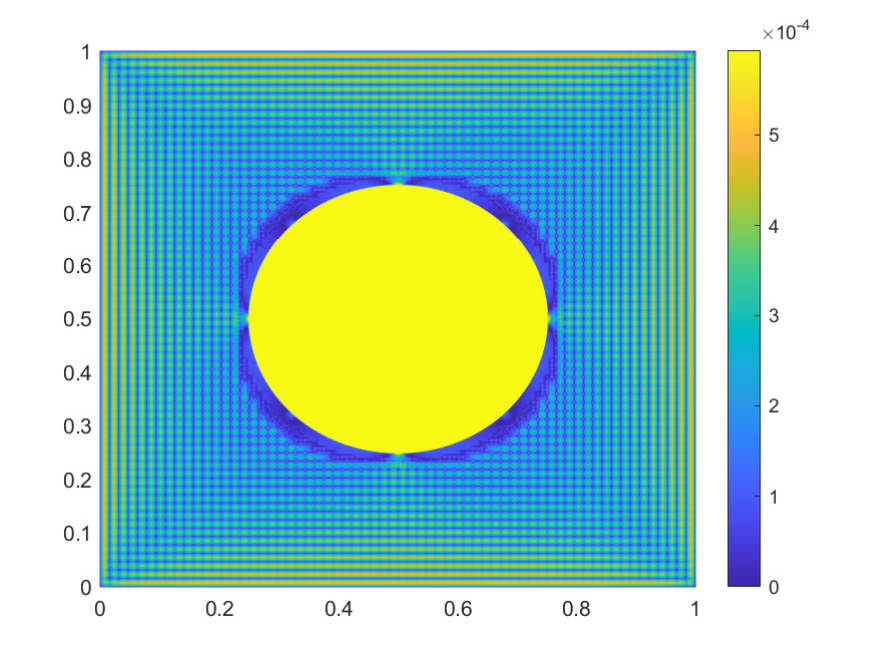}
			 \caption{\cref{ex:circle2}. Left: the plot of $\Gamma$. Middle: the plot of $u^{\text{ref}}$, which is the reference solution formed by $\mathcal{B}^{S,H^1_0(\Omega)}_{3,8}$.
				Right: the plot of the error $|u_6 - u^{\text{ref}}|$.}
			\label{ex:fig:circle2}
		\end{figure}
	\end{example}
	
	\begin{example} \label{ex:flower3}
		\normalfont
		Consider the general elliptic interface problem \eqref{model0}, where $a_+ = 10^3(2+\cos(4x-2) \cos(4y-2))$ and $a_- = 2+\cos(4x-2) \cos(4y-2)$,
		\[
		\begin{aligned}
		\Gamma = \{(x,y) \in \Omega \; : \; & x(\theta) = \tfrac{1}{2}(\tfrac{1}{2} + \tfrac{1}{4} \sin(3\theta))\cos(\theta)+\tfrac{1}{2},\\
		& y(\theta) = \tfrac{1}{2}(\tfrac{1}{2} + \tfrac{1}{4} \sin(3\theta))\sin(\theta)+\tfrac{1}{2}, \; \theta \in [0,2\pi) \},
		\end{aligned}
		\]
		$f_+ = -16 \sin(\pi (4x-2)) \sin(\pi (4y-2))$, $f_- = -16 \cos(\pi (4x-2)) \cos(\pi (4y-2))$, $g_b =0$, and $g = -\sin(\theta)-1$, $g_{\Gamma} = \cos(\theta)$ for $\theta \in [0,2\pi)$. Note that both jump conditions do not vanish.
The exact solution $u$ is unknown and is discontinuous across the interface $\Gamma$ due to nonzero $g$ for the first jump condition. See \cref{ex:tab:flower3} for numerical results and \cref{ex:fig:flower3} for plots.
		\begin{table}[htbp]
			\begin{center}
				 \resizebox{\textwidth}{!}{%
					\begin{tabular}{c | c c c c c c | c c c  c c}
						\hline
						& \multicolumn{6}{c}{$\mathcal{B}^{S,H^1_0(\Omega)}_{3,J}$ (ours)} \vline & \multicolumn{5}{c}{$\mathcal{B}^{2D,H^{1}_0(\Omega)}_{3,J}$ (traditional) or $\Phi^{2D}_{J}$ (FEM)}\\
						\hline
						$J$ & $N_{J}$ & $\kappa$ & $\|u_J-u^{\text{ref}}\|_2$ & order & $\|\nabla u_J-\nabla u^{\text{ref}}\|_2$ & order & $N_{J}$ & $\|u_J-u^{\text{ref}} \|_2$ & order & $\|\nabla u_J-\nabla u^{\text{ref}}\|_2$& order \\
						\hline
						4 & 3383 & 5.86E+3 & 6.51E-3 &   & 3.80E-1 &  & 225 & 4.63E-2 &  & 1.00 & \\
						5 & 14775 & 7.63E+3 & 1.77E-3 & 1.76  & 1.94E-1 & 0.908 & 961 & 2.27E-2 & 0.982 & 6.80E-1 & 5.42E-1 \\
						6 & 61277 & 8.65E+3 & 4.50E-4 &  1.92 & 9.56E-2 & 0.997 & 3969 & 1.31E-2 & 0.778 & 5.37E-1 & 3.33E-1 \\
						\hline
					\end{tabular}%
				}
			\end{center}
			\caption{Numerical results for \cref{ex:flower3}. The estimated constant $C_w$ in \eqref{cond:number} is less than $3$.}
			\label{ex:tab:flower3}
		\end{table}
		\begin{figure}[htbp]
			 \begin{overpic}[width=0.25\textwidth]{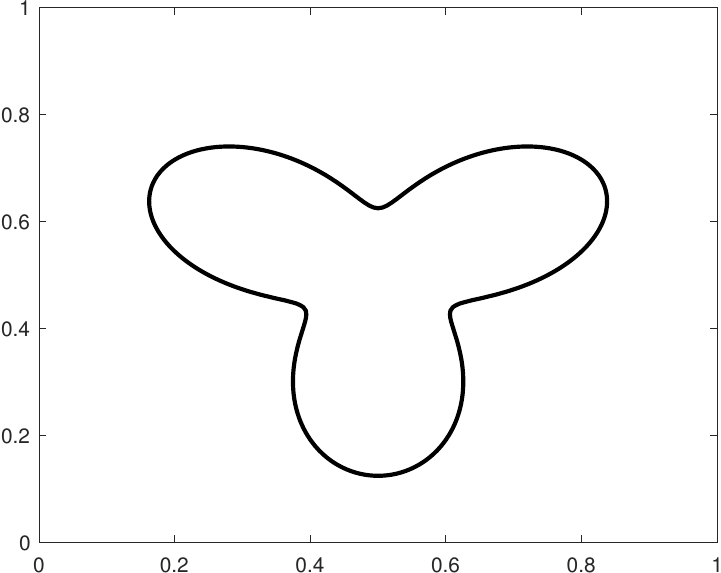}
			 	\put(47,40){$a_-$}
			 	\put(47,70){$a_+$}
			 \end{overpic}
			 \includegraphics[width=0.3\textwidth]{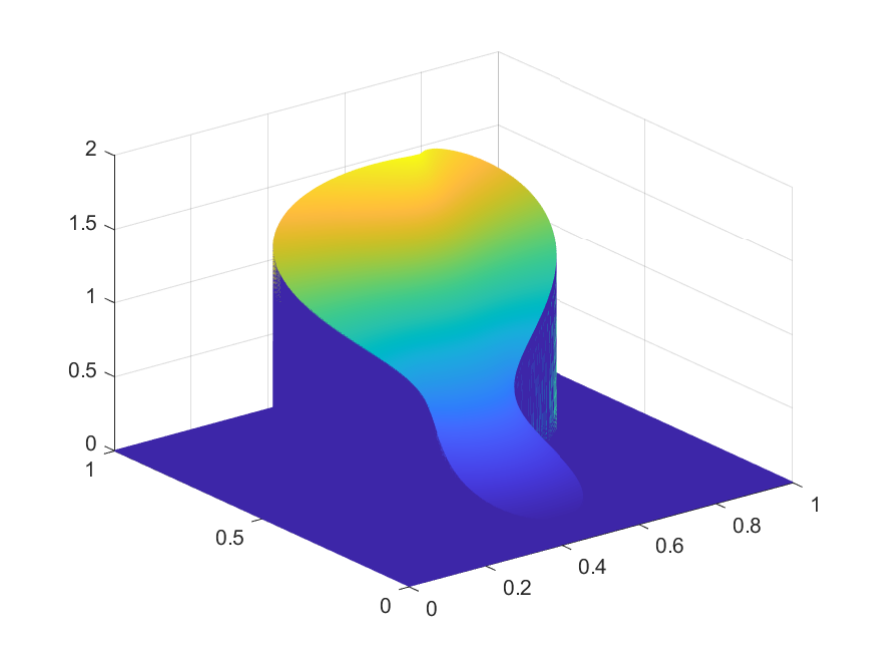}
			 \includegraphics[width=0.3\textwidth]{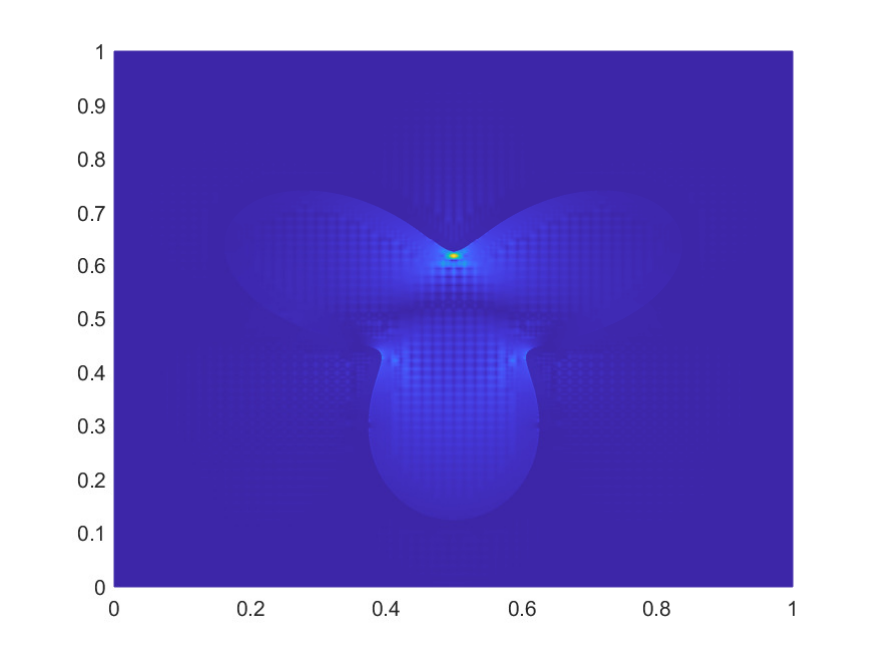}
			 \caption{\cref{ex:flower3}. Left: the plot of $\Gamma$. Middle: the plot of the reference solution $u^{\text{ref}}$, which is formed by $\mathcal{B}^{S,H^1_0(\Omega)}_{3,8}$.
				Right: the plot of the error $|u_6 - u^{\text{ref}}|$.}
			\label{ex:fig:flower3}
		\end{figure}
	\end{example}

\section{Proof of \cref{thm:converg} and justification of \cref{remark:quad}}
\label{sec:proof}
		
In this section, we shall first prove \cref{thm:converg} for the convergence rate of the Galerkin scheme using the biorthogonal wavelet on the unit interval $(0,1)$ described in \cref{sec:method}. Towards the end, we shall mathematically justify \cref{remark:quad}. 



To prove \cref{thm:converg}, we need three auxiliary results.
The first auxiliary result deals with the weighted Bessel property in  the fractional Sobolev space $H^\tau(\R^2)$ with $\tau\in \R$ for the wavelet system generated by the dual wavelet function $\tilde{\psi}$ and the dual refinable function $\tilde{\phi}$ satisfying certain properties. 
To prove the first auxiliary result,
we recall the bracket product for functions
$f,g : \R^2 \rightarrow \C$ as follows:
\[
[f,g](\xi):=\sum_{k \in \Z^2} f(\xi+2\pi k) \overline{g( \xi + 2\pi k)},
\quad \xi \in \R^2
\]
provided that the series converges absolutely for almost every $\xi \in \R^2$.

Using the ideas in \cite[Theorem~4.6.5]{hanbook} and \cite[Theorem~2.3]{HS09}, we can establish the following result.

%

\begin{theorem} \label{thm:Htau}
Let $\tilde{\eta}\in \{  \tilde{\phi}\otimes \tilde{\psi}, \tilde{\psi}\otimes \tilde{\phi}, \tilde{\psi}\otimes \tilde{\psi}\}$, where $\tilde{\phi},\tilde{\psi}$ are compactly supported functions in $H^{t}(\mathbb{R})$ for some $t>0$, satisfying \eqref{tphitpsi}, and $\vmo(\tilde{\psi}) \ge 2$.
Let $0<\tau_1<\tau_2<2$.
For any $\tau \in [\tau_1,\tau_2]$, there exists a positive constant $C$, which is independent of $\tau \in [\tau_1,\tau_2]$ but may depend on $\tau_1$ and $\tau_2$, such that
\be \label{Htau}
\sum_{j=0}^\infty \sum_{k\in \Z^2}
2^{2\tau j} |\la v, \tilde{\eta}_{j;k} \ra|^2\le C \|v\|_{H^{\tau}(\R^2)}^2,
\qquad \mbox{for all }\quad v \in H^{\tau}(\R^2),
\ee
where $\tilde{\eta}_{j;k}:=2^j \tilde{\eta}(2^j\cdot-k)$, which is the dilated and shifted version of the bivariate function $\tilde{\eta}$.
\end{theorem}

\begin{proof}
Because $\tilde{\eta}\in L^2(\R^2)$ has compact support, we have $[\wh{\tilde{\eta}}, \wh{\tilde{\eta}}](\xi)=\sum_{k\in \Z^2}
\la \tilde{\eta},\tilde{\eta}(\cdot-k)\ra e^{-i k\cdot \xi}$ (e.g., see \cite[Lemma~4.4.1]{hanbook}), which is a bivariate $2\pi\Z^2$-periodic trigonometric polynomial. Hence, $[\wh{\tilde{\eta}}, \wh{\tilde{\eta}}]\in L^\infty(\T^2)$, which can be also deduced from \cite[Proposition~2.6]{HS09} or
\cite[Lemma 6.3.2]{hanbook}.

Let $k:=(k_1,k_2)\in \Z^2$. We observe that
\begin{align*}
	\int_{[-\pi,\pi]^2} [\wh{v}(2^j \cdot),\wh{\tilde{\eta}}](\xi) e^{i k \cdot \xi} d\xi
	& = \int_{-\pi}^{\pi} \int_{-\pi}^{\pi} [\wh{v}(2^j \cdot,2^j \cdot),\wh{\tilde{\eta}}](\xi_1,\xi_2)e^{i(k_1 \xi_1 + k_2 \xi_2)} d \xi_1 d\xi_2 \\
	& = \int_\R \int_\R \wh{v}(2^j \xi_1,2^j \xi_2) \overline{\wh{\tilde{\eta}}(\xi_1,\xi_2)} e^{i(k_1 \xi_1 + k_2 \xi_2)} d \xi_1 d\xi_2\\
	& = \frac{1}{2^{2j}} \la \wh{v}, \wh{\tilde{\eta}}(2^{-j} \cdot ,2^{-j} \cdot) e^{-i(k_1 \xi_1 + k_2 \xi_2)} \ra = \frac{(2\pi)^2}{2^{2j}} \la v, \tilde{\eta}_{j;k} \ra,
\end{align*}
which due to Parseval's identity yields
\[
\sum_{k \in \Z^2} |\la v, \tilde{\eta}_{j;k} \ra|^2
= \frac{2^{2j}}{(2 \pi)^2} \int_{[-\pi,\pi]^2} | [\wh{v}(2^j \cdot), \wh{\tilde{\eta}}](\xi) |^2 d \xi.
\]
Since $(x + y)^2 \le 2(x^2 + y^2)$ for all $x,y\in \R$, we have
\begin{align*}
|[\wh{v}(2^j \cdot),\wh{\tilde{\eta}}](\xi)|^2 & \le 2 |\wh{v}(2^{j} \xi) \overline{\wh{\tilde{\eta}}(\xi)}|^2 + 2 \left|\sum_{k \in \Z^2 \backslash \{(0,0)\}} \wh{v}(2^{j} (\xi + 2\pi k)) \overline{\wh{\tilde{\eta}}(\xi + 2\pi k)}\right|^2.
\end{align*}
Hence, it follows that
\begin{align*}
	& \sum_{k \in \Z^2} |\la v, \tilde{\eta}_{j;k} \ra|^2
	\le \frac{2^{2j-1}}{\pi^2} \int_{[-\pi,\pi]^2} |\wh{v}(2^{j} \xi) \overline{\wh{\tilde{\eta}}(\xi)}|^2 d \xi  + \frac{2^{2j-1}}{\pi^2} \int_{[-\pi,\pi]^2} \left|\sum_{k \in \Z^2 \backslash \{(0,0)\}} \wh{v}(2^{j} (\xi + 2\pi k)) \overline{\wh{\tilde{\eta}}(\xi + 2\pi k)}\right|^2 d \xi \\
	& \le \frac{1}{2 \pi^2} \int_{\R^2} |\wh{v}(\xi) \wh{\tilde{\eta}}(2^{-j} \xi)|^2 \chi_{[-\pi,\pi]^2}(2^{-j} \xi) d\xi \\
&\qquad\qquad + \frac{2^{2j-1}}{\pi^2} \int_{[-\pi,\pi]^2} \Big[\sum_{k \in \Z^2 \backslash \{(0,0)\}} |\wh{v}(2^{j} (\xi + 2\pi k))|^2 \Big]
\Big[\sum_{k \in \Z^2 \backslash \{(0,0)\}} |\wh{\tilde{\eta}}(\xi + 2\pi k) |^2 \Big] d\xi \\
	& \le \frac{1}{2 \pi^2} \int_{\R^2} |\wh{v}(\xi) \wh{\tilde{\eta}}(2^{-j} \xi)|^2 \chi_{[-\pi,\pi]^2}(2^{-j} \xi) d\xi + \frac{2^{2j} C_{\tilde{\eta}}}{2\pi^2} \int_{[-\pi,\pi]^2} \sum_{k \in \Z^2 \backslash \{(0,0)\}} |\wh{v}(2^{j} (\xi + 2\pi k))|^2 d\xi \\
	& \le \frac{1}{2 \pi^2} \int_{\R^2} |\wh{v}(\xi) \wh{\tilde{\eta}}(2^{-j} \xi)|^2 \chi_{[-\pi,\pi]^2}(2^{-j} \xi) d\xi + \frac{C_{\tilde{\eta}}}{2 \pi^2} \int_{\R^2} |\wh{v}(\xi)|^2 \chi_{\R^{2} \backslash [-\pi,\pi]^2}(2^{-j} \xi) d\xi,
\end{align*}
where we already proved that $C_{\tilde{\eta}} := \| [\wh{\tilde{\eta}},\wh{\tilde{\eta}}] \|_{L^{\infty}(\T^2)} <\infty$.
As a result, we have
\be \label{subHtau}
\sum_{j=0}^{\infty} \sum_{k \in \Z^2} 2^{2 \tau j} |\la v, \tilde{\eta}_{j;k} \ra|^2 \le \frac{1}{4 \pi^2} \int_{\R^2} |\wh{v}(\xi)|^2 (1 + \|\xi\|^2 )^{\tau} (2B_1(\xi)+2C_{\tilde{\eta}}B_2(\xi)) d\xi,
\ee
where
\begin{align*}
	B_1(\xi) & := (1 + \|\xi\|^2)^{-\tau} \sum_{j=0}^{\infty} 2^{2\tau j} |\wh{\tilde{\eta}}(2^{-j} \xi)|^2 \chi_{[-\pi,\pi]^2}(2^{-j} \xi),\\
	B_2(\xi) & := (1 + \|\xi\|^2)^{-\tau} \sum_{j=0}^{\infty} 2^{2\tau j} \chi_{\R^{2} \backslash [-\pi,\pi]^2}(2^{-j} \xi).
\end{align*}

We first estimate $B_1(\xi)$. Recall that $\tilde{\eta}$ takes the form of $\tilde{\phi} \otimes \tilde{\psi}$, $\tilde{\psi} \otimes \tilde{\phi}$, or $\tilde{\psi} \otimes \tilde{\psi}$. By assumption, $\vmo(\tilde{\psi})\ge 2$; i.e., there is a positive constant $C_{\vmo}>0$ such that $|\widehat{\tilde{\psi}}(\xi)|\le C_{\vmo}|\xi|^2$ and $|\widehat{\tilde{\psi}}(\xi)|\le C_{\vmo}|\xi|$ for almost every $\xi \in [-\pi,\pi]$. Thus, one of the following inequalities holds: $|\wh{\tilde{\phi}}(\xi_1) \wh{\tilde{\psi}}(\xi_2)|^2 \le C_{\vmo}^2 |\xi_2|^4$, $|\wh{\tilde{\psi}}(\xi_1) \wh{\tilde{\phi}}(\xi_2)|^2 \le C_{\vmo}^2 |\xi_1|^4$, or $|\wh{\tilde{\psi}}(\xi_1) \wh{\tilde{\psi}}(\xi_2)|^2 \le C_{\vmo}^2 |\xi_1 \xi_2|^2$ for almost every $\xi_1, \xi_2 \in [-\pi,\pi]$. That is, $|\wh{\tilde{\eta}}(\xi)|^2 = |\wh{\tilde{\eta}}(\xi_1,\xi_2)|^2 \le C_{\vmo}^2 (|\xi_1|^2 + |\xi_2|^2)^2 = C_{\vmo}^2 \|\xi\|^4$ for almost every $\xi \in [-\pi,\pi]^2$. Define
\[
J_\xi:=\max\{0,\lceil \log_2(\| \xi \|/\pi) \rceil \}.
\]
Then, for all $\tau\in [\tau_1, \tau_2]$, because $0<\tau_1<\tau_2<2$ and $\tau-2<0$, we have
\begin{align*}
	B_1(\xi) & \le C_{\vmo}^2 (1 + \|\xi\|^2)^{-\tau} \|\xi\|^4  \sum_{j=J_\xi}^{\infty} 2^{2 j (\tau-2)}
	\le C_{\vmo}^2 (1 + \|\xi\|^2)^{-\tau} \|\xi\|^4 \frac{2^{2J_\xi(\tau-2)}}{1-2^{2(\tau-2)}}\\
	& \le C_{\vmo}^2 (1 + \|\xi\|^2)^{-\tau}  \frac{\|\xi\|^4}{\|\xi\|^{2(2-\tau)}} \frac{\pi^{2(2-\tau)}}{1-2^{2(\tau-2)}}
=C_{\vmo}^2
\Big( \frac{\|\xi\|^2}{1+\|\xi\|^2}\Big)^\tau \frac{\pi^{2(2-\tau)}}{1-2^{2(\tau-2)}}
\le \frac{\pi^{4}}{1-2^{2(\tau_2 - 2)}} C_{\vmo}^2,
\end{align*}
which implies that $B_1(\xi) \in L^{\infty}(\R^2)$. 
Define $j_\xi:=\max\{0,\lfloor \log_2(\|\xi\|/\pi) \rfloor \}$.
%
%
Meanwhile, for $\tau \in [\tau_1, \tau_2]$, by $\tau>\tau_1>0$, we have
\begin{align*}
	B_2(\xi) & \le (1 + \|\xi\|^2)^{-\tau} \sum_{j=0}^{j_\xi} 2^{2 \tau  j} \le (1 + \|\xi\|^2)^{-\tau} \frac{2^{2\tau (j_\xi+1)}}{2^{2\tau}-1} \le \left(\frac{2^{2\tau} \pi^{-2\tau} }{2^{2\tau}-1}\right)
\Big(\frac{\|\xi\|^2}{1+\|\xi\|^2}\Big)^\tau
\le \frac{2^{2\tau_1}}{2^{2\tau_1}-1},
\end{align*}
which implies that $B_2(\xi) \in L^{\infty}(\R^2)$. Continuing from \eqref{subHtau}, we have
\[
\sum_{j=0}^{\infty} \sum_{k \in \Z^2} 2^{2 \tau j} |\la v, \tilde{\eta}_{j;k} \ra|^2 \le C \frac{1}{4 \pi^2} \int_{\R^2} |\wh{v}(\xi)|^2  (1+\|\xi\|^2)^{\tau} d\xi,
\]
where $C:=2\|B_1\|_{L^{\infty}(\R^2)} + 2 C_{\tilde{\eta}} \|B_2\|_{L^{\infty}(\R^2)}$ with
\[
\|B_1\|_{L^{\infty}(\R^2)} \le \frac{\pi^{4}}{1-2^{2(\tau_2 - 2)}} C_{\vmo}^2
\quad
\text{and}
\quad
\|B_2\|_{L^{\infty}(\R^2)} \le \frac{2^{2\tau_1}}{2^{2\tau_1}-1}
\]
for all $\tau \in [\tau_1, \tau_2]$ with $0<\tau_1 <\tau_2<2$.
%
%
We obtain the desired conclusion.
\end{proof}

In preparation for the next auxiliary result, we introduce a few notations and present a few observations. From the discussion in \cref{sec:prelim} and more specifically by \eqref{PhiPsi2D}, recall that $\tilde{\Psi}^{2D}_j=\{\tilde{\Phi}_j\otimes \tilde{\Psi}_j, \tilde{\Psi}_j\otimes \tilde{\Phi}_j, \tilde{\Psi}_j\otimes \tilde{\Psi}_j\}$.
We can split the set $\tilde{\Psi}^{2D}_j$ into two groups:
\be \label{GxGy}
G_j^x:=[\tilde{\Psi}_j\otimes \tilde{\Phi}_j] \cup [\tilde{\Psi}_j\otimes \tilde{\Psi}_j]
\quad \text{and} \quad
G_j^y:=\tilde{\Phi}_j\otimes \tilde{\Psi}_j.
\ee
Note that any $\tilde{\alpha}_j\in \tilde{\Psi}^{2D}_j$ must belong to either $G_j^x$ or $G_j^y$.
%
%
%
%
We define an integration operation along one axis as follows:
\[
\mathring{\tilde{\alpha}}_j :=
\begin{cases}
	2^j \int_x^1 \tilde{\alpha}_j(t) dt, & \tilde{\alpha}_j \in \tilde{\Psi}_j,\\
	2^j \int_x^{1} \tilde{\alpha}_j(t,y) dt, & \tilde{\alpha}_j\in G_j^x,\\
	2^j \int_y^{1} \tilde{\alpha}_j(x,t) dt, & \tilde{\alpha}_j\in G_j^y.
\end{cases}
\]
%
%
Moreover, if $\tilde{\alpha}_j \in G_j^x$, then
\be \label{breve:Psi}
\mathring{\tilde{\alpha}}_j \in
\{\breve{\tilde{\Psi}}_j \otimes \Phi_j,
\breve{\tilde{\Psi}}_j\otimes \Psi_j\},\qquad \mbox{for all}\quad j\ge J_0,
\ee
where $\breve{\tilde{\Psi}}_{j} = \{\mathring{\tilde{\psi}}^{L}_{j;0}\} \cup \{\mathring{\tilde{\psi}}_{j;k} \setsp n_{l,\psi} \le k \le 2^{j}-n_{h,\psi}\}\cup \{ \mathring{\tilde{\psi}}^{R}_{j;2^{j}-1}\}$, and $n_{l,\psi},n_{h,\psi}$ are the same integers in \eqref{PhiPsij}.
If $\tilde{\alpha}_j \in G^{y}_j$, then $\mathring{\tilde{\alpha}}_j \in \Phi_j \otimes \breve{\tilde{\Psi}}_j$ for all $j \ge J_0$, where $\breve{\tilde{\Psi}}_j$ is defined the same way as before.
%



To prove \cref{thm:converg}, we shall also need the following second auxiliary result.

\begin{theorem} \label{thm:H2}
Let $\tilde{\Psi}_j^{2D}$ with $j\ge J_0$ and $J_0 \in \N$ be defined in \eqref{PhiPsi2D} (constructed by the direct approach in \cite{HM21a}). Assume that $\tilde{\phi},\tilde{\psi} \in H^t(\mathbb{R})$ for some $t>0$ and $\vmo(\tilde{\psi})
\ge \mm$.
Then there exists a positive constant $C$ such that
\be \label{v:J}
\sum_{j=J}^\infty \sum_{\tilde{\alpha}_j \in \tilde{\Psi}^{2D}_j} 2^{2j} |\la v, \tilde{\alpha}_j\ra|^2
	\le 2^{-2(m-1)J} C |v|^{2}_{H^{\mm}(\Omega)},
\qquad \mbox{for all } J\ge J_0, v \in H^{\mm}(\Omega)\cap H^1_0(\Omega),
\ee
where $|v|_{H^\mm(\Omega)}$ is the semi-norm of $v$ in $H^\mm(\Omega)$, i.e., $|v|_{H^\mm(\Omega)}^2:=\sum_{|\mu|=\mm} \|\partial^\mu v\|_{L^2(\Omega)}^2$, where $\mu:=(\mu_1,\mu_2)$ and $\partial^{\mu} := \tfrac{\partial^{\mu_1 + \mu_2}}{\partial x^{\mu_1} \partial y^{\mu_2}}$ with $\mu_1, \mu_2\in \N\cup\{0\}$.
\end{theorem}

\begin{proof}
Note that for all $j \ge J_0$, $\tilde{\Psi}_j^{2D}\subset\tilde{\mathcal{B}}_{J_0}^{2D}$, and $\tilde{\mathcal{B}}_{J_0}^{2D}$ is the dual part of the biorthogonal wavelet $(\tilde{\mathcal{B}}_{J_0}^{2D}, \mathcal{B}_{J_0}^{2D})$ in $L^2(\Omega)$. Define $V_J=\text{span}(\Phi^{2D}_J)$ and $h:=2^{-J}$. $V_J$ is just the finite element space constructed via tensor product. 

For $v\in H^\mm(\Omega)\cap H^1_0(\Omega)$, we define $I_h v$ to be the interpolation function of $v$ on the grid of $\Omega$ using the mesh size $h$, i.e., on the grid $\Omega \cap 2^{-J}\Z^2$ such that $[I_h v](p)=v(p)$ for all $p\in \Omega\cap 2^{-J}\Z^2$. By \cite[Theorem~4.6.14]{BS08}, there is a positive constant $C_0$, depending only on the primal refinable function $\phi$,
such that
\be \label{fem}
|v-I_h v|_{H^1(\Omega)}:=\| \nabla (v-I_h v)\|_{L^2(\Omega)}\le C_0 h^{\mm-1} |v|_{H^\mm(\Omega)}=C_0 2^{-(\mm-1)J} |v|_{H^\mm(\Omega)}.
\ee
Note that $I_h v\in H^1_0(\Omega)$.
Because  $I_h v\in V_J=\text{span}(\mathcal{B}^{2D}_{J_0,J})$, $\mathcal{B}^{2D}_{J_0,J}\subset \mathcal{B}^{2D}_{J_0}$, and $(\tilde{\mathcal{B}}_{J_0}^{2D}, \mathcal{B}_{J_0}^{2D})$ is a biorthogonal wavelet in $L^2(\Omega)$, we have the following crucial perpendicular condition:
\be \label{perp}
\la I_h v, \tilde{\alpha}_j\ra=0 \qquad \mbox{for all } \tilde{\alpha}_j\in \tilde{\Psi}^{2D}_j \mbox{ with } j\ge J.
\ee

We now estimate \eqref{v:J}.
We first handle the case, where $\tilde{\alpha}_j \in G_j^x$. 
Integrating by parts with respect to $x$ and using the fact that $v-I_h v\in H^1_0(\Omega)$ satisfies the homogeneous Dirichlet condition, we obtain
\[
2^j \la v - I_h v, \tilde{\alpha}_j\ra = \la \tfrac{\partial}{\partial x}(v - I_h v), \mathring{\tilde{\alpha}}_j\ra,\qquad \tilde{\alpha}_j \in G_j^x.
\]
Therefore, we conclude that
\begin{align*}
\sum_{j=J}^\infty \sum_{\tilde{\alpha}_j \in G_j^x} 2^{2j} |\la v, \tilde{\alpha}_j\ra|^2
&=\sum_{j=J}^\infty \sum_{\tilde{\alpha}_j \in G_j^x} 2^{2j} |\la v-I_h v, \tilde{\alpha}_j\ra|^2
= \sum_{j=J}^\infty \sum_{\tilde{\alpha}_j \in G_j^x} |\la \tfrac{\partial}{\partial x}(v - I_h v), \mathring{\tilde{\alpha}}_j\ra|^2\\
&\le \sum_{j=J}^\infty \sum_{\mathring{\tilde{\alpha}}_j \in [\breve{\tilde{\Psi}}_j\otimes \tilde{\Phi}_j]\cup[
\breve{\tilde{\Psi}}_j\otimes \tilde{\Psi}_j]}
|\la \tfrac{\partial}{\partial x}(v - I_h v), \mathring{\tilde{\alpha}}_j\ra|^2,
\end{align*}
where we used \eqref{breve:Psi} and the fact that $G_j^x\subseteq [\breve{\tilde{\Psi}}_j\otimes \tilde{\Phi}_j]\cup[\breve{\tilde{\Psi}}_j\otimes \tilde{\Psi}_j]$ to arrive at the last line.
Because $v - I_h v\in H_0^1(\Omega)$, we have $\tfrac{\partial}{\partial x}(v - I_h v)\in \LpO{2}$. Note that all the elements in $[\breve{\tilde{\Psi}}_j\otimes \tilde{\Phi}_j]\cup[
\breve{\tilde{\Psi}}_j\otimes \tilde{\Psi}_j]$ are supported in $(0,1)^2$ and belong to $H^{t}(\mathbb{R}^2)$ for some $t>0$. Since $\vmo(\tilde{\psi})\ge m$, all elements taking the form of $\mathring{\tilde{\psi}}_{j;k} \otimes \eta_j$, where $\eta_j \in \tilde{\Phi}_j \cup \tilde{\Psi}_j$, have compact support and at least one vanishing moment.
Hence, by the Bessel property in \cite[Lemma~7.1]{HM23} and \cite[Theorems~2.2 and~2.3]{han03}, there must exist a positive constant $C_1$, independent of $J$ and only depending on the wavelet, such that
\be \label{est1}
\sum_{j=J}^\infty \sum_{\tilde{\alpha}_j \in G_j^x} 2^{2j} |\la v, \tilde{\alpha}_j\ra|^2
\le \sum_{j=J}^\infty \sum_{\mathring{\tilde{\alpha}}_j \in [\breve{\tilde{\Psi}}_j\otimes \tilde{\Phi}_j]\cup[
\breve{\tilde{\Psi}}_j\otimes \tilde{\Psi}_j]}
|\la \tfrac{\partial}{\partial x}(v - I_h v), \mathring{\tilde{\alpha}}_j\ra|^2
\le C_1 \|\tfrac{\partial}{\partial x}(v - I_h v)\|^2_{\LpO{2}}.
\ee

We now handle the case, where $\tilde{\alpha}_j\in G_j^y$. By using a similar argument as before, we obtain
\[
2^j \la v - I_h v, \tilde{\alpha}_j\ra = \la \tfrac{\partial}{\partial y}(v - I_h v), \mathring{\tilde{\alpha}}_j\ra,\qquad \tilde{\alpha}_j \in G_j^y.
\]
Furthermore, there exists a positive constant $C_2$, independent of $J$, such that
\[
\sum_{j=J}^\infty \sum_{\tilde{\alpha}_j \in G_j^y} 2^{2j} |\la v, \tilde{\alpha}_j\ra|^2
\le \sum_{j=J}^\infty \sum_{\mathring{\tilde{\alpha}}_j \in \tilde{\Phi}_j\otimes \breve{\tilde{\Psi}}_j}
|\la \tfrac{\partial}{\partial y}(v - I_h v), \mathring{\tilde{\alpha}}_j\ra|^2
\le C_2 \|\tfrac{\partial}{\partial y}(v - I_h v)\|^2_{\LpO{2}}.
\]
Combining the above estimates with \eqref{est1}, we conclude that
\begin{align*}
\sum_{j=J}^\infty \sum_{\tilde{\alpha}_j \in \tilde{\Psi}_j^{2D}} 2^{2j} |\la v, \tilde{\alpha}_j\ra|^2
&=\sum_{j=J}^\infty \sum_{\tilde{\alpha}_j \in G_j^x \cup G_j^y} 2^{2j} |\la v, \tilde{\alpha}_j\ra|^2\\
&\le C_1 \|\tfrac{\partial}{\partial x}(v - I_h v)\|^2_{\LpO{2}}+C_2 \|\tfrac{\partial}{\partial y}(v - I_h v)\|^2_{\LpO{2}}\\
&\le \max(C_1,C_2) \| \nabla (v- I_h v)\|^2_{\LpO{2}}=
\max(C_1,C_2) |v- I_h v|^2_{H^1(\Omega)},
\end{align*}
which combined with the approximation result in \eqref{fem} further yields
\[
\sum_{j=J}^\infty \sum_{\tilde{\alpha}_j \in \tilde{\Psi}_j^{2D}} 2^{2j} |\la v, \tilde{\alpha}_j\ra|^2
\le \max(C_1,C_2) |v- I_h v|^2_{H^1(\Omega)}
\le \max(C_1,C_2) C_0^2 h^{2(\mm-1)} |v|^2_{H^{\mm}(\Omega)}.
\]
This proves \eqref{v:J} with $C:= \max(C_1,C_2) C_0^2<\infty$, where $h=2^{-J}$.
\end{proof}

We also need the following lemma in the proof of \cref{thm:converg}.

\begin{lemma}\label{lem:chi}
	Let $\Omega_-$ be a bounded open domain with a smooth boundary $\Gamma$. Then
	\be \label{est:chi}
	 \|\chi_{\Omega_-}\|^2_{H^{1/2-\gep}(\R^2)}
	\le C_\Gamma^2 \gep^{-1},\qquad\forall\;  0<\gep<1/4,
	\ee
	for a positive constant $C_\Gamma$ that only depends on $\Omega_-$ but is independent of $0<\gep<1/4$.
\end{lemma}

\begin{proof}
	For a compactly supported function $F\in L^2(\R^2)$, recall that its modulus of smoothness in the $L^2$-norm is defined by
	\be \label{modsmooth}
	\omega(F,s)_2:=\sup_{|t|\le s} \| F(\cdot+t)-F\|_{L^2(\R^2)},\qquad s>0.
	\ee
	For any $0<\tau<1$, the semi-norm $F \in H^\tau(\R^2)$ is
	\be \label{sobolevsemi}
	|F|_{H^\tau(\R^2)}:=
	\left(\int_0^\infty [s^{-\tau}\omega(F,s)_2]^2 \frac{ds}{s}\right)^{1/2}.
	\ee
	For $t\in \R^2$, we define
	\[
	E_t:=\{q \in \R^2 \quad : \quad
	q\in [(\R^2\bs \Omega_-)+t] \cap \Omega_-\quad \mbox{or}\quad q\in [\Omega_-+t]\cap [\R^2\bs \Omega_-]
	\},
	\]
	which has a measure of order $\mathcal{O}(\|t\|)$, because $\Gamma=\ol{\Omega_-}\cap \ol{\R^2\bs \Omega_-}$ is a closed smooth curve.
	Define $F:=\chi_{\Omega_-}$. Then
	for all $t\in \R^2$, we have
	\[
	\|F(p-t)-F(p)\|^2_{L^2(\R^2)}
	=\int_{\R^2} |F(p-t)-F(p)|^2 dp
	=\int_{E_t}  dp\le C^2_1 \|t\|,
	\]
	for a positive constant $C_1$ only depending on $\partial \Omega_-$. Consequently, we have
	\be \label{moduli}
	\omega(F, s)_2
	\le C_1 s^{1/2},\qquad \forall s\in (0,\infty).
	\ee
	Note $\omega(F,s)_2 \le 2 \|F\|_{L^2(\R^2)}=2|\Omega_-|^{1/2}$ by the triangle inequality. Hence, for $0<\gep<1/4$, we have
	\begin{align*}
		|F|^2_{H^{1/2-\gep}(\R^2)}
		& =\int_0^\infty [s^{-(1/2-\gep)}\omega(F,s)_2]^2 \frac{ds}{s}
		=\int_0^{1} [s^{-(1/2-\gep)}\omega(F,s)_2]^2 \frac{ds}{s}
+\int_{1}^\infty [s^{-(1/2-\gep)}\omega(F,s)_2]^2 \frac{ds}{s} \\
		& \le  C_1^2 \int_0^{1} s^{2\gep-1} s \frac{ds}{s}+
4|\Omega_-| \int_1^\infty s^{2\gep-1} \frac{ds}{s}
		=\frac{C_1^2}{2\gep}+
\frac{4}{1-2\gep} |\Omega_-|\le \frac{C_1^2}{2\gep}+
8 |\Omega_-|,
	\end{align*}
	where we used \eqref{moduli} for the first inequality.
	Since $\|F\|^2_{H^{1/2-\gep}(\R^2)} = \|F\|^2_{L^{2}(\R^2)} + |F|^2_{H^{1/2-\gep}(\R^2)}$ and $\|F\|^2_{L^{2}(\R^2)} = |\Omega_-|$, we proved the claim with $C_\Gamma^2:=\frac{1}{2} C_1^2 + \frac{9}{4}|\Omega_-|$ for all $0<\gep<1/4$.
\end{proof}

We are now ready to present the proof of \cref{thm:converg}.

\begin{proof}[Proof of \cref{thm:converg}]
We split the analysis and estimation in three regions: $u_+$, $u_-$, and $u_\Gamma$; the last of the three is the neighborhood of the interface curve $\Gamma$.

\underline{Proving the $H^1(\Omega)$ convergence.}
Because the interface $\Gamma$ is of class $\mathscr{C}^2$ and $u_+ \in H^2(\Omega_+)$ by our assumption~\eqref{assumption:u}, we can extend the function $u_+$ from $\Omega_+$ to the domain $\Omega$ and obtain a function $v_+\in H^2(\Omega)\cap H_0^1(\Omega)$ such that
$v_+=u_+$ in $\Omega_+$ and $\|v_+\|_{H^2(\Omega)} \le C_0 \|u_+\|_{H^2(\Omega_+)}$.
Therefore, by \cref{thm:H2} with $m\ge 2$, there exists a positive constant $C_+$ such that
\be \label{v+}
\sum_{j=J}^\infty \sum_{\tilde{\alpha}_j \in \tilde{\Psi}^{2D}_j} 2^{2j} |\la v_+, \tilde{\alpha}_j\ra|^2
\le 2^{-2J} C_+ |v_+|^{2}_{H^{2}(\Omega)},
\qquad \mbox{for all } J\ge J_0.
\ee
If $\tilde{\alpha}_j \in \tilde{\Psi}^{2D}_j$ and $\SSupp(\tilde{\alpha}_j)$ is completely inside $\Omega_+$, due to $v_+=u_+$ on $\Omega_+$, we must have
\[
\la v_+, \tilde{\alpha}_j\ra=\la u_+, \tilde{\alpha}_j\ra=\la u, \tilde{\alpha}_j\ra,
\qquad \mbox{if}\quad
\SSupp(\tilde{\alpha}_j)\subseteq \Omega_+.
\]
Consequently, we conclude from \eqref{v+} and $\|v_+\|_{H^2(\Omega)} \le C_0 \|u_+\|_{H^2(\Omega_+)}$ that
\be \label{v+:u+}
\sum_{j=J}^\infty \sum_{\substack{\tilde{\alpha}_j \in \tilde{\Psi}^{2D}_j,\\ \SSupp(\tilde{\alpha}_j)\subseteq \Omega_+}} 2^{2j} |\la u, \tilde{\alpha}_j\ra|^2
\le
\sum_{j=J}^\infty \sum_{\tilde{\alpha}_j \in \tilde{\Psi}^{2D}_j} 2^{2j} |\la v_+, \tilde{\alpha}_j\ra|^2 \le
2^{-2J} C_1 |u_+|^{2}_{H^{2}(\Omega_+)},\quad \forall\, J\ge J_0,
\ee
where $C_1:=C_+C_0^2$.
Similarly, by assumption~\eqref{assumption:u}, $u_-$ in $\Omega_-$ can be extended into a function $v_-\in H^2(\Omega)\cap H_0^1(\Omega)$ such that $v_-=u_-$ in $\Omega_-$ and
$\|v_-\|_{H^2(\Omega)} \le C_0 \|u_-\|_{H^2(\Omega_-)}$.
Therefore, by \cref{thm:H2} with $m\ge 2$ and the same argument, there exists a positive constant $C_2$ such that for all $J\ge J_0$,
\be \label{v-:u-}
\sum_{j=J}^\infty \sum_{\substack{\tilde{\alpha}_j \in \tilde{\Psi}^{2D}_j, \\ \SSupp(\tilde{\alpha}_j)\subseteq \Omega_-}} 2^{2j} |\la u, \tilde{\alpha}_j\ra|^2
\le
\sum_{j=J}^\infty \sum_{\tilde{\alpha}_j \in \tilde{\Psi}^{2D}_j} 2^{2j} |\la v_-, \tilde{\alpha}_j\ra|^2 \le
2^{-2J} C_2 |u_-|^{2}_{H^{2}(\Omega_-)},\quad \forall\, J\ge J_0.
\ee

We now handle the solution $u$ in a neighborhood of the interface $\Gamma$.
Because the closed curve $\Gamma$ is completely inside $\Omega$, we can assume that there exist two open neighborhoods $\Omega_0$ and $\Omega_\rho$ of $\Gamma$ such that $\Omega_0\subseteq \Omega_\rho \subseteq \Omega$, the closure of $\Omega_\rho$ is contained inside $\Omega$ and the closure of $\Omega_0$ is inside $\Omega_\rho$.
Since $\Gamma$ is a curve, we can take a compactly supported smooth function $\rho$ supported inside $\Omega_\rho$ such that $\rho=1$ in $\Omega_0$.
Define a bivariate function
$w:=\rho u$. Obviously, $\SSupp(w)\subseteq \Omega_\rho$ and hence $w$ can be regarded as a function in the whole space $\R^2$ by the zero extension outside $\Omega$.
Therefore, applying \cref{thm:Htau} with $\tau_1=5/4$ and $\tau_2=7/4$, for any $\tau\in [\tau_1,\tau_2]$, there exists a positive constant $C_3$, independent of $\tau \in [\tau_1,\tau_2]$, such that
\be \label{eta:singular}
\sum_{\tilde{\eta}\in \{ \tilde{\phi}\otimes \tilde{\psi}, \tilde{\psi}\otimes \tilde{\phi}, \tilde{\psi}\otimes \tilde{\psi}\}}
\sum_{j=0}^\infty \sum_{k\in \Z^2}
2^{2\tau j} |\la w, \tilde{\eta}_{j;k}\ra|^2\le 2^{-3/2} C_3 \|w\|^2_{H^\tau(\R^2)},\quad \forall\, \tau\in [\tau_1, \tau_2] := [5/4,7/4].
\ee
It is important to notice that $\rho$ is supported inside $\Omega_\rho$
and $\rho=1$ in $\Omega_0$.
Take any element
\be \label{talpha:Gamma}
\tilde{\alpha}_j\in \tilde{\Psi}^{2D}_j\quad \mbox{and}\quad
\SSupp(\tilde{\alpha}_j)\cap \Gamma\ne \emptyset
\ee
for $j\ge J_0$.
Because $\Gamma$ is away from the boundary $\partial \Omega$ and because the support of $\tilde{\alpha}_j$ becomes smaller and smaller and closer to the interface $\Gamma$ for $j$ large enough, any element $\tilde{\alpha}_j$ in \eqref{talpha:Gamma} cannot be the boundary wavelets, i.e., we must have
$\tilde{\alpha}_j=\tilde{\eta}_{j;k}:=2^{j}\tilde{\eta}(2^j\cdot-k)$ for some $k\in \Z^2$ and
$\tilde{\eta}\in \{ \tilde{\phi}\otimes \tilde{\psi}, \tilde{\psi}\otimes \tilde{\phi}, \tilde{\psi}\otimes \tilde{\psi}\}$.
In addition, $\rho$ takes value $1$ in $\Omega_0$ and the support of $\tilde{\alpha}_j$ will be contained inside $\Omega_0$ for large enough $j$. In conclusion,
there must exist a positive integer $\mathring{J}$ such that any element $\tilde{\alpha}_j$ in \eqref{talpha:Gamma} with $j\ge \mathring{J}$ must satisfy
\[
\SSupp(\tilde{\alpha}_j)\subset \Omega_0\quad \mbox{and}\quad
\tilde{\alpha}_j =\tilde{\eta}_{j;k}
\quad \mbox{for some } k\in \Z^2,
\]
where $\tilde{\eta}\in \{ \tilde{\phi}\otimes \tilde{\psi}, \tilde{\psi}\otimes \tilde{\phi}, \tilde{\psi}\otimes \tilde{\psi}\}$. Moreover, because $\rho=1$ on $\Omega_0$ and $\SSupp(\tilde{\alpha}_j)\subset \Omega_0$,
we have
\[
\la w, \tilde{\alpha}_j\ra=\la \rho u, \tilde{\alpha}_j\ra=\la u, \tilde{\alpha}_j\ra=\la u, \tilde{\eta}_{j;k}\ra
\]
for some unique $k\in \Z^2$, where we used the definition $w=\rho u$.

For simplicity of discussion, without loss of any generality, we can assume $\mathring{J}=J_0$, because we are only interested in large $J$ for proving the convergence rate. Hence, for any $\tilde{\alpha}_j$ satisfying \eqref{talpha:Gamma} with $j\ge J_0$, the above discussion implies that for $\tau\in [\tau_1,\tau_2]:=[5/4,7/4]$ and $J\ge J_0$, we have
\begin{align*}
\sum_{j=J}^\infty \sum_{\substack{\tilde{\alpha}_j\in \tilde{\Psi}_j^{2D}, \\ \SSupp(\tilde{\alpha}_j)\cap \Gamma\ne \emptyset}} 2^{2\tau j} |\la u, \tilde{\alpha}_j\ra|^2
&=\sum_{j=J}^\infty \sum_{\substack{\tilde{\alpha}_j\in \tilde{\Psi}_j^{2D}, \\ \SSupp(\tilde{\alpha}_j)\cap \Gamma\ne \emptyset}} 2^{2\tau j} |\la w, \tilde{\alpha}_j\ra|^2
\le \sum_{\tilde{\eta}\in \{ \tilde{\phi}\otimes \tilde{\psi}, \tilde{\psi}\otimes \tilde{\phi}, \tilde{\psi}\otimes \tilde{\psi}\}}
\sum_{j=0}^\infty \sum_{k\in \Z^2}
2^{2\tau j} |\la w, \tilde{\eta}_{j;k}\ra|^2.
\end{align*}
Now we conclude from the inequality \eqref{eta:singular} and the above estimation that
\be \label{singular}
\sum_{j=J}^\infty \sum_{\substack{\tilde{\alpha}_j\in \tilde{\Psi}_j^{2D}, \\ \SSupp(\tilde{\alpha}_j)\cap \Gamma\ne \emptyset}} 2^{2\tau j} |\la u, \tilde{\alpha}_j\ra|^2
\le 2^{-3/2} C_3 \|w\|_{H^\tau(\Omega)}^2,\quad \text{for all } J\ge J_0, \tau\in [\tau_1,\tau_2].
\ee
In particular, for $\tau\in [\tau_1,\tau_2]$, we have
\begin{align*}
& \sum_{j=2J-1}^\infty \sum_{\substack{\tilde{\alpha}_j\in \tilde{\Psi}_j^{2D}, \\ \SSupp(\tilde{\alpha}_j)\cap \Gamma\ne \emptyset}} 2^{2j}  |\la u, \tilde{\alpha}_j\ra|^2
=\sum_{j=2J-1}^\infty \sum_{\substack{\tilde{\alpha}_j\in \tilde{\Psi}_j^{2D}, \\ \SSupp(\tilde{\alpha}_j)\cap \Gamma\ne \emptyset}} 2^{-2(\tau-1)j} 2^{2\tau j}  |\la u, \tilde{\alpha}_j\ra|^2 \\
& \quad \le 2^{3/2} 2^{-4J (\tau-1)}
\sum_{j=2J-1}^\infty \sum_{\substack{\tilde{\alpha}_j\in \tilde{\Psi}_j^{2D}, \\ \SSupp(\tilde{\alpha}_j)\cap \Gamma\ne \emptyset}} 2^{2\tau j} |\la u, \tilde{\alpha}_j\ra|^2
\le 2^{-4J(\tau-1)} C_3 \|w\|_{H^{\tau}(\Omega)}^2,
\end{align*}
where we used $2^{2(\tau-1)}\le 2^{2(\tau_2-1)}\le 2^{3/2}$ due to $\tau_2=7/4$.
Consider $\tau:=3/2-2\gep\in [5/4,3/2)$ with $0<\gep<1/8$. Then obviously, $\tau\in [\tau_1,\tau_2]:=[5/4,7/4]$.
Since $2(\tau-1)=2(1/2-2\gep)=1-4\gep$, the above estimate can be equivalently re-expressed as follows:
\be \label{v2J}
\sum_{j=2J-1}^\infty
\sum_{\substack{\tilde{\alpha}_j\in \tilde{\Psi}_j^{2D}, \\ \SSupp(\tilde{\alpha}_j)\cap \Gamma\ne \emptyset}} 2^{2j }  |\la u, \tilde{\alpha}_j\ra|^2
\le
2^{-2J(1-4\gep)}  C_3 \|w\|_{H^{3/2-2\gep}(\Omega)}^2, \quad \text{for all } 0<\gep<1/8.
\ee

In the following, we estimate the quantity $\|w\|_{H^{3/2-2\gep}(\Omega)}$ for $0<\gep<1/8$, specifically for $\gep\to 0^+$.
Define $w_+:=w\chi_{\Omega_+}=\rho u_+$ and $w_-:=w\chi_{\Omega_-}=\rho u_-$. Then $w_+\in H^2(\Omega_+)$ and $w_-\in H^2(\Omega_-)$. Moreover, $\rho v_+$ and $\rho v_-$ are extensions of $w_+$ and $w_-$, respectively.
Because $w=\rho u \in H_0^1(\Omega)$, we consider $\nabla w$.
To estimate $\|w\|_{H^{3/2-2\gep}(\Omega)}$ for $0<\gep<1/8$, it suffices to estimate $\|\nabla w\|_{H^{1/2-2\gep}(\Omega)}$.
For simplicity of discussion, we only handle $\frac{\partial}{\partial x} w$ and we assume that $\Omega_-$ is inside $\Omega$ and $\partial \Omega_-\cap \partial \Omega=\emptyset$.
Because $w=\rho u \in H_0^1(\Omega)$, we have $\frac{\partial}{\partial x} w\in \LpO{2}$. Noting that $\rho v_+\in H^2(\R^2)$, we can rewrite
\[
w_x:=\frac{\partial}{\partial x} w=
\frac{\partial}{\partial x} [\rho v_+] \chi_{\Omega_+}
+\frac{\partial}{\partial x} [\rho v_-] \chi_{\Omega_-}
=
\frac{\partial}{\partial x} [\rho v_+]+F\chi_{\Omega_-}\quad \mbox{with}\quad
F:=\frac{\partial}{\partial x} [\rho v_-]-\frac{\partial}{\partial x} [\rho v_+],
\]
because $w=\rho u_+=\rho v_+$ in $\Omega_+$ and $w=\rho u_-=\rho v_-$ in $\Omega_-$. Note that $F\in H^1(\R^2)$ and $F$ has compact support by $v_+, v_-\in H^2(\Omega)\cap H^1_0(\Omega)$. Consequently, there exists a positive constant $C_4$ such that
\be \label{est:F}
\|F\|_{H^1(\R^2)} \le C_4 (\|u_+\|_{H^2(\Omega_+)}+\|u_-\|_{H^2(\Omega_-)}),
\ee
where $C_4$ only depends on the smooth function $\rho$ and the positive constant $C_0$ appearing in $\|v_+\|_{H^2(\Omega)}\le C_0 \|u_+\|_{H^2(\Omega_+)}$ and $\|v_-\|_{H^2(\Omega)}\le C_0 \|u_-\|_{H^2(\Omega_-)}$.
We still consider $\tau:=3/2-2\gep$ with $0<\gep<1/8$. Then $\tau_\gep:=\tau-1=1/2-2\gep \in [1/4, 1/2)$.
Hence, by $w_x=\frac{\partial}{\partial x} [\rho v_+]+F\chi_{\Omega_-}$, we have
\be \label{wx:htaueps}
\|w_x\|_{H^{\tau_\gep}(\R^2)}
\le \|[\rho v_+]_x\|_{H^{\tau_\gep}(\R^2)}
+\|F\chi_{\Omega_-}\|_{H^{\tau_\gep}(\R^2)}
\le C_\rho C_0 \|u_+\|_{H^2(\Omega_+)}+\|F\chi_{\Omega_-}\|_{H^{\tau_\gep}(\R^2)},
\ee
where 
we used  $\|[\rho v_+]_x\|_{H^{\tau_\gep}(\R^2)} = \|[\rho v_+]_x\|_{H^{\tau_\gep}(\Omega)} \le C_\rho \|v_+\|_{H^2(\Omega)}$ with the positive constant $C_\rho$ depending only on $\rho$, and the inequality $\|v_+\|_{H^2(\Omega)} \le C_0 \|u_+\|_{H^2(\Omega_+)}$.

Next, we estimate $\|F\chi_{\Omega_-}\|_{H^{\tau_\gep}(\R^2)}$.
Since $F\in H^1(\R^2)$, by \cite[Theorems~C.9 and C.10]{bsbook07} with $r=\tau_\gep, s=\tau_\gep+\gep, t=1$ and $d=2$,
there exists a positive constant $C_5$ only depending on $\Omega_-$ such that
\be \label{fxomega:htaueps}
\|F\chi_{\Omega_-}\|_{H^{\tau_\gep}(\R^2)}
\le C_5\gep^{-1/2} \|F\|_{H^1(\R^2)} \|\chi_{\Omega_-}\|_{H^{\tau_\gep+\gep}(\R^2)},
\ee
where the above factor $\gep^{-1/2}$ is from
$\sum_{p=0}^\infty 2^{-2p (s+t-d/2-r)}=\sum_{p=0}^\infty 2^{-2p\gep}=\frac{1}{1-2^{-2\gep}}\le \frac{1}{\gep \sqrt{2} \ln 2}$ for all $0<\gep<1/8$ in \cite[Proof of Theorem~C.10]{bsbook07} by noting $s+t-d/2-r=\gep>0$.
Noting $\tau_\gep+\gep=\tau-1+\gep=1/2-\gep$, we obtain from \eqref{est:chi} in \cref{lem:chi} that $\|\chi_{\Omega}\|_{H^{\tau_\gep+\gep}(\R^2)}\le C_\Gamma \gep^{-1/2}$.
Combining \eqref{est:F}, \eqref{wx:htaueps} and \eqref{fxomega:htaueps}, we obtain
\[
\|w_x\|_{H^{\tau_\gep}(\R^2)}
\le C_\rho C_0 \|u_+\|_{H^2(\Omega_+)}+
C_4 C_5 C_\Gamma (\|u_+\|_{H^2(\Omega_+)}+\|u_-\|_{H^2(\Omega_-)})
\gep^{-1}.
\]
An estimate for $\|w_y\|_{H^{\tau_\gep}(\Omega)}$ can be proved similarly.
Note that $\|w\|_{H^{3/2-2\gep}(\Omega)}^2=\|w\|_{H^1(\Omega)}^2+\|\nabla w\|_{H^{\tau_\gep}(\Omega)}^2$ by $\tau_\gep=\tau-1=1/2-2\gep$.
Noting that $w=\rho u\in H^1_0(\Omega)$ and
\[
\|w\|_{H^1(\Omega)}^2=\|\rho u_+\|_{H^1(\Omega_+)}^2+\|\rho u_-\|_{H^1(\Omega_-)}^2
\le C_\rho^2 [\|u_+\|^2_{H^2(\Omega_+)}+\|u_-\|^2_{H^2(\Omega_-)}],
\]
we conclude from the above inequality estimating $\|w_x\|_{H^{\tau_\eps}(\R^2)}$ and similarly $\|w_y\|_{H^{\tau_\eps}(\R^2)}$ that
\be \label{est:w}
\|w\|_{H^{3/2-2\gep}(\Omega)}^2 \le C_6 \gep^{-2} (\|u_+\|_{H^2(\Omega_+)}^2+\|u_-\|_{H^2(\Omega_-)}^2),
\qquad \forall\; 0<\gep<1/8,
\ee
where $C_6:=C_\rho^2+4(C_\rho C_0+C_4C_5C_\Gamma)^2+4 C_4^2 C_5^2C_\Gamma^2<\infty$.
Since $\tau=3/2 - 2\gep$ for $0<\gep<1/8$, we deduce from \eqref{v2J} and \eqref{est:w} that
$\tau-1=1/2-2\gep$, $-2J(1-4\gep)=-4J(\tau-1)$ and
\be \label{w2J:2}
\sum_{j=2J-1}^\infty \sum_{\substack{\tilde{\alpha}_j\in \tilde{\Psi}_j^{2D}, \\ \SSupp(\tilde{\alpha}_j)\cap \Gamma\ne \emptyset}} 2^{2j }  |\la u, \tilde{\alpha}_j\ra|^2
\le
2^{-4J(\tau-1)}  C_3 \|w\|_{H^{\tau}(\R^2)}^2
\le
C_3 C_6 2^{-2J} (H(\gep))^{-2} (\|u_+\|_{H^2(\Omega_+)}^2
 + \|u_-\|_{H^2(\Omega_-)}^2),
\ee
where $H(\gep):=\gep 2^{-4J\gep}$, which can be written as $H(\gep)=\gep h^{4\gep}$ with $h:=2^{-J}$.
Note that
\[
H'(\gep)=h^{4\gep}+4 \gep h^{4\gep} \log (h) = h^{4\gep} (1+4\gep \log(h)),
\]
where $\log$ is the natural logarithm. Setting $H'(\gep)=0$ gives $\gep=\frac{1}{4 \log(h^{-1})}>0$, i.e., $\gep=\frac{1}{4} |\log (h)|^{-1}$ because $0<h<1$.
Taking $\gep=\frac{1}{4}|\log(h)|^{-1}$ in
\eqref{w2J:2}, we conclude that $H(\gep)=\frac{1}{4}|\log (h)|^{-1} e^{-1}$, $(H(\gep))^{-2}=16 e^2 |\log (h)|^2$, and finally
we deduce from \eqref{w2J:2} that
\be \label{w2J:0}
\sum_{j=2J-1}^\infty
\sum_{\substack{\tilde{\alpha}_j\in \tilde{\Psi}_j^{2D}, \\ \SSupp(\tilde{\alpha}_j)\cap \Gamma\ne \emptyset}} 2^{2j }  |\la u, \tilde{\alpha}_j\ra|^2
\le
C_v |\log(h)|^2 2^{-2J} (\|u_+\|_{H^2(\Omega_+)}^2 + \|u_-\|_{H^2(\Omega_-)}^2),
\ee
where $C_v:=16 e^2 C_3 C_6 <\infty$.

Since $u\in H^1_0(\Omega)$, we have the following wavelet expansion
\begin{align*}
u
 = \sum_{\alpha\in \Phi^{2D}_{J_0}} \la u, \tilde{\alpha}\ra \alpha + \sum_{j=J_0}^\infty \sum_{\alpha_j\in \Psi^{2D}_j} \la u, \tilde{\alpha}_j\ra \alpha_j
 = \sum_{\alpha\in \Phi^{2D}_{J_0}} \la u, [2^{J_0} \tilde{\alpha}]\ra 2^{-J_0} \alpha + \sum_{j=J_0}^\infty \sum_{\alpha_j\in \Psi^{2D}_j} \la u, [2^j \tilde{\alpha}_j] \ra 2^{-j} \alpha_j.
\end{align*}
We define
\begin{align*}
I_1 & :=\sum_{j=J}^\infty \sum_{\substack{\alpha_j\in \Psi^{2D}_j,\\ \SSupp(\tilde{\alpha}_j)\subseteq \Omega_+}} \la u, [2^j \tilde{\alpha}_j]\ra [2^{-j} \alpha_j],
\qquad
I_2:=\sum_{j=J}^\infty \sum_{\substack{\alpha_j\in \Psi^{2D}_j,\\ \SSupp(\tilde{\alpha}_j)\subseteq \Omega_-}} \la u, [2^j \tilde{\alpha}_j]\ra [2^{-j}\alpha_j],\\
I_3 & :=
\sum_{j=2J-1}^\infty \sum_{\substack{\alpha_j\in \Psi_j,\\ \SSupp(\tilde{\alpha}_j)\cap \Gamma\ne \emptyset}} \la u, [2^j\tilde{\alpha}_j]\ra [2^{-j} \alpha_j], \quad \text{and} \quad
\mathring{u}_h :=
\sum_{\alpha\in \mathcal{B}_{J_0}^{2D}} \la u, \tilde{\alpha}\ra \alpha-I_1-I_2-I_3.
\end{align*}
%
Recall that $V_h^{wav}:=\mbox{span}(\mathcal{B}_{J_0,J}^{S,H^1(\Omega)})$.
Then obviously,
\[
u-\mathring{u}_J=I_1 + I_2 + I_3
\quad \mbox{and}\quad
\mathring{u}_h\in V_h^{wav}.
\]
Because $\mathcal{B}_{J_0}^{S,H^1(\Omega)}$ is a Riesz basis of $H^1_0(\Omega)$, we deduce from $u-\mathring{u}_h=I_1 + I_2 +I_3$ that there must exist a positive constant $C_7$, only depending on the wavelet basis $\mathcal{B}_{J_0}^{S,H^1(\Omega)}$,  such that
\begin{align*}
& \|u-\mathring{u}_h\|^2_{H^1(\Omega)} = \|I_1 + I_2 +I_3 \|^2_{H^1(\Omega)}\\
&\le C_7 \left( \sum_{j=J}^\infty \sum_{\substack{\alpha_j\in \Psi^{2D}_j,\\ \SSupp(\tilde{\alpha}_j)\subseteq \Omega_+}} |\la u, [2^j \tilde{\alpha}_j]\ra|^2 +
\sum_{j=J}^\infty \sum_{\substack{\alpha_j\in \Psi^{2D}_j,\\ \SSupp(\tilde{\alpha}_j)\subseteq \Omega_-}} |\la u, [2^j \tilde{\alpha}_j]\ra|^2 +\sum_{j=2J-1}^\infty
\sum_{\substack{\alpha_j\in \Psi^{2D}_j,\\ \SSupp(\tilde{\alpha}_j)\cap \Gamma\ne \emptyset}} |\la u, [2^j\tilde{\alpha}_j]\ra|^2\right)\\
& = C_7 \left( \sum_{j=J}^\infty \sum_{\substack{\alpha_j\in \Psi^{2D}_j, \\ \SSupp(\tilde{\alpha}_j)\subseteq \Omega_+}} 2^{2j} |\la u, \tilde{\alpha}_j\ra|^2 +
\sum_{j=J}^\infty \sum_{\substack{\alpha_j\in \Psi^{2D}_j,\\ \SSupp(\tilde{\alpha}_j)\subseteq \Omega_-}} 2^{2j} |\la u, \tilde{\alpha}_j\ra|^2 + \sum_{j=2J-1}^\infty
\sum_{\substack{\alpha_j\in \Psi^{2D}_j,\\ \SSupp(\tilde{\alpha}_j)\cap \Gamma\ne \emptyset}} 2^{2j} |\la u, \tilde{\alpha}_j\ra|^2\right).
\end{align*}
By \eqref{v+:u+}, \eqref{v-:u-}, and \eqref{w2J:0},
we have
\[
\begin{aligned}
\|u-\mathring{u}_h\|^2_{H^1(\Omega)}
& \le C_7
[\|u_+\|_{H^2(\Omega_+)}^2+
\|u_-\|_{H^2(\Omega_-)}^2]
(2^{-2J}C_1 + 2^{-2J} C_2 + 2^{-2J} |\log (h)|^2 C_v)\\
&\le C^2_8 \left[\|u_+\|_{H^2(\Omega_+)}^2+
\|u_-\|_{H^2(\Omega_-)}^2 \right]  2^{-2J} |\log (h)|^2 \\
& =
C_8^2 \left[\|u_+\|_{H^2(\Omega_+)}^2+
\|u_-\|_{H^2(\Omega_-)}^2\right] (h |\log (h)|)^2,
\end{aligned}
\]
where $C_8^2 := C_7(C_1 + C_2 + C_v)<\infty$. This proves
\be \label{diff:uuJ}
\|u-\mathring{u}_h\|_{H^1(\Omega)}
\le C_8  h |\log(h)| \sqrt{\|u_+\|_{H^2(\Omega_+)}^2+
\|u_-\|_{H^2(\Omega_-)}^2}.
\ee
By Cea's lemma, there exists a positive constant $C_a$, only depends on the diffusion coefficient $a$ and $\Omega$, such that
\[
\|u-u_h\|_{H^1(\Omega)}\le C_a \inf_{v\in V_h^{wav}}\|
u-v\|_{H^1(\Omega)}.
\]
Because $\mathring{u}_h\in V_h^{wav}$,
we conclude that
\[
\|u-u_h\|_{H^1(\Omega)}\le C_a \inf_{v\in V_h^{wav}}\| u-v\|_{H^1(\Omega)}
\le C_a \|u-\mathring{u}_h\|_{H^1(\Omega)}
\]
and consequently,
\be \label{est:u:H1}
\|u-u_h\|^2_{H^1(\Omega)} \le C_9
h^2 |\log(h)|^2
(\|u_+\|_{H^2(\Omega_+)}^2 + \|u_-\|_{H^2(\Omega_-)}^2),
\ee
where $C_9:=C_8^2 C_a^2$. This proves the first inequality in \eqref{converg:H1} for convergence in $H_0^1(\Omega)$. Because $N_J=\bo(h^{-2})$, the second inequality in \eqref{converg:H1} follows.

\underline{Proving the $\LpO{2}$ convergence.} We now use the Aubin-Nitsche's technique to prove \eqref{converg:L2} for $\LpO{2}$ convergence. Note that the bilinear form $\blf(u,v):=\la a \nabla u, \nabla v\ra$ defined in \eqref{weak} is symmetric.
Suppose that $w\in H_0^1(\Omega)$ satisfies
\be \label{w:sol}
\blf(w,v)=\la u-u_h, v\ra,\qquad v\in H_0^1(\Omega),
\ee
%
%
and its wavelet approximate solution $w_h\in V_h^{wav}$ satisfies
\[
\blf(w_h,v_h)=\la u-u_h,v_h\ra,\qquad v_h\in V_h^{wav}.
\]
%
By the same proof of the inequality of \eqref{est:u:H1}, we have
\[
\|\nabla (w-w_h)\|^2_{\LpO{2}}
\le C_9 h^2 |\log(h)|^2
(\|w_+\|_{H^2(\Omega_+)}^2+\|w_-\|_{H^2(\Omega_-)}^2)
\]
for some positive constant $C_9$.
Because $g=0$ and $g_\Gamma=0$ in the weak formulation \eqref{w:sol},
\cite{LUbook} (also see 
\cite[Theorem~2.1]{CZ98}) guarantees the existence of a positive constant $C_{10}$ such that
\[
\|w_+\|_{H^2(\Omega_+)}^2+\|w_-\|_{H^2(\Omega_-)}^2
\le C_{10} \|u-u_h\|^2_{\LpO{2}},
\]
where $u-u_h$ is treated as the source term for the solution $w$ in \eqref{w:sol}. Therefore,
\be \label{wh}
\|\nabla (w-w_h)\|^2_{\LpO{2}}
\le C_9 h^2 |\log(h)|^2
(\|w_+\|_{H^2(\Omega_+)}^2
+\|w_-\|_{H^2(\Omega_-)}^2)
\le C_9 C_{10} h^2 |\log (h)|^2 \|u-u_h\|^2_{\LpO{2}}.
\ee
Since $v=u-u_h\in H^1_0(\Omega)$,
we deduce from $\blf(w,v)=\la u-u_h, v\ra$  that
\[
\|u-u_h\|^2_{\LpO{2}}
=\blf(w, u-u_h)
=\blf(w-w_h,u-u_h),
\]
where we used the Galerkin orthogonality $\blf(w_h, u-u_h)=\blf(u-u_h, w_h)=0$ for $w_h\in V_h^{wav}$. Consequently, we deduce from \eqref{est:u:H1} and \eqref{wh} that
\begin{align*}
\|u-u_h\|^2_{\LpO{2}}
&=\blf(w-w_h,u-u_h)
\le C_{11} \|\nabla (w-w_h)\|_{\LpO{2}} \|\nabla (u-u_h)\|_{\LpO{2}}\\
&\le h |\log(h)| C_{11} \sqrt{C_9 C_{10}}
 \|u-u_h\|_{\LpO{2}}
\sqrt{C_9} h |\log(h)|
(\|u_+\|_{H^2(\Omega_+)}^2 +\|u_-\|_{H^2(\Omega_-)}^2)^{1/2}\\
&= C h^2 |\log(h)|^2 (\|u_+\|_{H^2(\Omega_+)}^2
+\|u_-\|_{H^2(\Omega_-)}^2)^{1/2}  \|u-u_h\|_{\LpO{2}},
\end{align*}
where $C:=C_{11} C_9 \sqrt{C_{10}}$, from which we conclude that the first inequality of \eqref{converg:L2} holds, i.e.,
\[
\|u-u_h\|_{\LpO{2}}
\le C h^2 |\log(h)|^2 (\|u_+\|_{H^2(\Omega_+)}^2
+\|u_-\|_{H^2(\Omega_-)}^2)^{1/2}.
\]
The second inequality of \eqref{converg:L2} follows trivially by noting $N_J=\bo(h^{-2})$.

\underline{Proving that the condition number is uniformly bounded.} Take $v_h \in V^{wav}_h$. Then, $v_h = \sum_{\eta \in \mathcal{B}^{S,H^1_0(\Omega)}_{J_0,J}} c_{\eta} \eta$. We want to find an upper bound for $\blf(v_h,v_h)$. Note that
\[
\blf(v_h,v_h) \le \|a\|_{L^\infty(\Omega)} \| \nabla v_h \|_{\LpO{2}}^2 \le \|a\|_{L^\infty(\Omega)} \left( \| v_h \|_{\LpO{2}}^2 + \| \nabla v_h \|_{\LpO{2}}^2 \right) \le C_{\mathcal{B},2} \|a\|_{L^\infty(\Omega)} \sum_{\eta \in \mathcal{B}^{S,H_1(\Omega)}_{J_0,J}} |c_\eta|^2,
\]
where we used the fact that $\mathcal{B}^{H_1(\Omega)}_{J_0}$ is a Riesz basis of the Sobolev space $H^{1}_0(\Omega)$ to arrive at the final inequality.  Since $v_h$ satisfies the zero Dirichlet boundary condition, by the Poincar\'e inequality, we have $\| v_h \|_{\LpO{2}} \le C_P \| \nabla v_h \|_{\LpO{2}}$ with $C_P$ being a positive constant that depends only on $\Omega$, which implies that
\[
\|v_h\|_{\LpO{2}}^2 + \|\nabla v_h\|_{\LpO{2}}^2 \le (1+C_P^2) \|\nabla v_h\|_{\LpO{2}}^2.
\]
Moreover, we have
\begin{align*}
	\blf(v_h,v_h) & \ge \|a^{-1}\|_{L^\infty(\Omega)}^{-1} \| \nabla v_h \|_{\LpO{2}}^2 \ge \|a^{-1}\|_{L^\infty(\Omega)}^{-1} (1+C_P^2)^{-1} (\|v_h\|_{\LpO{2}}^2 + \|\nabla v_h\|_{\LpO{2}}^2)\\
	& \ge C_{\mathcal{B},1} \|a^{-1}\|_{L^\infty(\Omega)}^{-1} (1+C_P^2)^{-1} \sum_{\eta \in \mathcal{B}^{S,H^1_0(\Omega)}_{J_0,J}} |c_\eta|^2,
\end{align*}
where we used the fact that $\mathcal{B}^{H^1_0(\Omega)}_{J_0}$ is a Riesz basis of the Sobolev space $H^{1}_0(\Omega)$ to arrive at the final inequality. Combining the lower and upper bounds of $\blf(v_h,v_h)$, we have
\[
C_{\mathcal{B},1} \|a^{-1}\|_{L^\infty(\Omega)}^{-1} (1+C_P^2)^{-1} \sum_{\eta \in \mathcal{B}^{S,H^1_0(\Omega)}_{J_0,J}} |c_\eta|^2
\le
\blf(v_h,v_h)
\le
C_{\mathcal{B},2} \|a\|_{L^\infty(\Omega)} \sum_{\eta \in \mathcal{B}^{S,H^1_0(\Omega)}_{J_0,J}} |c_\eta|^2,
\]
which gives an upper bound of the condition number in the form of $C_w \|a\|_{L^\infty(\Omega)} \|a^{-1}\|_{L^\infty(\Omega)}$, where $C_w := (1+C_P^2) C_{\mathcal{B},2} C_{\mathcal{B},1}^{-1}<\infty$.
\end{proof}

\underline{Justification for \cref{remark:quad}.} Consider the setting of \cref{thm:converg}. Let $u$ be the true solution of \eqref{model}. Also, let $u_h =u_J:=\sum_{\eta \in \mathcal{B}^{S,H^1_0(\Omega)}_{J_0,J}} c_{\eta} \eta$ be the solution obtained from solving \eqref{auJu} with exact integral computations. That is, the coefficient vector $\vec{c}_{J}:=\{c_{\eta}\}_{\eta \in \mathcal{B}^{S,H^1_0(\Omega)}_{J_0,J}}$ is obtained from solving $A_J \vec{c}_{J} = \vec{f}_J$, where the matrix $A_J$ and the vector $\vec{f}_J$ correspond to the left-hand and right-hand sides of \eqref{auJu}. Similarly, define  $u_h^{Q} =u_J^{Q} :=\sum_{\eta \in \mathcal{B}^{S,H^1_0(\Omega)}_{J_0,J}} c_{\eta}^{Q} \eta$, where the coefficient vector $\vec{c}^{Q}_{J}:=\{c^{Q}_{\eta}\}_{\eta \in \mathcal{B}^{S,H^1_0(\Omega)}_{J_0,J}}$ is obtained from solving $(A_J + E_J)\vec{c}^{Q}_J = \vec{f}_{J}+\vec{r}_J$. Here, $E_J$ and $\vec{r}_J$ are the error matrix and error vector coming from the interface curve approximation and the numerical quadrature used to compute $A_J$ and 
$\vec{f}_J$ respectively. In what follows, let $\|\cdot\|_2$ denote the Euclidean norm for vectors and the induced matrix 2-norm for matrices. Assume that the error matrix satisfies $\|E_J\|_2 < \theta \alpha C_{\mathcal{B},1}$, where $0< \theta <1$, $\alpha$ is the coercivity constant of the bilinear form, and $C_{\mathcal{B},1}$ is the lower bound Riesz basis constant of $\mathcal{B}^{S,H^1_0(\Omega)}_{J_0}$ in \eqref{stability}. Subtracting these two linear systems, we have 
\[
(A_J + E_J)(\vec{c}^{Q}_J - \vec{c}_J) = \vec{r}_J - E_J \vec{c}_J.
\]
By the Riesz basis property, we have $\|u_J^{Q} - u_J\|_{H^{1}(\Omega)} \le C_{\mathcal{B},2}^{1/2} \|\vec{c}_J^{Q} - \vec{c}_J\|_2$, where $C_{\mathcal{B},2}$ is the upper bound Riesz basis constant of $\mathcal{B}^{S,H^1_0(\Omega)}_{J_0}$ in \eqref{stability}. Furthermore, we have
\[
\|u_J^{Q} - u_J\|_{H^{1}(\Omega)} \le C_{\mathcal{B},2}^{1/2}\| (A_J + E_J)^{-1} \|_2 (\|\vec{r}_J\|_2 + \|E_J\|_2 \|\vec{c}_J\|_2).
\]
Let us find an upper bound for $\|A_J^{-1}\|_2$. Let $v_J:=\sum_{\eta \in \mathcal{B}^{S,H^1_0(\Omega)}_{J_0,J}} z_{\eta} \eta$, where $z_{\eta} \in \R$ for each $\eta$. Define the coefficient vector $\vec{z} := \{z_{\eta}\}_{\eta \in \mathcal{B}^{S,H^1_0(\Omega)}_{J_0,J}}$. Then, 
\[
\vec{z}^\tp A_J \vec{z} = B(v_J,v_J) \ge \alpha \|v_J\|^{2}_{H^{1}(\Omega)} \ge \alpha C_{\mathcal{B},1} \|\vec{z}\|_{2}^2,
\]
where we have used the coercivity constant and the lower bound Riesz basis constant of $\mathcal{B}^{S,H^1_0(\Omega)}_{J_0}$ in \eqref{stability}. Define $M:=\#\mathcal{B}^{S,H^1_0(\Omega)}_{J_0,J}$. Since $A_J$ is a symmetric positive definite matrix, we have 
\[
\alpha C_{\mathcal{B},1} \le \min_{\vec{z} \in \R^{M} \backslash\{\vec{0}\}} \frac{\vec{z}^\tp A_j \vec{z}}{\|\vec{z}\|_{2}^2} = \lambda_{\min}(A_J) = \|A_J^{-1}\|_{2}^{-1} 
\implies
\|A_J^{-1}\|_{2} \le (\alpha C_{\mathcal{B},1})^{-1}. 
\]
We are now ready to analyze the $H^{1}(\Omega)$-norm of the difference between $u$ and $u^{Q}_J$. By calculation,
\begin{align*}
	\|u - u_J^{Q}\|_{H^{1}(\Omega)} & \le \|u - u_J\|_{H^{1}(\Omega)} + \|u_J - u_J^{Q}\|_{H^{1}(\Omega)}\\
	& \le \|u - u_J\|_{H^{1}(\Omega)} + C_{\mathcal{B},2}^{1/2} \|A_J^{-1}\|_2 \| (I + A_J^{-1} E_J)^{-1} \|_2 (\|\vec{r}_J\|_2 + \|E_J\|_2 \|\vec{c}_J\|_2) \\
	& \le \|u - u_J\|_{H^{1}(\Omega)} + C_{\mathcal{B},2}^{1/2} \frac{\|A_J^{-1}\|_2}{1 - \|A_J^{-1}\|_2 \|E_J\|_2} (\|\vec{r}_J\|_2 + \|E_J\|_2 \|\vec{c}_J\|_2)\\
	& \le C h |\log(h)| + \frac{C_{\mathcal{B},2}^{1/2}}{\alpha C_{\mathcal{B},1} (1-\theta) } (\|\vec{r}_J\|_2 + \|E_J\|_2 \|\vec{c}_J\|_2),
\end{align*}
where we have used the condition $\|A_J^{-1}\|_2 \|E_J\|_2 <1$, the Neumann-series bound, and the $H^{1}(\Omega)$-error estimate in \eqref{converg:H1}. 
Therefore, if the numerical quadrature and interface curve approximation are accurate enough so that
\be \label{quad:err}
\|\vec{r}_J\|_2 + \|E_J\|_2 \|\vec{c}_J\|_2 \le C_E h |\log(h)|
\ee
for some $C_E >0$ that is independent of the scale level $J$, it follows that 
\[
\|u - u_h^{Q}\|_{H^{1}(\Omega)} = \|u - u_J^{Q}\|_{H^{1}(\Omega)} \le \tilde{C}  h |\log(h)|, 
\quad \tilde{C}:= C + \frac{C_{\mathcal{B},2}^{1/2} C_E}{\alpha C_{\mathcal{B},1} (1-\theta) }, 
\]
where $\tilde{C}$ is also independent of the scale level $J$. A similar result holds for the $L^2(\Omega)$ error. 

As discussed in the paragraph following \cref{remark:quad}, the assumption in \eqref{quad:err}  can be satisfied by combining the refinability of the wavelet basis, with a high-order approximation of the interface curve, and a sufficiently high-degree Gaussian quadrature rule. To control $\|E_J\|_2$, we 
recall that $\|E_J\|_2 \le \sqrt{\|E_J\|_1 \|E_J\|_{\infty}}$, where 
\[
\|E_J\|_{\infty} := \max_{i} \sum_{k} |(E_J)_{i,k}|, \quad \|E_J\|_{1} := \max_{k} \sum_{i} |(E_J)_{i,k}|,
\]  
and $(E_J)_{i,k}$ is the $(i,k)$-th entry of the matrix $E_J$. Therefore, it is enough to control the errors of the individual Galerkin integrals.

\end{document}